\documentclass[10pt, a4paper,reqno]{amsart}
\usepackage{amsmath,amssymb,amsthm}
\usepackage{xcolor,textcomp}
\usepackage{graphicx, epstopdf}
\usepackage{braket,amsfonts}
\usepackage{dsfont}
\usepackage{mathtools}
\usepackage{mwe}
\usepackage{fullpage}
\usepackage{hyperref}
\hypersetup{colorlinks=true,linkcolor=red,citecolor=red,filecolor=magenta,urlcolor=blue}
\usepackage[capitalise]{cleveref}
\usepackage{tikz}
\usepackage{pgfplots}
\usepackage{stmaryrd}
\usepackage[mathscr]{euscript}
\usepackage{enumerate}
\addtocontents{toc}{\protect\setcounter{tocdepth}{1}}
\usepackage[sort,nocompress]{cite}
\numberwithin{equation}{section}
\usepackage{comment}
\renewcommand{\mod}[1]{\left\lvert#1\right\rvert}

\newcommand{\norm}[1]{\left\|#1\right\|}
\newcommand{\mb}[1]{\mathbb{#1}}
\newcommand{\mc}[1]{\mathcal{#1}}
\newcommand{\ip}[2]{\left\langle#1,#2\right\rangle}
\newcommand{\vphi}{\varphi} 
\renewcommand{\vector}[2]{\begin{pmatrix}#1\\[4pt]#2\end{pmatrix}}
\newcommand{\cv}[3]{\begin{pmatrix}#1\\[4pt]#2\\[4pt]#3\end{pmatrix}}

\renewcommand{\Re}{\textnormal{Re}}
\renewcommand{\Im}{\textnormal{Im}}
\newcommand{\blue}{\color{black}}

\DeclareMathOperator{\per}{per} 

\pgfplotsset{compat=1.18}
\newtheorem{theorem}{Theorem}[section]
\newtheorem{lemma}[theorem]{Lemma}
\newtheorem{remark}[theorem]{Remark}

\newtheorem{definition}[theorem]{Definition}
\newtheorem{proposition}[theorem]{Proposition}

\title[Linearized compressible Navier-Stokes system (barotropic and non-barotropic)]{Null controllability of one-dimensional {\blue barotropic and non-barotropic} linearized compressible Navier-Stokes system using one boundary control}
\author[J. Kumbhakar]{Jiten Kumbhakar$^{\dagger}$}
\thanks{$^\dagger$Indian Institute of Science Education and Research Kolkata, Campus road, Mohanpur, West Bengal 741246, India; jk17ip021@iiserkol.ac.in.}
\begin{document}
%%%%Abstract%%%%
\begin{abstract}
In this article, we study boundary null controllability properties of the linearized compressible Navier-Stokes equations {\blue in the interval $(0,2\pi)$} for both barotropic and non-barotropic fluids using only one boundary control. {\blue We consider all the possible cases of the act of control for both systems (density, velocity and temperature). These controls are acting on the boundary and are given as the difference of the values at $x=0$ and $x=2\pi$. In this setup, using a boundary control acting in density, we first prove null controllability of both the barotropic and non-barotropic systems at large time in the spaces $(\dot{L}^2(0,2\pi))^2$ and $(\dot{L}^2(0,2\pi))^3$ respectively (where the dot represents functions with mean value zero). When the control is acting in the velocity component, we prove null controllability at large time in the spaces $\dot{H}^1_{\per}(0,2\pi)\times\dot{L}^2(0,2\pi)$ and $\dot{H}^1_{\per}(0,2\pi)\times(\dot{L}^2(0,2\pi))^2$ respectively. Further, in both cases, we prove that these null controllability results are} sharp with respect to the regularity of the initial states {\blue in velocity/ temperature case, and time in the density case. Finally, for both barotropic and non-barotropic fluids, we prove that, under some assumptions, the system cannot be approximately controllable at any time, whether there is a control acting in density, velocity or temperature.}
\end{abstract}
%%%%%%%%%%%%%%%%%%%%%%%%%%%%%%%%%%%%%%%%%%%%%%%%%%%%%%%%%%%%%%%%%%%%%%%%%
\maketitle
\tableofcontents
%%%%%%%%%%%%%%%%%%%%%%%%%%%%%%%%%%%%%%%%%%%%%%%%%%%%%%%%%%%%%%%%%%%%%%%%%
\section{Introduction and main results}
\subsection{Linearized compressible Navier-Stokes system in 1d}
Let $I=(0,+\infty)$ be the time interval and $\Omega\subset\mb{R}$ be a spatial domain. For a compressible, isentropic (barotropic) fluid, that is, when the pressure depends only on the density and the temperature is constant, the Navier-Stokes system in $I\times\Omega$ consists of the equation of continuity and the momentum equation
\begin{gather*}
\rho_t(t,x)+(\rho u)_x(t,x)=0,\\
\rho(t,x)[u_t(t,x)+u(t,x)u_x(t,x)]+{\blue p^b_x(t,x)}-(\lambda+2\mu)u_{xx}(t,x)=0,
\end{gather*}
where $\rho$ denotes the density of the fluid, $u$ is the velocity. The constants $\lambda,\mu$ are called the viscosity coefficients that satisfy $\mu>0,\lambda+\mu\geq0$. The pressure $\blue p^b$ satisfies the following constitutive equation in $I\times\Omega$ $${\blue p^b(t,x)}=a\rho^{\gamma}(t,x),\ \ (t,x)\in I\times\Omega,$$ for some constants $a>0$, $\gamma\geq1$. In the case of non-barotropic fluids, that is, when the pressure is a function of both density and temperature of the fluid, the Navier-Stokes system consists of the equation of continuity, the momentum equation, and an additional thermal energy equation
\begin{align*}
c_{\nu}\rho(t,x)&[\theta_t(t,x)+u(t,x)\theta_x(t,x)]+\theta(t,x){\blue p^{nb}_{\theta}(t,x)}u_x(t,x)-\kappa\theta_{xx}(t,x)-(\lambda+2\mu)u_x^2(t,x)=0,
\end{align*}
where $\theta$ is the temperature of the fluid, $c_{\nu}$ is the specific heat constant, and $\kappa$ is the heat conductivity constant. For an ideal gas, Boyles law gives the pressure ${\blue p^{nb}(t,x)}=R\rho(t,x)\theta(t,x)$ in $I\times\Omega$ with $R$ as the universal gas constant. See \cite[Chapter 1]{Feireisl04} for more about compressible flows.

\subsection{The barotropic case}
Let $T>0$ be a finite time. We first consider the Navier-Stokes system for compressible, isentropic (barotropic) fluids linearized around some constant steady state $(\bar{\rho},\bar{u})$ with $\bar{\rho}>0$ and $\bar{u}>0$
\begin{equation}\label{lcnse_b}
\begin{dcases}
\rho_t(t,x)+\bar{u}\rho_x(t,x)+\bar{\rho}u_x(t,x)=0,\ \ &\text{in}\ (0,T)\times(0,2\pi),\\
u_t(t,x)-\frac{\blue\lambda+2\mu}{\bar{\rho}}u_{xx}(t,x)+\bar{u}u_x(t,x)+a\gamma\bar{\rho}^{\gamma-2}\rho_x(t,x)=0,\ \ &\text{in}\ (0,T)\times(0,2\pi).
\end{dcases}
\end{equation}
The initial conditions are
\begin{equation}\label{in_cd_b}
\rho(0,x)=\rho_0(x),\ \ u(0,x)=u_0(x),\ \ x\in(0,2\pi).
\end{equation}
We will consider two different problems, based on the act of control, by imposing any one of the following boundary conditions on the system \eqref{lcnse_b}.
\begin{itemize}
\item \textbf{Control in density:}
\begin{equation}\label{bd_cd_b1}
\rho(t,0)=\rho(t,2\pi)+p(t),\ \ u(t,0)=u(t,2\pi),\ \ u_x(t,0)=u_x(t,2\pi),\ \ t\in (0,T).
\end{equation}
\item \textbf{Control in velocity:}
\begin{equation}\label{bd_cd_b2}
\rho(t,0)=\rho(t,2\pi),\ \ u(t,0)=u(t,2\pi)+q(t),\ \ u_x(t,0)=u_x(t,2\pi),\ \ t\in (0,T),
\end{equation}
\end{itemize}
where $p$ and $q$ are controls {\blue acting on the boundary and are given as the difference of the values at $x=0$ and $x=2\pi$}. %Our main goal is to study null controllability of the system $\eqref{lcnse_b}$ at a given time $T>0$ with the initial condition \eqref{in_cd_b} and one of the boundary conditions \eqref{bd_cd_b1} and \eqref{bd_cd_b2}. More precisely, given any initial state $(\rho_0,u_0)$ in some suitable Hilbert space, we want to find a boundary control $p$ (resp. $q$) such that the solution $(\rho,u)$ to the system \eqref{lcnse_b}-\eqref{in_cd_b}-\eqref{bd_cd_b1} (resp. \eqref{lcnse_b}-\eqref{in_cd_b}-\eqref{bd_cd_b2}) satisfies 
%\begin{equation*}
%(\rho(T,x),u(T,x))=(0,0)\ \ \text{in}\ (0,2\pi).
%\end{equation*}
{\blue
\begin{definition}
Let $H$ be a Hilbert space. We say the system \eqref{lcnse_b}-\eqref{in_cd_b}-\eqref{bd_cd_b1} $($resp. \eqref{lcnse_b}-\eqref{in_cd_b}-\eqref{bd_cd_b2}$)$ is 
\begin{itemize}
\item \textit{null controllable} at time $T$ in the space $H$ if, for any $(\rho_0,u_0)\in H$ , there exists a control $p\in L^2(0,T)$ $($resp. $q\in L^2(0,T))$ such that the associated solution satisfies
\begin{equation*}
(\rho(T),u(T))=(0,0).%\ \ \text{in}\ H.
\end{equation*}
\item \textit{approximately controllable} at time $T$ in the space $H$ if, for any $(\rho_0,u_0),(\rho_T,u_T)\in H$ and any $\epsilon>0$, there exists a control $p\in L^2(0,T)$ $($resp. $q\in L^2(0,T))$ such that the associated solution satisfies
\begin{equation*}
\norm{(\rho(T),u(T))-(\rho_T,u_T)}_H\leq\epsilon.
\end{equation*}
\end{itemize}
\end{definition}
%\begin{definition}
%Let $H$ be a Hilbert space. We say the system \eqref{lcnse_b}-\eqref{in_cd_b}-\eqref{bd_cd_b1} $($resp. \eqref{lcnse_b}-\eqref{in_cd_b}-\eqref{bd_cd_b2}$)$ is 
%\end{definition}
}
\noindent Our main goal {\blue in this article} is to study null controllability of the system $\eqref{lcnse_b}$ at a given time $T>0$ with the initial condition \eqref{in_cd_b} and one of the boundary conditions \eqref{bd_cd_b1} and \eqref{bd_cd_b2}.

\smallskip

Before stating our main results, we first introduce the Sobolev space for any $s>0$
\begin{equation*}
H^s_{\per}(0,2\pi)=\left\{\vphi \ : \ \vphi=\sum_{n\in\mb{Z}}c_ne^{inx},\ \sum_{n\in\mb{Z}}\mod{n}^{2s}\mod{c_n}^2<\infty\right\},
\end{equation*}
with the norm
\begin{equation*}
\norm{\vphi}_{H^s_{\per}(0,2\pi)}:=\left(\sum_{n\in\mb{Z}}(1+\mod{n}^{2})^s\mod{c_n}^2\right)^{\frac{1}{2}}.
\end{equation*}
For $s>0$, we denote $H^{-s}_{\per}(0,2\pi)$ to be the dual of the Sobolev space $H^s_{\per}(0,2\pi)$ with respect to the pivot space $L^2(0,2\pi)$. We also define the space
\begin{equation*}
\dot{L}^2(0,2\pi):=\left\{\vphi\in L^2(0,2\pi)\ : \ \int_{0}^{2\pi}\vphi(x)dx=0\right\}
\end{equation*}
and
\begin{equation*}
\dot{H}^s_{\per}(0,2\pi):=\left\{\vphi\in H^s_{\per}(0,2\pi)\ : \ \int_{0}^{2\pi}\vphi(x)dx=0\right\}.
\end{equation*}
If the system \eqref{lcnse_b}-\eqref{in_cd_b}-\eqref{bd_cd_b1} is null controllable at time $T$ by using a boundary control $p$, then integrating both equations in \eqref{lcnse_b}, we get a compatibility condition on the initial states
\begin{equation*}
a\gamma\bar{\rho}^{\gamma-2}\int_{0}^{2\pi}\rho_0(x)dx=\bar{u}\int_{0}^{2\pi}u_0(x)dx=-a\gamma\bar{\rho}^{\gamma-2}\bar{u}\int_{0}^{T}p(t)dt.
\end{equation*}
If the system \eqref{lcnse_b}-\eqref{in_cd_b}-\eqref{bd_cd_b2} is null controllable at time $T$ by using a boundary control $q$, then also we will get a similar compatibility condition on the initial states. Since every initial state $(\rho_0,u_0)$ in $(L^2(0,2\pi))^2$ will not satisfy this compatibility condition, we will work on the Hilbert space $(\dot{L}^2(0,2\pi))^2$ to avoid this difficulty.

\smallskip

\noindent When a boundary control $q$ {\blue is acting} in the velocity component, it is known in \cite{Chowdhury15b} that the system \eqref{lcnse_b}-\eqref{in_cd_b}-\eqref{bd_cd_b2} is null controllable at time $T>\frac{2\pi}{\bar{u}}$ provided that the initial state is regular enough, in particular, lies in the space $\dot{H}^{s+1}_{\per}(0,2\pi)\times\dot{H}^s_{\per}(0,2\pi)$ for $s>\frac{9}{2}$. In the first part of our article, we generalize this result (with respect to the regularity of initial states). {\blue In fact, we prove null controllability of \eqref{lcnse_b}-\eqref{in_cd_b}-\eqref{bd_cd_b2} at time $T>\frac{2\pi}{\bar{u}}$ in the optimal space $\dot{H}^1_{\per}(0,2\pi)\times \dot{L}^2(0,2\pi)$ (see Theorem \ref{thm_vel_b}). In addition,} we also prove null controllability of the system \eqref{lcnse_b} {\blue at time $T>\frac{2\pi}{\bar{u}}$ in $(\dot{L}^2(0,2\pi))^2$} when there is a boundary control $p$ {\blue acting} in the density component {\blue and that the null controllability fails when the time is small, in particular, when $0<T<\frac{2\pi}{\bar{u}}$ (see Theorem \ref{thm_den_b}). These results requires certain restrictions on the coefficients appearing in the system \eqref{lcnse_b}; otherwise the system is not even approximately controllable (Proposition \ref{prop_lac_app_b}). To be more precise, if the coefficients $\bar{\rho},\bar{u},\mu_0,b$ (defined by \eqref{diff_coeff_b}) satisfy $\frac{2\sqrt{b\bar{\rho}-\bar{u}^2}}{\mu_0}\in\mb{N}$, then the associated adjoint operator $A^*$ (defined by \eqref{op_A*_b}) admits an eigenvalue with algebraic multiplicity = geometric multiplicity = $2$, failing the unique continuation property (see the proof of Proposition \ref{prop_lac_app_b} for details). However, if $\frac{2\sqrt{b\bar{\rho}-\bar{u}^2}}{\mu_0}\notin\mb{N}$, then all the eigenvalues of $A^*$ have geometric multiplicity $1$ and in this case, we can achieve null controllability of the system \eqref{lcnse_b} by using one boundary control acting either in density or in velocity.

}

{\blue

\bigskip

\noindent The first main results concerning the null controllability of the system \eqref{lcnse_b} are stated below.
\begin{theorem}[Control in density]\label{thm_den_b}
Let us assume that $\frac{2\sqrt{b\bar{\rho}-\bar{u}^2}}{\mu_0}\notin\mb{N}$. Then,
\begin{enumerate}[(i)]
\item the system \eqref{lcnse_b}-\eqref{in_cd_b}-\eqref{bd_cd_b1} is null controllable at any time $T>\frac{2\pi}{\bar{u}}$ in the space $(\dot{L}^2(0,2\pi))^2$.\label{null_dens_b}
\item the system \eqref{lcnse_b}-\eqref{in_cd_b}-\eqref{bd_cd_b1} is not null controllable at small time $0<T<\frac{2\pi}{\bar{u}}$ in $(\dot{L}^2(0,2\pi))^2$.\label{lac_dens_b}
\end{enumerate}
\end{theorem}

\begin{theorem}[Control in velocity]\label{thm_vel_b}
Let us assume that $\frac{2\sqrt{b\bar{\rho}-\bar{u}^2}}{\mu_0}\notin\mb{N}$. Then,
\begin{enumerate}[(i)]
\item the system \eqref{lcnse_b}-\eqref{in_cd_b}-\eqref{bd_cd_b2} is null controllable at any time $T>\frac{2\pi}{\bar{u}}$ in $\dot{H}^1_{\per}(0,2\pi)\times\dot{L}^2(0,2\pi)$.\label{null_vel_b}
\item the system \eqref{lcnse_b}-\eqref{in_cd_b}-\eqref{bd_cd_b2} is not null controllable at any time $T>0$ in the space $\dot{H}^s_{\per}(0,2\pi)\times \dot{L}^2(0,2\pi)$ with $0\leq s<1$.\label{lac_vel_b}
\end{enumerate}
\end{theorem}
}

\begin{remark}\label{rem_lac_b}
\noindent Following the proof of Theorem \ref{thm_den_b} - Part \eqref{lac_dens_b}, lack of null controllability of the system \eqref{lcnse_b}-\eqref{in_cd_b}-\eqref{bd_cd_b2} cannot be obtained when the time is small, in particular, when $0<T<\frac{2\pi}{\bar{u}}$. {\blue However,} the lack of controllability at small time may be possible to obtain by constructing a Gaussian beam, as mentioned in \cite[Theorem 1.5]{Debayan15} for the interior control case. Further, null controllability of the system \eqref{lcnse_b} at time $T=\frac{2\pi}{\bar{u}}$ is inconclusive in both cases, whether there is a control acting in density or velocity.
\end{remark}
{\blue
The following result shows that the restriction on the coefficients $\Big(\frac{2\sqrt{b\bar{\rho}-\bar{u}^2}}{\mu_0}\notin\mb{N}\Big)$ is required to achieve null controllability of the system \eqref{lcnse_b}.
\begin{proposition}\label{prop_lac_app_b}
If $\frac{2\sqrt{b\bar{\rho}-\bar{u}^2}}{\mu_0}\in\mb{N}$, the system \eqref{lcnse_b}-\eqref{in_cd_b}-\eqref{bd_cd_b1} $($resp. \eqref{lcnse_b}-\eqref{in_cd_b}-\eqref{bd_cd_b2}$)$ is not approximately controllable at any time $T>0$ in the space $(L^2(0,2\pi))^2$.
\end{proposition}
\noindent We note here that, due to the backward uniqueness property of the system \eqref{lcnse_b}, null controllability at time $T$ will give us the approximate controllability at that time $T$ for both the systems \eqref{lcnse_b}-\eqref{in_cd_b}-\eqref{bd_cd_b1} and \eqref{lcnse_b}-\eqref{in_cd_b}-\eqref{bd_cd_b2}, see Section \ref{back_uniq} for more details.
}

\subsection{The non-barotropic case}
We next consider the Navier-Stokes system for compressible non-barotropic fluids linearized around some constant steady state $(\bar{\rho},\bar{u},\bar{\theta})$ with $\bar{\rho},\bar{u},\bar{\theta}>0$
\begin{align}\label{lcnse_nb}
\begin{dcases}
\rho_t(t,x)+\bar{u}\rho_x(t,x)+\bar{\rho}u_x(t,x)=0,\ &\text{in}\ (0,T)\times(0,2\pi),\\
u_t(t,x)-\frac{\lambda+2\mu}{\bar{\rho}}u_{xx}(t,x)+\frac{R\bar{\theta}}{\bar{\rho}}\rho_x(t,x)+\bar{u}u_x(t,x)+R\theta_x(t,x)=0, \ &\text{in}\ (0,T)\times(0,2\pi),\\
\theta_t(t,x)-\frac{\kappa}{\bar{\rho}c_{\nu}}\theta_{xx}(t,x)+\frac{R\bar{\theta}}{c_{\nu}}u_x(t,x)+\bar{u}\theta_x(t,x)=0, \ &\text{in}\ (0,T)\times(0,2\pi).
\end{dcases}
\end{align}
The initial conditions are
\begin{equation}\label{in_cd_nb}
\rho(0,x)=\rho_0(x),\ \ u(0,x)=u_0(x),\ \ \theta(0,x)=\theta_0(x), \ \ x\in(0,2\pi).
\end{equation}
In this case, we will consider three different problems, based on the act of control, by imposing any one of the following boundary conditions on the system \eqref{lcnse_nb}.
\begin{itemize}
\item \textbf{Control in density:}
\begin{equation}\label{bd_cd_nb1}
\rho(t,0)=\rho(t,2\pi)+p(t),\ u(t,0)=u(t,2\pi), \ u_x(t,0)=u_x(t,2\pi),\ \theta(t,0)=\theta(t,2\pi),\ \theta_x(t,0)=\theta_x(t,2\pi).
\end{equation}
\item \textbf{Control in velocity:}
\begin{equation}\label{bd_cd_nb2}
\rho(t,0)=\rho(t,2\pi),\ u(t,0)=u(t,2\pi)+q(t), \ u_x(t,0)=u_x(t,2\pi),\ \theta(t,0)=\theta(t,2\pi),\ \theta_x(t,0)=\theta_x(t,2\pi).
\end{equation}
\item \textbf{Control in temperature:}
\begin{equation}\label{bd_cd_nb3}
\rho(t,0)=\rho(t,2\pi),\ u(t,0)=u(t,2\pi), \ u_x(t,0)=u_x(t,2\pi),\ \theta(t,0)=\theta(t,2\pi)+r(t),\ \theta_x(t,0)=\theta_x(t,2\pi).
\end{equation}
\end{itemize}
for $t\in (0,T)$, where $p,q$ and $r$ are controls {\blue acting on the boundary and are given as the difference of the values at $x=0$ and $x=2\pi$}.

In this case also, we want to prove null controllability of the system \eqref{lcnse_nb} at a given time $T>0$ depending on the act of the control. Similar to the barotropic case, we will work on the Hilbert space $(\dot{L}^2(0,2\pi))^3$ {\blue to avoid the compatibility condition of the initial states.}
{\blue
\begin{definition}
Let $H$ be a Hilbert space. We say the system \eqref{lcnse_nb}-\eqref{in_cd_nb}-\eqref{bd_cd_nb1} $($resp. \eqref{lcnse_nb}-\eqref{in_cd_nb}-\eqref{bd_cd_nb2}, \eqref{lcnse_nb}-\eqref{in_cd_nb}-\eqref{bd_cd_nb3}$)$ is 
\begin{itemize}
\item \textit{null controllable} at time $T$ in the space $H$ if, for any $(\rho_0,u_0,\theta_0)\in H$ , there exists a control $p\in L^2(0,T)$ $($resp. $q,r\in L^2(0,T))$ such that the associated solution satisfies
\begin{equation*}
(\rho(T),u(T),\theta(T))=(0,0,0).%\ \ \text{in}\ H.
\end{equation*}
\item \textit{approximately controllable} at time $T$ in the space $H$ if, for any given $(\rho_0,u_0,\theta_0)$,\\$(\rho_T,u_T,\theta_T)\in H$ and any $\epsilon>0$, there exists a control $p\in L^2(0,T)$ $($resp. $q,r\in L^2(0,T))$ such that the associated solution satisfies
\begin{equation*}
\norm{(\rho(T),u(T),\theta(T))-(\rho_T,u_T,\theta_T)}_H\leq\epsilon.
\end{equation*}
\end{itemize}
\end{definition}
%\begin{definition}
%Let $H$ be a Hilbert space. We say the system \eqref{lcnse_nb}-\eqref{in_cd_nb}-\eqref{bd_cd_nb1} $($resp. \eqref{lcnse_nb}-\eqref{in_cd_nb}-\eqref{bd_cd_nb2}, \eqref{lcnse_nb}-\eqref{in_cd_nb}-\eqref{bd_cd_nb3}$)$ is 
%\end{definition}
We next study mainly the null controllability of the system \eqref{lcnse_nb} at a given time $T>0$ starting from the initial condition \eqref{in_cd_nb} and with one of the boundary conditions \eqref{bd_cd_nb1}-\eqref{bd_cd_nb2} and \eqref{bd_cd_nb3}. Since the additional thermal energy equation satisfied by $\theta$ do not have any coupling with the density $\rho$, we can expect similar controllability results like the barotropic case. However, in this case, we have two parabolic equations with diffusion coefficients $\frac{\lambda+2\mu}{\bar{\rho}}$ and $\frac{\kappa}{\bar{\rho}c_{\nu}}$ and therefore by looking at \cite{Cara10,Luca13,Ammar14}, one question arises naturally:

``\textit{Under what conditions on these diffusion coefficients, the system \eqref{lcnse_nb} is null controllable?}"

\noindent In fact, we will prove that there exist diffusion coefficients for which the system \eqref{lcnse_nb} may not even be approximately controllable at any time $T>0$ in $(L^2(0,2\pi))^2$. However, under some stronger assumptions on the diffusion coefficients, we can prove null controllability of \eqref{lcnse_nb} at any given time $T>\frac{2\pi}{\bar{u}}$ in appropriate spaces (see Theorem \ref{thm_null_nb}).

Before going any further, we first denote the (positive) diffusion coefficients for the non-barotropic system}
\begin{equation}\label{diff_coeff_nb}
\lambda_0:=\frac{\lambda+2\mu}{\bar{\rho}},\ \ \kappa_0:=\frac{\kappa}{\bar{\rho}c_{\nu}},
\end{equation}
and define the set
\begin{equation}
\mc{S}:=\left\{(\lambda_0,\kappa_0)\ : \ \sqrt{\frac{\lambda_0}{\kappa_0}}\notin\mb{Q}\right\}.
\end{equation}
{\blue The reason behind introducing such a set $\mc{S}$ is explained at the end of this section. First we will state our next main results which concerns null and approximate controllability of the system \eqref{lcnse_nb}.

\begin{theorem}\label{thm_null_nb}
Let us assume that $(\lambda_0,\kappa_0)\in\mc{S}$ be such that there exists a $M>0$ with the property that 
\begin{equation}\label{cond_coeff_nb}
\mod{\sqrt{\frac{\lambda_0}{\kappa_0}}-\frac{a}{b}}>\frac{1}{b^M}
\end{equation}
holds for all rational numbers $\frac{a}{b}$. We further assume that all the eigenvalues of $A^*$ (defined by \eqref{op_A^*_nb}) have geometric multiplicity equal to $1$. Then,
\begin{enumerate}[(i)]
\item the system \eqref{lcnse_nb}-\eqref{in_cd_nb}-\eqref{bd_cd_nb1} is null controllable at any time $T>\frac{2\pi}{\bar{u}}$ in the space $(\dot{L}^2(0,2\pi))^3$.\label{null_dens_nb}
\item the systems \eqref{lcnse_nb}-\eqref{in_cd_nb}-\eqref{bd_cd_nb2} and \eqref{lcnse_nb}-\eqref{in_cd_nb}-\eqref{bd_cd_nb3} are null controllable at any time $T>\frac{2\pi}{\bar{u}}$ in the space $\dot{H}^1_{\per}(0,2\pi)\times(\dot{L}^2(0,2\pi))^2$.\label{null_vel_temp_nb}
%\item the system \eqref{lcnse_nb}-\eqref{in_cd_nb}-\eqref{bd_cd_nb1} is null controllable at any time $T>\frac{2\pi}{\bar{u}}$ in the space $(\dot{L}^2(0,2\pi))^3$.\label{null_dens_nb}
\end{enumerate}
\end{theorem}

\begin{proposition}\label{prop_lac_null_nb}
The following statements hold:
\begin{enumerate}[(i)]
\item The system \eqref{lcnse_nb}-\eqref{in_cd_nb}-\eqref{bd_cd_nb1} is not null controllable at small time $0<T<\frac{2\pi}{\bar{u}}$ in the space $(\dot{L}^2(0,2\pi))^3$.\label{lac_dens_nb}
\item The systems \eqref{lcnse_nb}-\eqref{in_cd_nb}-\eqref{bd_cd_nb2} and \eqref{lcnse_nb}-\eqref{in_cd_nb}-\eqref{bd_cd_nb3} are not null controllable at any time $T>0$ in the space $\dot{H}^s_{\per}(0,2\pi)\times(\dot{L}^2(0,2\pi))^2$ for any $0\leq s<1$.\label{lac_vel_temp_nb}
\end{enumerate}
\end{proposition}
}
\begin{remark}%\label{rem_lac_nb}
\noindent {\blue Similar to the barotropic case (Remark \ref{rem_lac_b}),} lack of null controllability of the system \eqref{lcnse_nb}-\eqref{in_cd_nb}-\eqref{bd_cd_nb2} or \eqref{lcnse_nb}-\eqref{in_cd_nb}-\eqref{bd_cd_nb3} {\blue is open} when the time is small, in particular, when $0<T<\frac{2\pi}{\bar{u}}$. {\blue Moreover,} null controllability of the system \eqref{lcnse_nb} at time $T=\frac{2\pi}{\bar{u}}$ is inconclusive in all cases, whether there is a control act in density, velocity or temperature.
\end{remark}
{\blue
Like the barotropic case, null controllability at some time $T$ implies approximate controllability at that time $T$ of the system \eqref{lcnse_nb}, thanks to the backward uniqueness property of \eqref{lcnse_nb} (Section \ref{back_uniq}) and the following result shows that the restriction $(\lambda_0,\kappa_0)\in\mc{S}$ is not sufficient to conclude null controllability of \eqref{lcnse_nb}.
\begin{proposition}\label{prop_lac_null_app_nb}
There exist constants $(\lambda_0,\kappa_0)\in\mc{S}$ and $\bar{\rho},\bar{u},\bar{\theta},R,c_0>0$ for which the systems \eqref{lcnse_nb}-\eqref{in_cd_nb}-\eqref{bd_cd_nb1}, \eqref{lcnse_nb}-\eqref{in_cd_nb}-\eqref{bd_cd_nb2} and \eqref{lcnse_nb}-\eqref{in_cd_nb}-\eqref{bd_cd_nb3} are not approximately controllable at any time $T>0$ in the space $(L^2(0,2\pi))^3$.
\end{proposition}

\begin{remark}
Similar to the barotropic case, there exist constants for which the operator $A^*$ (defined by \eqref{op_A^*_nb}) has eigenvalues with geometric multiplicity greater than $1$. However, characterization of these constants is quite difficult due to the complicated cubic characteristic polynomial \eqref{ch_pol_rn_nb}.
\end{remark}

}

\subsection{An Ingham-type inequality}
{\blue
One of the main ingredients to prove the null controllability results for both barotropic and non-barotropic systems (Theorem \ref{thm_den_b} - Theorem \ref{thm_vel_b} - Theorem \ref{thm_null_nb}) is the following Ingham-type inequality:}
\begin{lemma}[{\cite[Proposition 1.6]{Bhandari22}}]\label{lem_ingham}
Let $\{\nu_n^h\}_{n\in\mathbb Z}$ and $\{\nu_n^p\}_{n\in \mathbb Z}$ be two sequences in $\mathbb C$ with the following properties: there exists $N\in \mathbb N$, such that
\begin{itemize}
\item[(H1)] for all $n, l\in \mathbb Z$,  $\nu_n^h\neq\nu_l^h$ unless $n=l$;
\item[(H2)] $\nu_n^h=\beta+\tau ni+e_n$ for all  $|n|\geq N$;
\end{itemize}
where $\tau>0, \beta\in\mathbb C$ and $\{e_n\}_{|n|\geq N}\in\blue\ell_2$.

Also, there exists constants $A_0\geq0$, $B_0\geq\delta$ with $\delta>0$ and some $\epsilon>0, r>1$ for which $\{\nu_n^p\}_{n\in\mathbb Z}$ satisfies
\begin{itemize}
\item[(P1)] for all $n, l\in \mb Z$, $\nu_n^p\neq\nu_l^p$ unless $n=l$;
\item[(P2)] $\frac{-\Re(\nu_n^p)}{|\Im(\nu_n^p)|}\geq \widehat c$ for some $\widehat c>0$ and for all $\mod{n}\geq N$;
\item[(P3)] $\mod{\nu_n^p-\nu_l^p}\geq\delta\mod{n^r-l^r}$ for all $n\neq l$ with $\mod{n},\mod{l}\geq N$ and
\item[(P4)] $\blue \epsilon(A_0+B_0\mod{n}^r)\leq\mod{\nu_n^p}\leq A_0+B_0\mod{n}^r$ for all $\mod{n}\geq N$.
\end{itemize}
We also assume that	the families are disjoint, i.e.,  $$ \left\{\nu_n^h, n\in \mathbb{Z}\right\}\cap\left\{\nu_n^p, n\in \mathbb{Z}\right\}=\emptyset.$$ 

\smallskip

\noindent Then, for any time $T>\frac{2\pi}{\tau}$, there exists a positive constant $C$ depending only on $T$ such that
\begin{align}\label{ingham-ineq}
\int_0^T \left|\sum_{n\in \mathbb Z}a_n e^{\nu_n^pt}+\sum_{n\in\mathbb Z}b_ne^{\nu_n^ht} \right|^2dt\geq C\left(\sum_{n\in\mathbb Z}| a_n|^2e^{2\Re(\nu_n^p)T}+\sum_{n\in\mathbb Z}| b_n|^2\right),
\end{align}
for all sequences $\{a_n\}_{n\in \mathbb Z}$ and $\{b_n\}_{n\in\mathbb Z}$ in $\blue\ell_2$.
\end{lemma}

{\blue
	
\smallskip

\begin{remark}\label{rem_ingham}	
In the proof of Lemma \ref{lem_ingham}, we (\cite{Bhandari22}) have used the following parabolic and hyperbolic Ingham inequalities for the families $(e^{\nu_n^pt})_{n\in\mb{Z}}$ and $(e^{\nu_n^ht})_{n\in\mb{Z}}$ respectively:
\begin{align}
\int_{0}^{T}\mod{\sum_{n\in\mb{Z}}a_ne^{\nu_n^pt}}^2dt&\geq C_1\sum_{n\in\mb{Z}}\mod{a_n}^2e^{2\Re(\nu_n^p)T},\ \ \text{for any}\ T>0,\label{para_ingham}\\
C_2\sum_{n\in\mb{Z}}\mod{b_n}^2\leq\int_{0}^{T}\mod{\sum_{n\in\mb{Z}}b_ne^{\nu_n^ht}}^2dt&\leq C_3\sum_{n\in\mb{Z}}\mod{b_n}^2,\ \ \text{for any}\ T>\frac{2\pi}{\tau},\label{hyper_ingham}
\end{align}
for some $C_i>0, i=1,2,3$. If the sequence $(\nu_n^h)_{n\in\mb{Z}}$ satisfy hypotheses (H1)-(H2), then the hyperbolic Ingham inequality \eqref{hyper_ingham} can be deduced from the proof of Ingham \cite{Ingham}; see for instance \cite[Proposition 3.1]{Chowdhury14}. On the other hand, the proof of parabolic Ingham inequality \eqref{para_ingham} requires the existence of a biorthogonal family $(q_k)_{k\in\mb{Z}}\subset L^2(0,T)$ of $(e^{\nu_n^pt})_{n\in\mb{Z}}$ with the estimate $\norm{q_k}_{L^2(0,T)}\leq C_{\epsilon}e^{\epsilon\Re(\nu_n^p)}$ for some $C_{\epsilon}>0$, see for instance \cite[Proposition 3.2]{Lopez02} and \cite[Proposition 3.2-3.3]{Chowdhury14}; see also \cite[Theorem 1.1]{Hansen91} for the existence of a biorthogonal family in this setup. Note that, the hypotheses (P1)--(P4) can be relaxed to the following:
\begin{equation}\label{new_hyp}
\begin{cases}
\Re(-\nu_n^p)\geq\hat{c}\mod{\nu_n^p},\ \ \mod{\nu_n^p-\nu_m^p}\geq\delta\mod{n-m},\ \ \forall n,m\in\mb{Z},\\
\sum_{n\in\mb{Z}}\frac{1}{\mod{\nu_n^p}}<\infty,
\end{cases}
\end{equation}
for some $\hat{c},\delta>0$. In this setup, we refer to \cite[Proposition 3.4]{Cara10} for a proof of the inequality \eqref{para_ingham}; see also \cite[Proposition 3.2]{Lopez02} for a version of \eqref{para_ingham} when the sequence $(\nu_n^p)_{n\in\mb{Z}}\subset\mb{R}$. Moreover, when the eigenvalues $(\nu_n^p)_{n\in\mb{Z}}$ fails to satisfy the gap condition (hypothesis (P3)) but admits a good approximation (by rational numbers), there exists a biorthogonal sequence to the family $(e^{\nu_n^pt})_{n\in\mb{Z}}$ in $L^2(0,T)$ with the required estimate (see for instance \cite[Lemma 2]{Luca13}), giving the inequality \eqref{para_ingham} in this case also. As a consequence, the combined parabolic-hyperbolic Ingham-type inequality \eqref{ingham-ineq} can also be deduced under these new assumptions on the sequence $(\nu_n^p)_{n\in\mb{Z}}$.
\end{remark}
}

\subsection*{Notations} For any vector $\mathbf{v}$, we denote its transpose by $\mathbf{v}^{\dagger}$ (instead of $\mathbf{v}^T$). Throughout the article, $C>0$ denotes a generic constant that may depend on the time $T$. {\blue Further, we denote the complex conjugate of $f(t,x)$ by $\overline{f(t,x)}$ instead to $\bar{f}(t,x)$ to avoid confusion with the real constants $\bar{\rho},\bar{u}$ and $\bar{\theta}$.}

\bigskip

Proving null controllability of the systems \eqref{lcnse_b} and \eqref{lcnse_nb} using a boundary control is equivalent to proving an observability inequality for the corresponding adjoint systems. Spectrum of the associated linearized operators (for the adjoint systems) and the above Ingham-type inequality \eqref{ingham-ineq} plays a crucial role to prove such observability inequalities. For the system \eqref{lcnse_b} (barotropic fluids), spectrum of the associated adjoint operator consists of two branches of complex eigenvalues, namely, the hyperbolic and parabolic branches. The hyperbolic branch has eigenvalues with the real part converging to $-\frac{b\bar{\rho}}{\mu_0}$, whereas real part of the parabolic branch diverges to $-\infty$. We have obtained explicit expressions of eigenvalues and eigenfunctions in terms of a Riesz basis (See Lemma \ref{lemma_eigen_b} for details). For the non-barotropic fluids (that is, system \eqref{lcnse_nb}), we get three branches of complex eigenvalues; one is of the hyperbolic type, and two are parabolic types. Similar to the barotropic case, the real part of the hyperbolic branch converges to $-\frac{R\bar{\theta}}{\lambda_0}$ and real parts of both the parabolic branches diverge to $-\infty$. In this case, we have obtained explicit expressions of eigenfunctions and asymptotic behavior of the eigenvalues (Lemma \ref{lemma_eigen_nb}). We also proved that the eigenfunctions form a Riesz basis in $(\dot{L}^2(0,2\pi))^2$ for the barotropic system ({\blue Proposition \ref{prop_rb_b}}) and {\blue in} $(\dot{L}^2(0,2\pi))^3$ for the non-barotropic system ({\blue Proposition \ref{prop_rb_nb}}). Then, by writing the solutions to the corresponding adjoint systems in terms of the eigenfunctions, the null controllability results have been proved using the combined parabolic-hyperbolic Ingham type inequality \eqref{ingham-ineq}.

\bigskip

A vast amount of literature is available on the controllability of Navier-Stokes equations for incompressible fluids. For instance, one can see the works of Coron \cite{Coron95}, Coron and Fursikov \cite{Coron96}, Fursikov and Imanuvilov \cite{Fursikov96,Fursikov99}, Imanuvilov \cite{Imanuvilov98,Imanuvilov01}, Fern\'andez-Cara et al.  \cite{Cara04,Cara06}, Guerrero \cite{Guerrero06}, Coron and Guerrero \cite{Coron09}, Chapouly \cite{Chapouly09}, Coron and Lissy \cite{Coron14}, Badra, Ervedoza and Guerrero \cite{Badra16}, Coron, Marbach and Sueur \cite{Coron20}. In comparison, for compressible fluids, less works are available on the Navier-Stokes system's controllability. In this context, we first mention the work of Ervedoza et al. \cite{Ervedoza12}, where the authors established local exact controllability of one dimensional {\blue (nonlinear)} compressible Navier-Stokes system at a large time $T$ in the space $H^3(0,L)\times H^3(0,L)$ using two boundary controls. This result has been improved in \cite{Ervedoza18} where the null controllability is achieved in the space $H^1(0,L)\times H^1(0,L)$. {\blue However, studying the controllability of the (nonlinear) compressible Navier-Stokes system using only one boundary control is challenging and an interesting open problem. In this article, we focus only on the linearized system and study controllability properties.}

It is known in \cite{Chowdhury12} that, for barotropic fluids, the one-dimensional compressible Navier-Stokes system linearized around $(\bar{\rho},0)$ (with $\bar{\rho}>0$) cannot be null controllable at any time $T>0$ by using a boundary control or a localized distributed control. For the linearized system around $(\bar{\rho},\bar{u})$ (with $\bar{\rho},\bar{u}>0$), the authors in \cite{Chowdhury14} proved null controllability of the Navier-Stokes equations (with homogeneous periodic boundary conditions) for viscous, compressible isothermal barotropic fluids at time $T$ (large) in the space $\dot{H}^1_{\per}(0,2\pi)\times L^2(0,2\pi)$, when there is an interior control {\blue acting} only in the velocity equation. They also proved that the space $\dot{H}^1_{\per}(0,2\pi)\times L^2(0,2\pi)$ is optimal in the sense that if we choose the initial state from $\dot{H}^s_{\per}(0,2\pi)\times L^2(0,2\pi)$ with $0\leq s<1$, the linearized system cannot be null controllable at any time $T>0$. In the case of linearization around $(\bar{\rho},\bar{u})$ with $\bar{\rho},\bar{u}>0$, the compressible Navier-Stokes system \eqref{lcnse_b} is equivalent (in some sense) to the transformed system in \cite{Martin13}. Using a moving distributed control, the authors in \cite{Martin13} proved the null controllability of a one-dimensional structurally damped wave equation in the space $H^{s+2}\times H^s$ for $s>\frac{15}{2}$. There is a generalization to this result in higher dimensions by Chaves-Silva, Rosier, and Zuazua \cite{Silva14}. Inspired by the work of Martin, Rosier and Rouchon \cite{Martin13}, Chowdhury and Mitra in \cite{Chowdhury15b} proved the null controllability of the same compressible Navier-Stokes system linearized around $(\bar{\rho},\bar{u})$ at time $T$ (large) by using a boundary control acting on the velocity component through periodic conditions, provided the initial states are regular enough, more precisely, in the space $\dot{H}^{1+s}_{\per}(0,2\pi)\times \dot{H}^s_{\per}(0,2\pi)$ with $s>4.5$. However, the question of null controllability at a large time $T$ in the space $\dot{H}^{1+s}_{\per}(0,2\pi)\times \dot{H}^s_{\per}(0,2\pi)$ with $s\leq4.5$ was unaddressed in \cite{Chowdhury15b}, and up to the author's knowledge, there has been no improvement of this result. In this article, we have answered this question (Theorem \ref{thm_vel_b}). In fact, we have proved null controllability of the linearized compressible Navier-Stokes system for barotropic fluids \eqref{lcnse_b} at large time $T$ in the space {\blue $\dot{H}^1_{\per}(0,2\pi)\times \dot{L}^2(0,2\pi)$} by using one boundary control acting in the velocity component. We have also proved that our result is optimal in the sense that the system cannot be null controllable by a boundary control ({\blue acting} in velocity) when the initial states belong to the space $\dot{H}^s_{\per}(0,2\pi)\times \dot{L}^2(0,2\pi)$ with $0\leq s<1$. {\blue On the other hand,} when a control {\blue is acting} only in the density component through periodic boundary conditions, we have established null controllability of the linearized system \eqref{lcnse_b}-\eqref{in_cd_b}-\eqref{bd_cd_b1} at large time $T$ in the space $\blue (\dot{L}^2(0,2\pi))^2$ and that null controllability fails at small time $T$. {\blue In this context, it is worth mentioning that the authors in \cite{Chowdhury15b} could only proved null controllability in the space $\dot{H}^{1+s}_{\per}(0,2\pi)\times\dot{H}^s_{\per}(0,2\pi)$ with $s>4.5$ is because of the biorthogonal estimate (corresponding to the hyperbolic family $(e^{\nu_n^ht})_{n\in\mb{Z}}$) of order $\mod{k}^4$ (see Proposition 3.2 in \cite{Chowdhury15b}), which forces the initial state to be more regular. However, in our case, we have used the Ingham-type inequality \eqref{ingham-ineq} which do not require any biorthogonal estimate of the family $(e^{\nu_n^ht})_{n\in\mb{Z}}$, giving the optimal space for null controllability of \eqref{lcnse_b}. Furthermore, null controllability of the system \eqref{lcnse_b} under the assumption $\frac{2\sqrt{b\bar{\rho}-\bar{u}^2}}{\mu_0}\in\mb{N}$ was not addressed properly in \cite[Remark 3.4]{Chowdhury15b} and we have proved that, under this assumption, the system \eqref{lcnse_b} fails to satisfy the unique continuation property; as a result the system \eqref{lcnse_b} cannot be approximately controllable in $(L^2(0,2\pi))^2$ at any time $T>0$.

}

\smallskip

For the non-barotropic fluids, it is known in \cite{Debayan15} that the compressible Navier-Stokes system linearized around $(\bar{\rho},0,\bar{\theta})$ (with $\bar{\rho},\bar{\theta}>0$) is not null controllable at any time $T>0$ by using a boundary control or a localized distributed control. For the linearization around $(\bar{\rho},\bar{u},\bar{\theta})$ with $\bar{\rho},\bar{u},\bar{\theta}>0$, it is only known that the system is not null controllable at small time by a localized interior control or a boundary control acting on the velocity component (see \cite[Theorem 1.5]{Debayan15} for instance). To the author's knowledge, no controllability result is known for the linearized system around $(\bar{\rho},\bar{u},\bar{\theta})$, that is, the system \eqref{lcnse_nb}, when the time is large, which is studied for the first time in this article. {\blue Like the barotropic case, controllability of the (nonlinear) compressible Navier-Stokes system for non-barotropic fluids using boundary control(s) is intricate to study and is an open problem. In this article, we study null and approximate controllability of only the linearized version, mainly the system \eqref{lcnse_nb}. Since the system \eqref{lcnse_nb} consists of a transport equation coupled with two parabolic equations, it is worth mentioning some results known for the coupled parabolic equations. In \cite{Cara10}, the authors considered a $2$-parabolic system with diffusion coefficients $d_1,d_2>0$ and with zeroth order coupling. They proved that the coupled parabolic system is (boundary) approximately controllable at time $T>0$ if and only if $d_1=d_2$ or $\sqrt{\frac{d_1}{d_2}}\notin\mb{Q}$. Moreover, they also proved that, when $d_1=d_2$, the system is (boundary) null controllable at any time $T>0$. If $\sqrt{\frac{d_1}{d_2}}\notin\mb{Q}$, the authors in \cite{Luca13} provided an example of a system which is approximately controllable but not null controllable at any time $T>0$. This phenomena occurs because eigenvalues of the associated operator condensate; as a consequence, fails to satisfy the gap condition, which is very crucial to obtain $L^2$-estimate of the biorthogonal family. However, they \cite{Luca13} also proved that, if $d_1=1$ and $\sqrt{d_2}\notin\mb{Q}$ is such that we can approximate it as $\mod{\sqrt{d_2}-\frac{a}{b}}>\frac{C}{b^N}$ for some $C,N>0$ and all rational numbers $\frac{a}{b}$, then the system is null controllable at any time $T>0$. Such approximation is referred as "Diophantine approximations". Thus our assumption in Theorem \ref{thm_null_nb} seems appropriate. We refer to \cite{Ammar14} for more insights in this matter, in terms of condensation index of the eigenvalues and minimal time for null controllability of one dimensional coupled parabolic equations. In the context of controllability results for general coupled parabolic equations, we refer to the works of \cite{Guerrero07,Benabdallah20,Benabdallah14,Ammar11a,Ammar05,Ammar11} (and references therein).

}

The main difficulty in the linearized compressible Navier-Stokes system is the presence of transport and parabolic coupling. The thermoelasticity system is also an example involving both transport and parabolic effects. Lebeau and Zuazua\cite{Lebeau98} have studied distributed {\blue controllability} for thermoelasticity systems. Following \cite{Lebeau98}, Beauchard et al. in \cite{Beauchard20} proved null controllability for some coupled transport-parabolic systems when an interior control {\blue is acting on the system}. They proved null controllability at large time $T$ in the space $L^2(0,2\pi)\times \dot{L}^2(0,2\pi)$ by one interior control {\blue acting} in the density equation and in the space $\dot{H}^2(0,2\pi)\times H^2(0,2\pi)$ when only one interior control {\blue is acting} in the velocity equation; see also \cite{Koenig23} {\blue for an improvement of the controllability space to $\dot{H}^1(0,2\pi)\times L^2(0,2\pi)$ in the velocity (internal) control case}.

{\blue
	
In \cite{Bhandari22}, Bhandari, Chowdhury, Dutta and the author considered the linearized compressible Navier-Stokes system \eqref{lcnse_b} with Dirichlet and mixed (Periodic-Dirichlet type) boundary conditions. We proved that the system \eqref{lcnse_b} (with Dirichlet boundary conditions) is null controllable at large time $T$ in the space $L^2(0,1)\times L^2(0,1)$ by using a boundary control acting in the density part. On the other hand, when a boundary control is acting in the velocity component, we proved that the system \eqref{lcnse_b} (with Dirichlet-Periodic boundary conditions) is null controllable at large time $T$ in the space $H^{\frac{1}{2}}(0,1)\times L^2(0,1)$. We have applied the Ingham-type inequality \eqref{ingham-ineq} and the moments method to prove these controllability results.
In contrast to \cite{Bhandari22},} the main contribution of this article is that we prove the null controllability of the one-dimensional linearized compressible Navier-Stokes system for both barotropic and non-barotropic fluids by using only one boundary control {\blue (and with different boundary conditions than \cite{Bhandari22})}. We consider all the possible cases of the act of control for both systems \eqref{lcnse_b} and \eqref{lcnse_nb}. Further, we {\blue obtain} better regularity of the initial states for the controllability of barotropic system \eqref{lcnse_b} compared to \cite{Chowdhury15b}. In the case of non-barotropic fluids, since the transport equation does not affect the temperature equation, it is pretty natural to obtain similar spaces of null controllability of the system \eqref{lcnse_nb}. The combined parabolic-hyperbolic Ingham type inequality (\Cref{lem_ingham}) helps us obtain each case's best possible results (with respect to the state space). Our results cannot be obtained as a consequence of interior control results by the extension method. In addition, when the boundary control acts in the density component, we prove that both systems \eqref{lcnse_b} and \eqref{lcnse_nb} are not null controllable at small time. The proof is inspired from \cite{Beauchard20} and is independent of that in \cite{Debayan15}.

The result stated in \Cref{thm_den_b} is similar to the results in \cite{Beauchard20}, showing that we can achieve the space $(\dot{L}^2(0,2\pi))^2$ in the case of only one boundary control ({\blue acting} in density) also. Likewise the case of interior control \cite{Chowdhury14,Beauchard20,Koenig23}, we also obtain similar results for our boundary control case ({\blue acting} in velocity) (\Cref{thm_vel_b}).

\bigskip

The rest of the article is organized as follows:
\begin{itemize}
\item[--] In Section 2, we prove {\blue all the controllability results for the barotropic system} \eqref{lcnse_b} at a large time $T$ using a boundary control that acts either in density or velocity, that is, \Cref{thm_den_b} and \Cref{thm_vel_b}. {\blue The proof of lack of approximate controllability at any time $T$ under the restriction on the coefficients (\Cref{prop_lac_app_b}) is also included in this section.}
\item[--] In Section 3, {\blue we consider the non-barotropic system \eqref{lcnse_nb} and give all the related controllability results} based on the act of the control, namely the proofs of \Cref{thm_null_nb} and {\blue \Cref{prop_lac_null_nb}. We have also included the proof of lack of approximate controllability result at any time $T$ (\Cref{prop_lac_null_app_nb}).}
\item [--] In section 4, we give few comments and {\blue open questions regarding controllability results under Dirichlet or Neumann boundary conditions and the backward uniqueness property.}
\item[--] For the sake of completeness, we give the proof of well-posedness result (\Cref{lem_ex_sg_nb}) for the non-barotropic system \eqref{lcnse_nb} in Appendix \ref{app_well_p_nb}.
\end{itemize}
\subsection*{Acknowledgments}
The author would like to thank his PhD supervisor Dr. Shirshendu Chowdhury for suggesting this problem and fruitful discussions. The author would also like to thank Dr. Rajib Dutta for careful reading and improvement of the manuscript. {\blue The author would like to express sincere thanks to the anonymous referees for their helpful comments, suggestions and corrections that greatly improved the manuscript.}

This work is supported by the Prime Minister's Research Fellowship (ref. no. 41-1/2018-TS-1/PMRF), Government of India.
%%%%%%%%%%%%%%%%%%%%%%%%%%%%%%%%%%%%%%%%%%%%%%%%%%%%%%%%%%%%%%%%%%%%%%%%%
\section{Controllability of linearized compressible Navier-Stokes system (barotropic)}
\subsection{Functional setting}
We denote the positive constants
\begin{equation}\label{diff_coeff_b}
\mu_0:=\frac{\blue\lambda+2\mu}{\bar{\rho}},\ \ b:=a\gamma\bar{\rho}^{\gamma-2}.
\end{equation}
{\blue Here we re-denote the diffusion coefficient by $\mu_0$ (instead of $\lambda_0$) to separate it from the non-barotropic case.} We define the inner product in the space $(L^2(0,2\pi))^2$ as follows
\begin{equation*}
\ip{\vector{f_1}{g_1}}{\vector{f_2}{g_2}}_{\blue L^2\times L^2}:=b\int_{0}^{2\pi}f_1(x)\overline{f_2(x)}dx+\bar{\rho}\int_{0}^{2\pi}g_1(x)\overline{g_2(x)}dx,
\end{equation*}
for $f_i,g_i\in L^2(0,2\pi), i=1,2$. {\blue From now on-wards, the notation $\ip{\cdot}{\cdot}_{L^2\times L^2}$ means the above inner product in $L^2\times L^2$.} We write the system \eqref{lcnse_b} in abstract differential equation
\begin{equation}\label{lcnse_abs_b}
U^{\prime}(t)=AU(t),\ \ U(0)=U_0,\ \ t\in (0,T),
\end{equation}
where $U:=(\rho,u)^{\dagger}, U_0:=(\rho_0,u_0)^{\dagger}$ and the operator $A$ is given by
\begin{equation*}
A:=\begin{pmatrix}-\bar{u}\partial_x&-\bar{\rho}\partial_x\\[4pt]-b\partial_x&\mu_0\partial_{xx}-\bar{u}\partial_x\end{pmatrix}
\end{equation*}
with the domain
\begin{equation*}
\mc{D}(A):=H^1_{\per}(0,2\pi)\times H^2_{\per}(0,2\pi).
\end{equation*}
The adjoint of the operator $\blue A$ is given by
\begin{equation}\label{op_A*_b}
A^*:=\begin{pmatrix}\bar{u}\partial_x&\bar{\rho}\partial_x\\[4pt]b\partial_x&\mu_0\partial_{xx}+\bar{u}\partial_x\end{pmatrix}
\end{equation}
with the same domain $\mc{D}(A^*)=\mc{D}(A)$. The adjoint system is then given by
\begin{equation}\label{lcnse_adj_b}
\begin{dcases}
-\sigma_t(t,x)-\bar{u}\sigma_x(t,x)-\bar{\rho}v_x(t,x)=0,\ &\text{in}\ (0,T)\times(0,2\pi),\\
-v_t(t,x)-\mu_0v_{xx}(t,x)-\bar{u}v_x(t,x)-b\sigma_x(t,x)=0,\ &\text{in}\ (0,T)\times(0,2\pi),\\
\sigma(t,0)=\sigma(t,2\pi),\ \ v(t,0)=v(t,2\pi),\ \  v_x(t,0)=v_x(t,2\pi),\ &t\in (0,T),\\
\sigma(T,x)=\sigma_T(x),\ \ v(T,x)=v_T(x),\ \ &x\in(0,2\pi).
\end{dcases}
\end{equation}
We now write the adjoint system with source terms $f$ and $g$.
\begin{equation}\label{lcnse_adj_b_s}
\begin{dcases}
-\sigma_t(t,x)-\bar{u}\sigma_x(t,x)-\bar{\rho}v_x(t,x)=f,\ &\text{in}\ (0,T)\times(0,2\pi),\\
-v_t(t,x)-\mu_0v_{xx}(t,x)-\bar{u}v_x(t,x)-b\sigma_x(t,x)=g,\ &\text{in}\ (0,T)\times(0,2\pi),\\
\sigma(t,0)=\sigma(t,2\pi),\ \ v(t,0)=v(t,2\pi),\ \ v_x(t,0)=v_x(t,2\pi),\ &t\in (0,T),\\
\sigma(T,x)=\sigma_T(x),\ \ v(T,x)=v_T(x),\ \ &x\in(0,2\pi).
\end{dcases}
\end{equation}

\subsection{Well-posedness of the systems}
This section devotes to the well-posedness of the system \eqref{lcnse_b} under the boundary conditions \eqref{bd_cd_b1}, \eqref{bd_cd_b2} and the initial conditions \eqref{in_cd_b}, and the adjoint system \eqref{lcnse_adj_b_s}.

\smallskip

\noindent When there is no control acting on the system, we have the existence of solutions to the system \eqref{lcnse_b} using semigroups.
\begin{lemma}[{\cite[Lemma 2.1]{Chowdhury14}}]
The operator $A$ (resp. $A^*$) generates a $\mc{C}^0$-semigroup of contractions on $(L^2(0,2\pi))^2$. Moreover, for every $U_0\in(L^2(0,2\pi))^2$ the system \eqref{lcnse_abs_b} admits a unique solution $U$ in $\mc{C}^0([0,T];(L^2(0,2\pi))^2)$ and
\begin{equation*}
\norm{U(t)}_{(L^2(0,2\pi))^2}\leq C\norm{U_0}_{(L^2(0,2\pi))^2}
\end{equation*}
for all $t\geq0$.
\end{lemma}
\noindent The following lemma shows the existence of a unique solution to the adjoint system \eqref{lcnse_adj_b_s}.
\begin{lemma}\label{well_posed_adj_b}\textbf{}
\begin{enumerate}
\item {\blue For any given source term $(f,g)\in L^2(0,T;(L^2(0,2\pi))^2)$ and given terminal data $(\sigma_T,v_T)\in (L^2(0,2\pi))^2$, the adjoint system \eqref{lcnse_adj_b_s} has a unique solution $(\sigma,v)$ in the space}
\begin{equation*}
\mc{C}^0([0,T];L^2(0,2\pi))\times[\mc{C}^0([0,T];L^2(0,2\pi))\cap L^2(0,T;H^1_{\per}(0,2\pi))].
\end{equation*}
{\blue Furthermore, we have the hidden regularity property $\sigma(\cdot,2\pi)\in L^2(0,T)$.}
\item {\blue For any given $(f,g)\in L^2(0,T;H^1_{\per}(0,2\pi)\times L^2(0,2\pi))$ and $(\sigma_T,v_T)\in H^{-1}_{\per}(0,2\pi)\times L^2(0,2\pi)$, the system \eqref{lcnse_adj_b} admits a unique solution $(\sigma,v)\in\mc{C}^0([0,T];H^{-1}_{\per}(0,2\pi)\times L^2(0,2\pi))$.

In particular, when $(\sigma_T,v_T)=(0,0)$, the solution $(\sigma,v)$ belong to the space $$\mc{C}^0([0,T];H^1_{\per}(0,2\pi))\times[\mc{C}^0([0,T];H^1_{\per}(0,2\pi))\cap L^2(0,T;H^2_{\per}(0,2\pi))].$$}
\end{enumerate}
\end{lemma}
{\blue Proof of the first part follows from \cite[Proposition 2.2]{Bhandari22}, see also Appendix D in \cite{Bhandari22} for the hidden regularity result. For the second part, we refer to \cite[Proposition 2.5]{Chowdhury14}, see also \cite[Chapter 4]{Girinon08}.}

\smallskip

\noindent Once we have the existence results of the homogeneous system (without any boundary control) associated to the system \eqref{lcnse_b}, we can now guarantee the existence of a unique solution to the system \eqref{lcnse_b} (in the sense of transposition) when there is a boundary control $p$ (resp. $q$) {\blue acting} in density (resp. velocity) in the space $L^2(0,T)$. Before writing the statements, let us first define the notion of a solution in the sense of transposition.
\begin{definition}\label{def_sol_b}
{\blue	We give the following definitions based on the act of the control.}
\begin{enumerate}
\item For any given initial state $(\rho_0,u_0)\in(L^2(0,2\pi))^2$ and boundary control $p\in L^2(0,T)$, a function $(\rho,u)\in L^2(0,T;(L^2(0,2\pi))^2)$ is a solution to the system \eqref{lcnse_b}-\eqref{in_cd_b}-\eqref{bd_cd_b1} if for any given $(f,g)\in L^2(0,T;(L^2(0,2\pi))^2)$, the following identity holds true:
\begin{align*} 
&\blue\int_{0}^{T}\ip{(\rho(t,\cdot),u(t,\cdot))^{\dagger}}{(f(t,\cdot),g(t,\cdot))^{\dagger}}_{L^2\times L^2}dt\\
&\blue=\ip{(\rho_0(\cdot),u_0(\cdot))^{\dagger}}{(\sigma(0,\cdot),v(0,\cdot))^{\dagger}}_{L^2\times L^2}+b\int_{0}^{T}\left[\bar{u}\overline{\sigma(t,2\pi)}+\bar{\rho}\overline{v(t,2\pi)}\right]p(t)dt,
\end{align*}
where $(\sigma,v)$ is the unique weak solution to the adjoint system \eqref{lcnse_adj_b_s} with $\blue(\sigma_T,v_T)=(0,0)$.

\smallskip 
	
\item For any given initial state $(\rho_0,u_0)\in H^1_{\per}(0,2\pi)\times L^2(0,2\pi)$ and boundary control $q\in L^2(0,T)$, a function $(\rho,u)\in L^2(0,T;H^{-1}_{\per}(0,2\pi)\times L^2(0,2\pi))$ is a solution to the system \eqref{lcnse_b}-\eqref{in_cd_b}-\eqref{bd_cd_b2} if for any given $(f,g)\in L^2(0,T;H^1_{\per}(0,2\pi)\times L^2(0,2\pi))$, the following identity holds true: 
\begin{align*}
&\blue\int_{0}^{T}\ip{(\rho(t,\cdot),u(t,\cdot))^{\dagger}}{(f(t,\cdot),g(t,\cdot))^{\dagger}}_{H^{-1}_{\per}\times L^2,H^1_{\per}\times L^2}dt\\
&\blue=\ip{(\rho_0(\cdot),u_0(\cdot))^{\dagger}}{(\sigma(0,\cdot),v(0,\cdot))^{\dagger}}_{L^2\times L^2}+\bar{\rho}\int_{0}^{T}\left[b\overline{\sigma(t,2\pi)}+\bar{u}\overline{v(t,2\pi)}+\mu_0\overline{v_x(t,2\pi)}\right]q(t)dt,
\end{align*}
{\blue where} $(\sigma,v)$ is the unique weak solution to the adjoint system \eqref{lcnse_adj_b_s} with $(\sigma_T,v_T)=(0,0)$.
\end{enumerate}
\end{definition}
\begin{proposition}[{\cite[Theorem 2.4]{Bhandari22}}]
For any given initial state $(\rho_0,u_0)\in (L^2(0,2\pi))^2$ and boundary control $p\in L^2(0,T)$, the system \eqref{lcnse_b}-\eqref{in_cd_b}-\eqref{bd_cd_b1} admits a unique solution $(\rho,u)$ in the space
\begin{equation*}
\mc{C}^0([0,T];L^2(0,2\pi))\times[\mc{C}^0([0,T];L^2(0,2\pi))\cap L^2(0,T;H^1_{\per}(0,2\pi))].
\end{equation*}
\end{proposition}
\begin{proposition}
For any given initial state $(\rho_0,u_0)\in\blue H^1_{\per}(0,2\pi)\times L^2(0,2\pi)$ and boundary control $q\in L^2(0,T)$, the system \eqref{lcnse_b}-\eqref{in_cd_b}-\eqref{bd_cd_b2} admits a unique solution $(\rho,u)$ in the space
\begin{equation*}
\blue\mc{C}^0([0,T];H^{-1}_{\per}(0,2\pi))\times[\mc{C}^0([0,T];H^{-1}_{\per}(0,2\pi))\cap L^2(0,T;L^2(0,2\pi))].
\end{equation*}
%Moreover, the operator $q\mapsto(\rho,u)$ is linear and continuous from $L^2(0,T)$ into $L^2(0,T;(H^1_{\per}(0,2\pi))^{\prime})\times L^2(0,T;L^2(0,2\pi))$.
\end{proposition}
{\blue The proof of this result (velocity case) can be done in a standard fashion using the semigroup property of the homogeneous system and the parabolic regularity, see for instance \cite{Chowdhury13,Girinon08}.}
%{\blue

%TO DO. Write well-posedness results for velocity case properly.

%Write the equivalent identity to controllability in each cases separately.
%}

\subsection{Spectral analysis of $A^*$}
We denote the spectrum of $A^*$ by $\sigma(A^*)$. The following lemma gives behavior of the spectrum of the operator $A^*$.
\begin{lemma}\label{lemma_eigen_b}
The following statements hold.
\begin{enumerate}[(i)]
\item $\ker(A^*)=\text{span}\left\{\vector{1}{1},\vector{1}{-1}\right\}$.
\item $\sup\left\{\Re(\nu)\ :\ \nu\in\sigma(A^*),\ \nu\neq0\right\}<0$.
\item The spectrum of $A^*$ consists of the eigenvalue $\nu_0=0$ and pairs of complex eigenvalues $\{\nu_n^h,\nu_n^p\}_{n\in\mb{Z}^*}$ given as
\begin{align}
\blue\nu_n^h&\blue=-\frac{1}{2}\left(\mu_0n^2-\sqrt{\mu_0^2n^4-4b\bar{\rho}n^2}-2\bar{u}in\right),\label{exp_ev_h_b}\\
\blue\nu_n^{p}&\blue=-\frac{1}{2}\left(\mu_0n^2+\sqrt{\mu_0^2n^4-4b\bar{\rho}n^2}-2\bar{u}in\right),\label{exp_ev_p_b}
\end{align}
for all $n\in\mb{Z}^*$.
\item The eigenvalues satisfy the following properties
\begin{equation*}
\begin{cases}
\lim_{\mod{n}\to\infty}\Re(\nu_n^h)=-\omega_0, \ \ \lim_{\mod{n}\to\infty}\frac{\Re(\nu_n^p)}{n^2}=-\mu_0\\[4pt]
\lim_{\mod{n}\to\infty}\frac{\Im(\nu_n^h)}{n}=\bar{u},\ \ \lim_{\mod{n}\to\infty}\frac{\Im(\nu_n^p)}{n}=\bar{u}
\end{cases}
\end{equation*}
with $\omega_0=\frac{b\bar{\rho}}{\mu_0}$.
\item The eigenfunctions of $A^*$ corresponding to $\nu_n^h$ and $\nu_n^p$ are respectively
\begin{equation}\label{exp_e_fns_b}
\Phi_n^h=\vector{\xi_n^h}{\eta_n^h}=\vector{\bar{\rho}}{\nu_2^n-\bar{u}}e^{inx},\ \ \Phi_n^p=\vector{\xi_n^p}{\eta_n^p}=\vector{\frac{\bar{\rho}}{\nu_1^n-\bar{u}}}{1}e^{inx},
\end{equation}
for $\blue n\in\mb{Z}^*$, where {\blue $\nu_1^n=\frac{1}{in}\nu_n^p$ and $\nu_2^n=\frac{1}{in}\nu_n^h$ for $n\in\mb{Z}^*$.}
%\begin{equation}\label{exp_nu_n}
%\nu_1^n:=\frac{1}{2}\left(\mu_0in+2\bar{u}+i\sqrt{\mu_0^2n^2-4b\bar{\rho}}\right),\ \ \nu_2^n:=\frac{1}{2}\left(\mu_0in+2\bar{u}-i\sqrt{\mu_0^2n^2-4b\bar{\rho}}\right),\ \ n\in\mb{Z}.
%\end{equation}
%\item The eigenfunctions $\{\Phi_n^h,\ \Phi_n^{p}\ : \ n\in\mb{Z}^*\}$ of $A^*$ forms a Riesz basis of $(\dot{L}^2(0,2\pi))^2$.
\end{enumerate}
\end{lemma}
\begin{proof}
%We will prove only the parts (2), (3) and (5). Part (4) can be proved using part (2).

%\smallskip
{\blue
We prove each part separately.}

\begin{enumerate}
{\blue
\item[Part-(i).] This part follows immediately from the fact that $A^*(\xi,\eta)^{\dagger}=0$ implies $(\xi,\eta)=$ constant.

%Let $\Phi=(\xi,\eta)^{\dagger}\in\mc{D}(A^*)$ be such that $A^*\Phi=0$. This gives $\bar{u}\xi_x+\bar{\rho}\eta_x=0$ and $\mu_0\eta_{xx}+\bar{u}\eta_x+b\xi_x=0$ and therefore we have $\mu_0\bar{u}\eta_{xx}+(\bar{u}^2-b\bar{\rho})\eta_x=0$. The boundary conditions $\eta(0)=\eta(2\pi)$ and $\eta_x(0)=\eta_x(2\pi)$ implies $\eta=$ constant and consequently $\xi=$ constant.

\item[Part-(ii).]} Let $\Phi=(\xi,\eta)^{\dagger}\blue\in\mc{D}(A^*)$ be the eigenfunction of $A^*$ corresponding to the eigenvalue $\nu\neq0$. Then, we have
\begin{equation*}
\ip{A^*\vector{\xi}{\eta}}{\vector{\xi}{\eta}}_{\blue L^2\times L^2}=\ip{\nu\vector{\xi}{\eta}}{\vector{\xi}{\eta}}_{\blue L^2\times L^2},
\end{equation*}
that is,
\begin{align*}
\quad b\bar{u}\int_{0}^{2\pi}\overline{\xi(x)}\xi_x(x)dx+b\bar{\rho}\int_{0}^{2\pi}\overline{\xi(x)}\eta_x(x)dx+\mu_0\bar{\rho}\int_{0}^{2\pi}\overline{\eta(x)}\eta_{xx}(x)dx+\bar{\rho}\bar{u}\int_{0}^{2\pi}\overline{\eta(x)}\eta_x(x)dx\\
+b\bar{\rho}\int_{0}^{2\pi}\xi_x(x)\overline{\eta(x)}dx=\blue\nu b\int_{0}^{2\pi}\mod{\xi(x)}^2dx+\nu\bar{\rho}\int_{0}^{2\pi}\mod{\eta(x)}^2dx.
\end{align*}
An integration by parts yields
\begin{equation*}
\blue\Re(\nu)=-\frac{\mu_0\bar{\rho}\norm{\eta_x}_{L^2(0,2\pi)}^2}{b\norm{\xi}_{L^2(0,2\pi)}^2+\bar{\rho}\norm{\eta}_{L^2(0,2\pi)}^2}<0,
\end{equation*}
{\blue which proves part (ii).
}	
\item[{\blue Parts-(iii),(v).}] We denote
\begin{equation*}
\vphi_n(x):=e^{inx},\ \ n\in\mb{Z}.
\end{equation*}
Then the set $\left\{\vector{\vphi_n}{0},\vector{0}{\vphi_n}\ {\blue ; \ n\in\mb{Z}}\right\}$ forms an orthogonal basis of $(L^2(0,2\pi))^2$. Let us define
\begin{equation*}
E_n:=\begin{pmatrix}\vphi_n&0\\[4pt]0&\vphi_n\end{pmatrix},\ \ \text{and}\ \Phi_n:=(\xi_n,\eta_n)^{\dagger},
\end{equation*}
for all $n\in\mb{Z}$. Then, we have the following relation
\begin{equation}
\blue A^*E_n\Phi_n=E_nR_n\Phi_n,\ \ n\in\mb{Z},
\end{equation}
where the matrix $R_n$ for $n\in\mb{Z}$ is given by
\begin{equation}
\blue R_n:=\begin{pmatrix}\bar{u}in&\bar{\rho}in\\[4pt]bin&-\mu_0n^2+\bar{u}in\end{pmatrix},\ \ n\in\mb{Z}.
\end{equation}
Thus, if $(\alpha_n,\nu_n)$ is an eigenpair of $R_n$, then $(E_n\alpha_n,{\blue\nu_n})$ will be an eigenpair of $A^*$. Therefore, it's remains to find the eigenvalues and eigenvectors of the matrix $R_n$ for $n\in\mb{Z}$. The characteristics equation of $R_n$ is
\begin{equation}
\blue\nu^2-(-\mu_0n^2+2\bar{u}in)\nu-\mu_0\bar{u}in^3-\bar{u}^2n^2+b\bar{\rho}n^2=0,
\end{equation}
for all $n\in\mb{Z}$. Therefore, the eigenvalues of the matrix $R_n$ are
\begin{equation*}
\quad\quad\quad\blue\nu_n^h:=\frac{1}{2}\left(-\mu_0n^2+2\bar{u}in+\sqrt{\mu_0^2n^4-4b\bar{\rho}n^2}\right),\ \ \nu_n^p:=\frac{1}{2}\left(-\mu_0n^2+2\bar{u}in-\sqrt{\mu_0^2n^4-4b\bar{\rho}n^2}\right),
\end{equation*}
for all $n\in\mb{Z}$. Note that, {\blue $0$ cannot be an eigenvalue of the matrix $R_n$ for all $n\in\mb{Z}^*$ }and $\bar{u}$ cannot be an eigenvalue of $R_n$ for all $n\in\mb{Z}$, because $b,\bar{\rho},\mu_0,\bar{u}>0$. {\blue Let us denote $\nu_1^n:=\frac{1}{in}\nu_n^p$ and $\nu_2^n:=\frac{1}{in}\nu_n^h$.} To find the eigenvectors of the matrix $R_n$, we first consider the equation
\begin{equation*}
\blue R_n\alpha_n^h=\nu_n^h\alpha_n^h,\ \ n\in\mb{Z},
\end{equation*}
where $\alpha_n^h:=(\alpha_1^n,\alpha_2^n)^{\dagger}$, that is,
\begin{align}
\blue(\bar{u}in-\nu_n^h)\alpha_1^n+\bar{\rho}in\alpha_2^n=0,%\label{ev_rn_b_1_1}\\
\ \ \blue bin\alpha_1^n+(-\mu_0n^2+\bar{u}in-\nu_n^h)\alpha_2^n=0,%\label{ev_rn_b_1_2}
\end{align}
for all $n\in\mb{Z}$. One solution is given by
\begin{equation}%\label{ev_rn_b_h}
\alpha_n^h=\vector{\alpha_1^n}{\alpha_2^n}:=\vector{\bar{\rho}}{\nu_2^n-\bar{u}},\ \ n\in\mb{Z}.
\end{equation}
We next consider the equation
\begin{equation*}
\blue R_n\alpha_n^p=\nu_n^p\alpha_n^p,\ \ n\in\mb{Z},
\end{equation*}
where $\alpha_n^p:=(\beta_1^n,\beta_2^n)^{\dagger}$, that is,
\begin{align}
\blue(\bar{u}in-\nu_n^p)\beta_1^n+\bar{\rho}in\beta_2^n=0,%\label{ev_rn_b_2_1}\\
\blue \ \ bin\beta_1^n+(-\mu_0n^2+\bar{u}in-\nu_n^p)\beta_2^n=0,%\label{ev_rn_b_2_2}
\end{align}
for all $n\in\mb{Z}$. One solution is given by
\begin{equation}%\label{ev_rn_b_p}
\alpha_n^p=\vector{\beta_1^n}{\beta_2^n}:=\vector{\frac{\bar{\rho}}{\nu_1^n-\bar{u}}}{1},\ \ n\in\mb{Z}.
\end{equation}
Thus, the eigenvectors of $R_n$ corresponding to the eigenvalues $\blue\nu_n^h$ and $\blue\nu_n^p$ are respectively
\begin{equation*}
\alpha_n^h=\vector{\alpha_1^n}{\alpha_2^n}=\vector{\bar{\rho}}{\nu_2^n-\bar{u}},\ \ \alpha_n^p=\vector{\beta_1^n}{\beta_2^n}=\vector{\frac{\bar{\rho}}{\nu_1^n-\bar{u}}}{1},\ \ n\in\mb{Z}.
\end{equation*}
Hence, the eigenvalues of the operator $A^*$ are $\blue\nu_0:=0$ and
\begin{equation*}
\quad\quad\quad\blue\nu_n^h:=\frac{1}{2}\left(-\mu_0n^2+2\bar{u}in+\sqrt{\mu_0^2n^4-4b\bar{\rho}n^2}\right),\ \ \nu_n^p:=\frac{1}{2}\left(-\mu_0n^2+2\bar{u}in-\sqrt{\mu_0^2n^4-4b\bar{\rho}n^2}\right),
\end{equation*}
for $n\in\mb{Z}^*$ and the corresponding eigenfunctions are respectively
\begin{equation*}
\Phi_n^h:=\vector{\xi_n^h}{\eta_n^h}=E_n\alpha_n^h=\alpha_n^he^{inx},\ \ \Phi_n^p:=\vector{\xi_n^p}{\eta_n^p}=E_n\alpha_n^p=\alpha_n^pe^{inx},
\end{equation*}
for all $\blue n\in\mb{Z}^*$ and $x\in(0,2\pi)$. {\blue This proves parts (iii) and (v).}

{\blue
\item[Part-(iv).] Follows immediately from the expression of the eigenvalues $\nu_n^h$ and $\nu_n^p$, given by \eqref{exp_ev_h_b}-\eqref{exp_ev_p_b}. Indeed, we can write
\begin{align*}
\nu_n^h=-\frac{2b\bar{\rho}}{\mu_0+\sqrt{\mu_0^2-\frac{4b\bar{\rho}}{n^2}}}+\bar{u}in,\ \text{and}\ \ \nu_n^p=-\frac{n^2}{2}\left(\mu_0+\sqrt{\mu_0^2-\frac{4b\bar{\rho}}{n^2}}\right)+\bar{u}in
\end{align*}
for $n\in\mb{Z}^*$.

%\textit{Part-(vi).} Denote $\Phi_0^h:=\vector{1}{1}$, $\tilde{\Phi}_0^h:=\vector{1}{-1}$ and $\Psi_n(x):=\vector{\bar{\rho}}{0}e^{inx}$, $\tilde{\Psi}_n(x):=\vector{0}{1}e^{inx}$ for $n\in\mb{Z}$. Then, the set of eigenfunctions $\{\Phi_n^h,\ \Phi_n^p,\ ;\ n\in\mb{Z}\}$ of $A^*$ is quadratically close to $\left\{\Psi_n,\ \tilde{\Psi}_n;\ n\in\mb{Z}\right\}$ in $(L^2(0,2\pi))^2$. Indeed, we have for a large $N\in\mb{N}$
%\begin{equation*}
%\sum_{\mod{n}>N}\left(\norm{\Phi_n^h-\Psi_n}_{(L^2(0,2\pi))^2}^2+\norm{\Phi_n^p-\tilde{\Psi}_n}_{(L^2(0,2\pi))^2}^2\right)\leq C\sum_{\mod{n}>N}\frac{1}{\mod{n}^2}<\infty,
%\end{equation*}
%thanks to the fact that $\mod{\nu_2^n-\bar{u}}\leq\frac{C}{\mod{n}}$ and $\mod{\nu_1^n-\bar{u}}\geq C\mod{n}$ for large $n$. Using a result of Bao-Zhu Guo \cite{Guo01}
}
\end{enumerate}
\end{proof}

{\blue
	
\smallskip

From the expression of the eigenvalues given by \eqref{exp_ev_h_b}-\eqref{exp_ev_p_b}, we can further deduce several important properties, which are given by the following Lemma:
\begin{lemma}[Properties of the eigenvalues]\label{prop_ev_b}
Let $n,l\in\mb{Z}^*$. Then,
\begin{enumerate}[(i)]
\item $\nu_n^h=\nu_l^h$ if and only if $n=l$.
\item If $n_1:=\frac{2\sqrt{b\bar{\rho}-\bar{u}^2}}{\mu_0}\in\mb{N}$, then $\nu_{n_1}^p=\nu_{-n_1}^p$ and $\nu_n^p\neq\nu_l^p$ for remaining $n,l\in\mb{Z}^*$ with $n\neq l$.
\item If $n_0:=\frac{2\sqrt{b\bar{\rho}}}{\mu_0}\in\mb{N}$, then $\nu_{j}^h=\nu_j^p$ for $j=\pm n_0$ and $\nu_n^h\neq\nu_l^p$ for all $n,l\in\mb{Z}^*\setminus\{\pm n_0\}$.
\item If $\frac{2\sqrt{b\bar{\rho}}}{\mu_0},\frac{2\sqrt{b\bar{\rho}-\bar{u}^2}}{\mu_0}\notin\mb{N}$, then all the eigenvalues of $A^*$ are simple.
\end{enumerate}
\end{lemma}

We give a detailed proof of this Lemma in Appendix \ref{app_prop_ev_b} to simplify the presentation of the article. From this Lemma, we note that when $n_0=\frac{2\sqrt{b\bar{\rho}}}{\mu_0}\in\mb{N}$, the matrix $R_j$ admits an eigenvalue $\nu_j:=-\frac{\mu_0j^2}{2}+i\bar{u}j$ of multiplicity $2$ with the eigenvectors $\alpha_j:=\vector{\bar{\rho}}{\nu_2^{j}-\bar{u}}$ for $j=\pm n_0$. 
%In this case we will multiply the eigenvector by $\nu_j$ and denote the new eigenvector as $\alpha_j:=\vector{\nu_j\bar{\rho}}{\nu_j(\nu_2^j-\bar{u})}$ (the reason for this will be explained later, more precisely in Remark \ref{rem_obs_gen_b}). 
Let $\tilde{\alpha}_{j}=(\tilde{\alpha}_1^j,\tilde{\alpha}_2^j)$ be the generalized eigenvector corresponding to $\nu_{j}$ for $j=\pm n_0$, then we have the following set of relations:
\begin{align}\label{gen_ev_rn_b}
\begin{cases}
(i\bar{u}j-\nu_j)\tilde{\alpha}_1^j+i\bar{\rho}j\tilde{\alpha}_2^j=\bar{\rho},\\
ibj\tilde{\alpha}_1^j+(-\mu_0j^2+i\bar{u}j-\nu_j)\tilde{\alpha}_2^j=\nu_2^j-\bar{u},
\end{cases}
\end{align}
for $j=\pm n_0$. Thus, if $\frac{2\sqrt{b\bar{\rho}}}{\mu_0}\in\mb{N}$, the operator $A^*$ admits generalized eigenfunction corresponding to the eigenvalue $\nu_j^h=\nu_j^p=\nu_j$ for $j=\pm n_0$. We denote the generalized eigenfunction corresponding to $\nu_j$ by $\tilde{\Phi}_j:=\tilde{\alpha}_je^{ijx}$ for $j=\pm n_0$. Also, recall that the set of eigenfunctions corresponding to the eigenvalue $\nu_0=0$ is $\left\{\Phi_0:=\vector{1}{1},\ \tilde{\Phi}_0:=\vector{1}{-1}\right\}$. Then, with the above mentioned properties of the eigenvalues, we can prove that the set of (generalized) eigenfunctions of $A^*$ form a Riesz basis of $(L^2(0,2\pi))^2$.
\begin{proposition}\label{prop_rb_b}
If $\frac{2\sqrt{b\bar{\rho}}}{\mu_0}\in\mb{N}$, the set of (generalized) eigenfunctions $$\mc{E}(A^*):=\left\{\Phi_n^h,\ \Phi_n^p\ :\ n\in\mb{Z}^*\setminus\{\pm n_0\};\ \Phi_j,\ \tilde{\Phi}_j\ :\ j=0,\pm n_0\right\}$$ form a Riesz basis in $(L^2(0,2\pi))^2$. In particular, when $\frac{2\sqrt{b\bar{\rho}}}{\mu_0}\notin\mb{N}$, the set of eigenfunctions $$\left\{\Phi_n^h,\ \Phi_n^p\ :\ n\in\mb{Z}^*\right\}$$ of $A^*$ form a Riesz basis in $(\dot{L}^2(0,2\pi))^2$.
\end{proposition}
\begin{proof}
Denote $\Psi_n(x):=\vector{\bar{\rho}}{0}e^{inx}$, $\tilde{\Psi}_n(x):=\vector{0}{1}e^{inx}$ for $n\in\mb{Z}$. Then, the set of generalized eigenfunctions $\left\{\Phi_n^h,\ \Phi_n^p\ :\ n\in\mb{Z}^*\setminus\{\pm n_0\};\ \Phi_j,\ \tilde{\Phi}_j\ :\ j=0,\pm n_0\right\}$ of $A^*$ is quadratically close to the orthogonal basis $\left\{\Psi_n,\ \tilde{\Psi}_n;\ n\in\mb{Z}\right\}$ in $(L^2(0,2\pi))^2$. Indeed, we have for a large $N\in\mb{N}$
\begin{equation*}
\sum_{\mod{n}>N}\left(\norm{\Phi_n^h-\Psi_n}_{(L^2(0,2\pi))^2}^2+\norm{\Phi_n^p-\tilde{\Psi}_n}_{(L^2(0,2\pi))^2}^2\right)\leq C\sum_{\mod{n}>N}\frac{1}{\mod{n}^2}<\infty,
\end{equation*}
thanks to the fact that $\mod{\nu_2^n-\bar{u}}\leq\frac{C}{\mod{n}}$ and $\mod{\nu_1^n-\bar{u}}\geq C\mod{n}$ for large $n$. Since the set $\left\{\Psi_n,\tilde{\Psi}_n\ ;\ n\in\mb{Z}\right\}$ is an orthogonal basis of $(L^2(0,2\pi))^2$, this Proposition is now an immediate consequence of the result of Bao-Zhu Guo \cite[Theorem 6.3]{Guo01}.
\end{proof}
}

\begin{comment}
\begin{remark}
Note that, for all $\mod{n}$ large, all the eigenvalues of $A^*$ are simple. There may be multiple eigenvalues of $A^*$, depending on the constants $\bar{\rho},\bar{u},\mu_0,b$, but that would be only finitely many. Thus, without loss of generality, we can assume that $A^*$ has simple eigenvalues. Indeed, for the case of finite number of multiple eigenvalues, one can adapt a similar approach as in Section 4.2 of \cite{Chowdhury14} (by considering suitable generalized eigenfunctions) to prove the required observability inequality. One can also see \cite[Remarks, page 178]{Komornik05} for a version of Ingham inequality in the case of repeated eigenvalues.
\end{remark}
\end{comment}
\subsection{Observation estimates}\label{sec_obs_est_b}
As mentioned in the introduction, we need to prove certain observability inequalities to achieve null controllability of the system \eqref{lcnse_b} and to do so, we need lower bound estimates of the corresponding observation terms (when control acting in density or velocity). {\blue Looking at the Definition \ref{def_sol_b} of the solution to \eqref{lcnse_b} (in the sense of transposition), let us} first define the observation operators associated to the system \eqref{lcnse_b} as follows:
\begin{itemize}
\item The observation operator $\mc{B}^*_{\rho}:\mc{D}(A^*)\to\mb{C}$ to the system \eqref{lcnse_b}-\eqref{in_cd_b}-\eqref{bd_cd_b1} is defined by
\begin{equation}\label{obs_op_den_b}
\mc{B}^*_{\rho}\Phi:=\bar{u}\xi(2\pi)+\bar{\rho}\eta(2\pi),\ \ \text{for}\ \Phi:=(\xi,\eta)\in\mc{D}(A^*).
\end{equation}
\item The observation operator $\mc{B}^*_u:\mc{D}(A^*)\to\mb{C}$ to the system \eqref{lcnse_b}-\eqref{in_cd_b}-\eqref{bd_cd_b2} is defined by
\begin{equation}\label{obs_op_vel_b}
\mc{B}^*_u\Phi:=b\xi(2\pi)+\bar{u}\eta(2\pi)+\mu_0\eta_x(2\pi),\ \ \text{for}\ \Phi:=(\xi,\eta)\in\mc{D}(A^*).
\end{equation}
\end{itemize}
Recall that $\mc{E}(A^*)$ denotes the set of all (generalized) eigenfunctions of $A^*$. The following result proves that these observation terms are non-zero for all $\Phi\in\mc{E}(A^*)\setminus\{\Phi_0,{\blue\tilde{\Phi}_j},\ j=0,\pm n_0\}$, and have positive lower bounds for all $n\in\mb{Z}^*$.
\begin{lemma}\label{lem_obs_est_b}
For all $\Phi_{\nu}\in\mc{E}(A^*)\setminus\{\Phi_0,{\blue\tilde{\Phi}_j},\ j=0,\pm n_0\}$, the observation operators satisfy $\mc{B}_{\rho}^*\Phi_{\nu}\neq0$ and $\mc{B}_u^*\Phi_{\nu}\neq0$. Moreover, we have the following estimates
\begin{align}
&\mod{\mc{B}_{\rho}^*\Phi_n^h}\geq C,\ \ \mod{\mc{B}_{\rho}^*\Phi_n^p}\geq C,\label{obs_est_d_b}\\
&\mod{\mc{B}_u^*\Phi_n^h}\geq \frac{C}{\mod{n}},\ \ \mod{\mc{B}_u^*\Phi_n^p}\geq C\mod{n},\label{obs_est_v_b}
\end{align}
for some $C>0$ and all $n\in\mb{Z}^*$.
\end{lemma}
\begin{proof}
{\blue
Recall from the proof of Lemma \ref{lemma_eigen_b} that eigenvectors $(\alpha_1^n,\alpha_2^n)^{\dagger}$ and $(\beta_1^n,\beta_2^n)^{\dagger}$ of the matrix $R_n$ satisfies the following equations:
\begin{align}
(\bar{u}in-\nu_n^h)\alpha_1^n+\bar{\rho}in\alpha_2^n=0,\ \
bin\alpha_1^n+(-\mu_0n^2+\bar{u}in-\nu_n^h)\alpha_2^n=0,\label{eqn_ev_rn_b_1}\\
(\bar{u}in-\nu_n^p)\beta_1^n+\bar{\rho}in\beta_2^n=0,\ \
bin\beta_1^n+(-\mu_0n^2+\bar{u}in-\nu_n^p)\beta_2^n=0,\label{eqn_ev_rn_b_2}
\end{align}
for $n\in\mb{Z}$. Also, recall the expressions of $\nu_1^n=\frac{1}{in}\nu_n^p$ and $\nu_2^n=\frac{1}{in}\nu_n^h$. We will use these equation to conclude the proof of this result. Note that}
\begin{align*}
\mc{B}^*_{\rho}\Phi_n^h=\bar{u}\xi_n^h(2\pi)+\bar{\rho}\eta_n^h(2\pi)=\bar{u}\alpha_1^n+\bar{\rho}\alpha_2^n=\nu_2^n\alpha_1^n\neq0,\\
\mc{B}^*_{\rho}\Phi_n^p=\bar{u}\xi_n^p(2\pi)+\bar{\rho}\eta_n^p(2\pi)=\bar{u}\beta_1^n+\bar{\rho}\beta_2^n=\nu_1^n\beta_1^n\neq0,
\end{align*}
for all $n\in\mb{Z}^*$, {\blue thanks to the first equations of \eqref{eqn_ev_rn_b_1}-\eqref{eqn_ev_rn_b_2}. The estimates on $\mc{B}^*_{\rho}\Phi_n^h$ and $\mc{B}^*_{\rho}\Phi_n^p$ follows directly from the above expressions.}

\smallskip

{\blue
\noindent For the parabolic frequencies, we have}
\begin{align*}
\mc{B}_u^*\Phi_n^h&=b\xi_n^h(2\pi)+\bar{u}\eta_n^h(2\pi)+\mu_0(\eta_n^h)_x(2\pi)=b\alpha_1^n+(\bar{u}+\mu_0in)\alpha_2^n=\nu_2^n\alpha_2^n\neq0,\\
\mc{B}_u^*\Phi_n^p&=b\xi_n^p(2\pi)+\bar{u}\eta_n^p(2\pi)+\mu_0(\eta_n^p)_x(2\pi)=b\beta_1^n+(\bar{u}+\mu_0in)\beta_2^n=\nu_1^n\beta_2^n\neq0,
\end{align*}
for all $n\in\mb{Z}^*$, {\blue thanks to the second equations in \eqref{eqn_ev_rn_b_1}-\eqref{eqn_ev_rn_b_2}. Since $\mod{\alpha_2^n}\geq\frac{C}{\mod{n}}$ and $\nu_2^n$ is bounded (away from zero) for all $n\in\mb{Z}^*$, the estimate on $\mc{B}^*_u\Phi_n^h$ and $\mc{B}^*\Phi_n^p$ follows directly from the above expressions.
}
\end{proof}

{\blue 
\begin{remark}\label{rem_obs_gen_b}
For the generalized eigenfunction $\tilde{\Phi}_j\in\mc{E}(A^*)$ ($j=\pm n_0$), we can choose $\tilde{\alpha}_1^j$ and $\tilde{\alpha}_2^j$ accordingly so that $\mc{B}^*_{\rho}\tilde{\Phi}_j=\bar{u}\tilde{\alpha}_1^j+\bar{\rho}\tilde{\alpha}_2^j\neq0$ and $\mc{B}^*_u\tilde{\Phi}_j=b\tilde{\alpha}_1^j+(\bar{u}+\mu_0ij)\tilde{\alpha}_2^j\neq0$ for $j=\pm n_0$.

%we have
%\begin{align*}
%\mc{B}^*_{\rho}\tilde{\Phi}_j&=\bar{u}\tilde{\alpha}_1^j+\bar{\rho}\tilde{\alpha}_2^j=\frac{\nu_j\tilde{\alpha}_1^j+\nu_j\bar{\rho}}{ij},\\
%\mc{B}^*_u\tilde{\Phi}_j&=b\tilde{\alpha}_1^j+(\bar{u}+\mu_0ij)\tilde{\alpha}_2^j=\frac{\nu_j\tilde{\alpha}_2^j+\nu_j(\nu_2^j-\bar{u})}{ij},
%\end{align*}
%for $j=\pm n_0$; thanks to the set of equations \eqref{gen_ev_rn_b}. 
\end{remark}
}

\subsection{Observability inequalities}

{\blue In this section, we prove our main null controllability results of the system \eqref{lcnse_b}, namely Theorem \ref{thm_den_b} and Theorem \ref{thm_vel_b}. We first state two results which are equivalent to null controllability of the system \eqref{lcnse_b} using controls acting in density and velocity respectively. The proofs are standard (see for instance \cite[Section 2.3.4]{Micu04},\cite[Section 4.3]{Zuazua07}; see also our paper \cite{Bhandari22}), so we skip the details.
\begin{theorem}\label{equiv_thm_den_b}
Let $T>0$ be given. Then, the system \eqref{lcnse_b}-\eqref{in_cd_b}-\eqref{bd_cd_b1} is null controllable at time $T$ in the space $(\dot{L}^2(0,2\pi))^2$ if and only if the inequality
\begin{equation}\label{obs_inq_b_den}
\norm{(\sigma(0),v(0))^{\dagger}}_{(\dot{L}^2(0,2\pi))^2}^2\leq C\int_{0}^{T}\mod{\bar{u}\sigma(t,2\pi)+\bar{\rho}v(t,2\pi)}^2dt
\end{equation}
holds for all solutions $(\sigma,v)^{\dagger}$ of the adjoint system \eqref{lcnse_adj_b} with terminal data $(\sigma_T,v_T)^{\dagger}\in\mc{D}(A^*)$.
\end{theorem}
\begin{theorem}\label{equiv_thm_vel_b}
Let $T>0$ be given. Then, the system \eqref{lcnse_b}-\eqref{in_cd_b}-\eqref{bd_cd_b2} is null controllable at time $T$ in the space $\dot{H}^1_{\per}(0,2\pi)\times\dot{L}^2(0,2\pi)$ if and only if the inequality
\begin{equation}\label{obs_inq_b_vel}
\norm{(\sigma(0),v(0))^{\dagger}}_{\dot{H}^{-1}_{\per}(0,2\pi)\times\dot{L}^2(0,2\pi)}^2\leq C\int_{0}^{T}\mod{b\sigma(t,2\pi)+\bar{u}v(t,2\pi)+\mu_0v_x(t,2\pi)}^2dt
\end{equation}
holds for all solutions $(\sigma,v)^{\dagger}$ of the adjoint system \eqref{lcnse_adj_b} with terminal data $(\sigma_T,v_T)^{\dagger}\in\mc{D}(A^*)$.
\end{theorem}
The inequalities \eqref{obs_inq_b_den} and \eqref{obs_inq_b_vel} are referred as observability inequalities for the systems \eqref{lcnse_b}-\eqref{in_cd_b}-\eqref{bd_cd_b1} and \eqref{lcnse_b}-\eqref{in_cd_b}-\eqref{bd_cd_b2} respectively. To prove these inequalities, we will use the Ingham-type inequality \eqref{ingham-ineq} to obtain a lower bound of the observation terms (given in the right hand sides of \eqref{obs_inq_b_den} and \eqref{obs_inq_b_vel}) together with the upper bounds of norms of $(\sigma(0),v(0))^{\dagger}$ in the respective spaces. 

\smallskip

Let $\frac{2\sqrt{b\bar{\rho}-\bar{u}^2}}{\mu_0}\notin\mb{N}$. We first assume that $\frac{2\sqrt{b\bar{\rho}}}{\mu_0}\notin\mb{N}$, that is, all the eigenvalues of $A^*$ are simple (Lemma \ref{prop_ev_b}-(iv)) and prove null controllability of the system \eqref{lcnse_b} (Theorem \ref{thm_den_b}-Part(i) and Theorem \ref{thm_vel_b}-Part(i)). In the case of multiple eigenvalues (when $\frac{2\sqrt{b\bar{\rho}}}{\mu_0}\in\mb{N}$), we give detailed proof of Theorem \ref{thm_den_b}-Part(i) at the end of this section. The proof of Theorem \ref{thm_vel_b}-Part(i) in the presence of multiple eigenvalues will be similar to that of Theorem \ref{thm_den_b}-Part(i) and so we give some comments at the end of this section.

\smallskip

\subsubsection{The case of simple eigenvalues}
Let us assume that $\frac{2\sqrt{b\bar{\rho}-\bar{u}^2}}{\mu_0},\frac{2\sqrt{b\bar{\rho}}}{\mu_0}\notin\mb{N}$ and let $(\sigma_T,v_T)^{\dagger}\in(\dot{L}^2(0,2\pi))^2$.} Since the set of eigenfunctions {\blue $\{\Phi_n^h,\Phi_n^p\ ; \ n\in\mb{Z}^*\}$ form a Riesz basis in $(\dot{L}^2(0,2\pi))^2$ (thanks to Proposition \ref{prop_rb_b})}, therefore any $(\sigma_T,v_T)^{\dagger}\in (\dot{L}^2(0,2\pi))^2$ can be written as
\begin{equation*}
(\sigma_T,v_T)^{\dagger}=\sum_{n\in\mb{Z}^*}\left(a_n^h\Phi_n^h+a_n^p\Phi_n^p\right),
\end{equation*}
for some $(a_n^h)_{n\in\mb{Z}^*},(a_n^p)_{n\in\mb{Z}^*}\in\ell_2$.
Then the solution to the adjoint system \eqref{lcnse_adj_b} is
\begin{equation*}
(\sigma(t,x),v(t,x))^{\dagger}=\sum_{n\in\mb{Z}^*}a_n^he^{\nu_n^h(T-t)}\Phi_n^h+\sum_{n\in\mb{Z}^*}a_n^pe^{\nu_n^p(T-t)}\Phi_n^p,
\end{equation*}
for $(t,x)\in (0,T)\times(0,2\pi)$. Thus we get
\begin{equation*}
\sigma(t,x)=\bar{\rho}\sum_{n\in\mb{Z}^*}a_n^he^{\nu_n^h(T-t)}e^{inx}+\sum_{n\in\mb{Z}^*}a_n^pe^{\nu_n^p(T-t)}\frac{\bar{\rho}}{\nu_1^n-\bar{u}}e^{inx},
\end{equation*}
and
\begin{equation*}
v(t,x)=\sum_{n\in\mb{Z}^*}a_n^he^{\nu_n^h(T-t)}(\nu_2^n-\bar{u})e^{inx}+\sum_{n\in\mb{Z}^*}a_n^pe^{\nu_n^p(T-t)}e^{inx},
\end{equation*}
for all $(t,x)\in (0,T)\times(0,2\pi)$. %To prove the observability inequality, we need an upper bound of the norm of $(\sigma(0),v(0))^{\dagger}$ and a lower bound estimate for the respective observation terms. We first estimate the upper bounds of the norm of $(\sigma(0),v(0))^{\dagger}$. We have

\smallskip

\noindent\underline{\blue Estimates on the norms of $(\sigma(0),v(0))^{\dagger}$}:
We have
\begin{align}\label{est_norm_b1}
\norm{(\sigma(0),v(0))^{\dagger}}_{(\dot{L}^2(0,2\pi))^2}^2&\leq C\left[\sum_{n\in\mb{Z}^*}\mod{a_n^h}^2e^{2\Re(\nu_n^h)T}\norm{\Phi_n^h}_{(\dot{L}^2(0,2\pi))^2}^2+\sum_{n\in\mb{Z}^*}\mod{a_n^p}^2e^{2\Re(\nu_n^p)T}\norm{\Phi_n^p}_{(\dot{L}^2(0,2\pi))^2}^2\right]\\
&\leq C\left[\sum_{n\in\mb{Z}^*}\mod{a_n^h}^2\left(1+\mod{\nu_2^n-\bar{u}}^2\right)e^{2\Re(\nu_n^h)T}\norm{e^{inx}}_{\dot{L}^2(0,2\pi)}^2\right.\notag\\
&\left.\quad\quad\quad\quad+\sum_{n\in\mb{Z}^*}\mod{a_n^p}^2\left(\frac{1}{\mod{\nu_1^n-\bar{u}}^2}+1\right)e^{2\Re(\nu_n^p)T}\norm{e^{inx}}_{\dot{L}^2(0,2\pi)}^2\right]\notag\\
&\leq C\left[\sum_{n\in\mb{Z}^*}\mod{a_n^h}^2+\sum_{n\in\mb{Z}^*}\mod{a_n^p}^2e^{2\Re(\nu_n^p)T}\right]\notag,
\end{align}
since the sequences $1+\mod{\nu_2^n-\bar{u}}^2$ and $1+\frac{1}{\mod{\nu_1^n-\bar{u}}^2}$ are bounded for all $n\in\mb{Z}^*$. We also have
\begin{align}\label{est_norm_b2}
&\norm{(\sigma(0),v(0))^{\dagger}}_{{\blue\dot{H}^{-1}_{\per}(0,2\pi)}\times \dot{L}^2(0,2\pi)}^2\\
&\leq C\left[\sum_{n\in\mb{Z}^*}\mod{a_n^h}^2e^{2\Re(\nu_n^h)T}\norm{\Phi_n^h}_{{\blue\dot{H}^{-1}_{\per}(0,2\pi)}\times \dot{L}^2(0,2\pi)}^2+\sum_{n\in\mb{Z}^*}\mod{a_n^p}^2e^{2\Re(\nu_n^p)T}\norm{\Phi_n^p}_{{\blue\dot{H}^{-1}_{\per}(0,2\pi)}\times \dot{L}^2(0,2\pi)}^2\right]\notag\\
%&\leq C\left[\sum_{n\in\mb{Z}^*}\mod{a_n^h}^2\left(\norm{e^{inx}}_{{\blue\dot{H}^{-1}_{\per}(0,2\pi)}}^2+\mod{\nu_2^n-\bar{u}}^2\norm{e^{inx}}_{\dot{L}^2(0,2\pi)}^2\right)\right.\\
%&\quad\quad\quad\left.+\sum_{n\in\mb{Z}^*}\mod{a_n^p}^2\left(\frac{1}{\mod{\nu_1^n-\bar{u}}^2}+1\right)e^{2\Re(\nu_n^p)T}\norm{e^{inx}}_{\dot{L}^2(0,2\pi)}^2\right]\notag\\
&\leq C\left[\sum_{n\in\mb{Z}^*}\mod{a_n^h}^2{\blue\frac{1}{\mod{n}^{2}}}+\sum_{n\in\mb{Z}^*}\mod{a_n^p}^2e^{2\Re(\nu_n^p)T}\right],\notag
\end{align}
since the sequences $\nu_2^n-\bar{u},\frac{1}{\nu_1^n-\bar{u}}\sim_{+\infty}\frac{1}{n}$.
We now find the lower bounds of the respective observation terms and prove our main results for the barotropic case. We use the Ingham-type inequality (Lemma \ref{lem_ingham}) to obtain these bounds. {\blue First, we show that the eigenvalues $(\nu_n^h)_{n\in\mb{Z}^*}$ and $(\nu_n^p)_{n\in\mb{Z}^*}$ satisfy all the hypotheses of Lemma \ref{lem_ingham}. Recall the set of eigenvalues $(\nu_n^h)_{n\in\mb{Z}^*}$ and $(\nu_n^p)_{n\in\mb{Z}^*}$ of the operator $A^*$:
\begin{align*}
\nu_n^h&=-\frac{1}{2}\left(\mu_0n^2-\sqrt{\mu_0^2n^4-4b\bar{\rho}n^2}-2\bar{u}in\right),\\
\nu_n^{p}&=-\frac{1}{2}\left(\mu_0n^2+\sqrt{\mu_0^2n^4-4b\bar{\rho}n^2}-2\bar{u}in\right),
\end{align*}
for $n\in\mb{Z}^*$.

\begin{itemize}
	
\item Due to the assumption on the coefficients ($\frac{2\sqrt{b\bar{\rho}}}{\mu_0},\frac{2\sqrt{b\bar{\rho}-\bar{u}^2}}{\mu_0}\notin\mb{N}$), we have $\nu_n^h\neq\nu_l^h$, $\nu_n^p\neq\nu_l^p$ for all $n,l\in\mb{Z}^*$ with $n\neq l$ and the families are disjoint, that is, $\{\nu_n^h,n\in\mb{Z}^*\}\cap\{\nu_n^p,n\in\mb{Z}^*\}=\emptyset$ (Lemma \ref{prop_ev_b}).

\item We now rewrite $\nu_n^h$ as
\begin{equation*}
\nu_n^h=-\omega_0+\bar{u}in-\omega_0\frac{\mu_0n^2-\sqrt{\mu_0^2n^4-4b\bar{\rho}n^2}}{\mu_0n^2+\sqrt{\mu_0^2n^4-4b\bar{\rho}n^2}},\ \ \mod{n}\geq n_0.
\end{equation*}
This shows that the family $(\nu_n^h)_{n\in\mb{Z}^*}$ satisfies hypothesis (H2) of Lemma \ref{lem_ingham} with $\beta=-\omega_0, \tau=\bar{u}$ and $e_n=-\omega_0\frac{\mu_0n^2-\sqrt{\mu_0^2n^4-4b\bar{\rho}n^2}}{\mu_0n^2+\sqrt{\mu_0^2n^4-4b\bar{\rho}n^2}}$ for $\mod{n}\geq n_0$. Note that $\mod{e_n}\leq\frac{C}{\mod{n}^2}$ and therefore $(e_n)_{\mod{n}\geq n_0}\in\ell_2$. 

\item On the other hand, we have for all $\mod{n}\geq n_0$
\begin{align*}
\frac{-\Re(\nu_n^p)}{\mod{\Im(\nu_n^p)}}=\frac{1}{2}\frac{\mu_0n^2+\sqrt{\mu_0^2n^4-4b\bar{\rho}n^2}}{\bar{u}n}\geq\frac{\mu_0}{2\bar{u}},
\end{align*}
which verifies hypothesis (P2) of Lemma \ref{lem_ingham}. 

\item We now compute for $\mod{n},\mod{l}\geq n_0$ with $n\neq l$
\begin{align*}
\mod{\nu_n^p-\nu_l^p}^2%&=\frac{1}{4}\left(\mu_0(n^2-l^2)+\sqrt{\mu_0^2n^4-4b\bar{\rho}n^2}-\sqrt{\mu_0^2l^4-4b\bar{\rho}l^2}\right)^2+\bar{u}^2(n-l)^2\\
&\geq\frac{1}{4}\left(\mu_0(n^2-l^2)+\mu_0n^2\sqrt{1-\frac{4b\bar{\rho}}{\mu_0^2n^2}}-\mu_0l^2\sqrt{1-\frac{4b\bar{\rho}}{\mu_0^2l^2}}\right)^2.
\end{align*}
Let $\mod{n}>\mod{l}$. Then, we have $\mu_0n^2\sqrt{1-\frac{4b\bar{\rho}}{\mu_0^2n^2}}>\mu_0l^2\sqrt{1-\frac{4b\bar{\rho}}{\mu_0^2l^2}}$ and this implies
\begin{equation*}
\mod{\nu_n^p-\nu_l^p}^2\geq\frac{\mu_0^2}{4}(n^2-l^2)^2,\ \ \text{implies}\ \ \mod{\nu_n^p-\nu_l^p}\geq\frac{\mu_0}{2}(n^2-l^2).
\end{equation*}
%that is,
%\begin{equation*}
%\mod{\nu_n^p-\nu_l^p}\geq\frac{\mu_0}{2}(n^2-l^2).
%\end{equation*}
We similarly have for $\mod{n}<\mod{l}$
\begin{equation*}
\mod{\nu_n^p-\nu_l^p}\geq\frac{\mu_0}{2}(l^2-n^2).
\end{equation*}
This proves that $(\nu_n^p)_{\mod{n}\geq n_0}$ satisfies hypothesis (P3) of Lemma \ref{lem_ingham} with $r=2$ and $\delta=\frac{\mu_0}{2}$. 

\item Finally, we have for $\mod{n}\geq n_0$
\begin{align*}
\mod{\nu_n^p}^2%&=\frac{1}{4}\left(\mu_0n^2+\sqrt{\mu_0^2n^4-4b\bar{\rho}n^2}\right)^2+\bar{u}^2n^2\\
&=\frac{\mu_0^2}{4}n^4\left(1+\sqrt{1-\frac{4b\bar{\rho}}{\mu_0^2n^2}}\right)^2+\bar{u}^2n^2,
\end{align*}
and therefore
\begin{align*}
\frac{\mu_0^2}{4}n^4\leq\mod{\nu_n^p}^2\leq\frac{\mu_0^2}{2}n^4,\ \ \forall\mod{n}\geq n_0.
\end{align*}
This proves that the family $(\nu_n^p)_{\mod{n}\geq n_0}$ satisfies hypothesis (P4) of Lemma \ref{lem_ingham} with $\epsilon=\frac{1}{\sqrt{2}}$, $A_0=0$ and $B_0=\frac{\mu_0}{\sqrt{2}}>\delta$.
\end{itemize}

\bigskip

We are now ready to prove the null controllability results of the system \eqref{lcnse_b} in the case of simple eigenvalues.

}
%Note that, the eigenvalues $(\nu_n^h)_{n\in\mb{Z}^*}$ satisfies hypotheses (H1)-(H2) with $\tau=\bar{u}, \beta=-\omega_0$ and $(\nu_n^p)_{n\in\mb{Z}^*}$ satisfies hypotheses (P1)-(P4) with $r=2$.

\bigskip

\noindent\underline{Proof of Theorem \ref{thm_den_b}-Part \eqref{null_dens_b}:}
Let $T>\frac{2\pi}{\bar{u}}$. {\blue Thanks to Theorem \ref{equiv_thm_den_b}, it is enough to prove the observability inequality \eqref{obs_inq_b_den}, that is,}
\begin{equation*}%\label{obs_inq_den_b}
\int_{0}^{T}\mod{\bar{u}\sigma(t,2\pi)+\bar{\rho}v(t,2\pi)}^2dt\geq C\norm{(\sigma(0),v(0))^{\dagger}}_{(\dot{L}^2(0,2\pi))^2}^2,
\end{equation*}
for all $(\sigma_T,v_T)^{\dagger}\in\mc{D}(A^*)$. {\blue Recall the operator $\mc{B}^*_{\rho}$ given by \eqref{obs_op_den_b}. Then, we can write the observation term as}
\begin{align*}
\blue \int_{0}^{T}\mod{\bar{u}\sigma(t,2\pi)+\bar{\rho}v(t,2\pi)}^2dt&\blue=\int_{0}^{T}\mod{\sum_{n\in\mb{Z^*}}a_n^he^{\nu_n^h(T-t)}\mc{B}_{\rho}^*\Phi_n^h+\sum_{n\in\mb{Z}^*}a_n^pe^{\nu_n^p(T-t)}\mc{B}_{\rho}^*\Phi_n^p}^2dt.
\end{align*}
Using the combined parabolic-hyperbolic Ingham type inequality \eqref{ingham-ineq} (Lemma \ref{lem_ingham}) and the observation estimates \eqref{obs_est_d_b}, we obtain
\begin{align*}
\blue \int_{0}^{T}\mod{\bar{u}\sigma(t,2\pi)+\bar{\rho}v(t,2\pi)}^2dt&\geq C\left[\sum_{n\in\mb{Z^*}}\mod{a_n^h}^2e^{2\Re(\nu_n^h)T}\mod{\mc{B}_{\rho}^*\Phi_n^h}^2+\sum_{n\in\mb{Z}^*}\mod{a_n^p}^2e^{2\Re(\nu_n^p)T}\mod{\mc{B}_{\rho}^*\Phi_n^p}^2\right]\\
&\geq C\left[\sum_{n\in\mb{Z^*}}\mod{a_n^h}^2+\sum_{n\in\mb{Z^*}}\mod{a_n^p}^2e^{2\Re(\nu_n^p)T}\right]
\end{align*}
This estimate together with the norm estimate \eqref{est_norm_b1}, the observability inequality \eqref{obs_inq_b_den} {\blue follows}. This completes the proof {\blue in the case of simple eigenvalues}.\qed

\smallskip

\noindent\underline{Proof of Theorem \ref{thm_vel_b}-Part \eqref{null_vel_b}:}
Let $T>\frac{2\pi}{\bar{u}}$. {\blue Similar to the density case, it is enough to prove the observability inequality \eqref{obs_inq_b_vel}, that is,}
\begin{equation*}%\label{obs_inq_vel_b}
\int_{0}^{T}\mod{b\sigma(t,2\pi)+\bar{u}v(t,2\pi)+\mu_0v_x(t,2\pi)}^2dt\geq C\norm{(\sigma(0),v(0))^{\dagger}}_{{\blue\dot{H}^{-1}_{\per}(0,2\pi)}\times \dot{L}^2(0,2\pi)}^2,
\end{equation*}
for all $(\sigma_T,v_T)^{\dagger}\in\mc{D}(A^*)$. We have
\begin{align*}
&\int_{0}^{T}\mod{b\sigma(t,2\pi)+\bar{u}v(t,2\pi)+\mu_0v_x(t,2\pi)}^2dt=\int_{0}^{T}\mod{\sum_{n\in\mb{Z^*}}a_{{\blue n}}^he^{\nu_n^h(T-t)}\mc{B}_u^*\Phi_n^h+\sum_{n\in\mb{Z^*}}a_n^pe^{\nu_n^p(T-t)}\mc{B}_u^*\Phi_n^p}^2dt,
\end{align*}
{\blue where $\mc{B}^*_u$ is defined in \eqref{obs_op_vel_b}.} Using the combined parabolic-hyperbolic Ingham type inequality \eqref{ingham-ineq} (Lemma \ref{lem_ingham}), we obtain
\begin{align*}
&\int_{0}^{T}\mod{b\sigma(t,2\pi)+\bar{u}v(t,2\pi)+\mu_0v_x(t,2\pi)}^2dt\\
&\geq C\left[\sum_{n\in\mb{Z^*}}\mod{a_n^h}^2e^{2\Re(\nu_n^h)T}\mod{\mc{B}^*_u\Phi_n^h}^2+\sum_{n\in\mb{Z^*}}\mod{a_n^p}^2e^{2\Re(\nu_n^p)T}\mod{\mc{B}^*_u\Phi_n^p}^2\right]\\
&\geq C\left[\sum_{n\in\mb{Z^*}}\mod{a_n^h}^2\frac{1}{\mod{n}^2}+\sum_{n\in\mb{Z^*}}\mod{a_n^p}^2\mod{n}^2e^{2\Re(\nu_n^p)T}\right],
\end{align*}
thanks to the estimate \eqref{obs_est_v_b}. {\blue Combining this estimate and \eqref{est_norm_b2}, we deduce that}
\begin{align*}
\blue\int_{0}^{T}\mod{b\sigma(t,2\pi)+\bar{u}v(t,2\pi)+\mu_0v_x(t,2\pi)}^2dt&\blue\geq C\norm{(\sigma(0),v(0))^{\dagger}}_{\dot{H}^{-1}_{\per}(0,2\pi)\times \dot{L}^2(0,2\pi)}.
\end{align*}
This proves the observability inequality \eqref{obs_inq_b_vel} {\blue and hence the proof is complete for simple eigenvalues.}\qed

{\blue
\subsubsection{The case of multiple eigenvalues}\label{sec_multi_ev_b}
In this section, we prove null controllability of the system \eqref{lcnse_b} in the presence of multiple eigenvalues. The proof will be similar in both cases (control acting in density or velocity), so we present a detailed proof for the density case and give brief details for the velocity case. The proof is inspired from \cite[Section 4.4]{Komornik05} and \cite[Section 4.2]{Chowdhury14} and throughout the proof, we assume the conditions $n_0=\frac{2\sqrt{b\bar{\rho}}}{\mu_0}\in\mb{N}$ and $\frac{2\sqrt{b\bar{\rho}-\bar{u}^2}}{\mu_0}\notin\mb{N}$. Then, we only have two multiple eigenvalues $\nu_{n_0}^h=\nu_{n_0}^p=:\nu_{n_0}$ and $\nu_{-n_0}^h=\nu_{-n_0}^p=:\nu_{-n_0}$ with the generalized eigenfunctions  $\left\{\Phi_{n_0},\tilde{\Phi}_{n_0}\right\}$ and $\left\{\Phi_{-n_0},\tilde{\Phi}_{-n_0}\right\}$ respectively, where $\Phi_j:=(\xi_j,\eta_j)^{\dagger}$ and $\tilde{\Phi}_j:=(\tilde{\xi}_j,\tilde{\eta}_j)^{\dagger}$ for $j=\pm n_0$. 

\bigskip

\noindent\underline{Control in density.} Let $(\sigma_T,v_T)^{\dagger}\in(\dot{L}^2(0,2\pi))^2$. We decompose it as
\begin{equation}
(\sigma_T,v_T)^{\dagger}=(\sigma_{T,1},v_{T,1})^{\dagger}+(\sigma_{T,2},v_{T,2})^{\dagger},
\end{equation}
where
\begin{equation*}
(\sigma_{T,1},v_{T,1})^{\dagger}=\sum_{j=\pm n_0}(a_j\Phi_j+\tilde{a}_j\tilde{\Phi}_j)
\end{equation*}
and
\begin{equation*}
(\sigma_{T,2},v_{T,2})^{\dagger}=\sum_{n\in\mb{Z}^*\setminus\{\pm n_0\}}(a_n^h\Phi_n^h+a_n^p\Phi_n^p).
\end{equation*}
Let $(\sigma_1,v_1)^{\dagger}$ and $(\sigma_2,v_2)^{\dagger}$ be the solutions of the adjoint system \eqref{lcnse_adj_b} with the terminal data $(\sigma_{T,1},v_{T,1})^{\dagger}$ and $(\sigma_{T,2},v_{T,2})^{\dagger}$ respectively. Then, we have
\begin{equation}
(\sigma_1,v_1)^{\dagger}=\sum_{j=\pm n_0}e^{\nu_j(T-t)}\left(a_j\Phi_j+(T-t)\tilde{a}_j\tilde{\Phi}_j\right)
\end{equation}
and
\begin{equation}\label{sigma_2_b}
(\sigma_2,v_2)^{\dagger}=\sum_{n\in\mb{Z}^*\setminus\{\pm n_0\}}\left(a_n^he^{\nu_n^h(T-t)}\Phi_n^h+a_n^pe^{\nu_n^p(T-t)}\Phi_n^p\right)
\end{equation}
In the expression of $(\sigma_2,v_2)^{\dagger}$, all eigenvalues are simple, so we have the following observability inequality
\begin{equation}
\int_{0}^{T}\mod{\bar{u}\sigma_2(t,2\pi)+\bar{\rho}v_2(t,2\pi)}^2dt\geq C\norm{(\sigma_2(0),v_2(0))^{\dagger}}_{(\dot{L}^2(0,2\pi))^2}^2.
\end{equation}
Note that $\bar{u}\sigma_1(t,2\pi)+\bar{\rho}v_1(t,2\pi)=\sum_{j=\pm n_0}e^{\nu_j(T-t)}\left(a_j\mc{B}^*_{\rho}\Phi_j+(T-t)\tilde{a}_j\mc{B}^*_{\rho}\tilde{\Phi}_j\right)$. We first add the term $e^{\nu_{n_0}(T-t)}\left(a_{n_0}\mc{B}^*_{\rho}\Phi_{n_0}+(T-t)\tilde{a}_{n_0}\mc{B}^*_{\rho}\tilde{\Phi}_{n_0}\right)$ in the above inequality.
Denote
\begin{align*}
\mc{Y}(t):=\bar{u}\sigma_2(t,2\pi)+\bar{\rho}v_2(t,2\pi)+e^{\nu_{n_0}(T-t)}\left(a_{n_0}\mc{B}^*_{\rho}\Phi_{n_0}+(T-t)\tilde{a}_{n_0}\mc{B}^*_{\rho}\tilde{\Phi}_{n_0}\right)
\end{align*}
and
\begin{align*}
\mc{Z}(t):=\mc{Y}(t)-\frac{1}{2\delta}\int_{-\delta}^{\delta}e^{\nu_{n_0}s}\mc{Y}(t+s)ds
\end{align*}
for $t\in(\delta,T-\delta)$ with $\delta>0$ (chosen later accordingly). Then, we have the following estimate (see \cite[Section 4.4]{Komornik05} for details).
\begin{equation}\label{inq_1}
\int_{\delta}^{T-\delta}\mod{\mc{Z}(t)}^2dt\leq C\int_{0}^{T}\mod{\mc{Y}(t)}^2dt.
\end{equation}
We now prove that
\begin{equation}
\int_{\delta}^{T-\delta}\mod{\mc{Z}(t)}^2dt\geq C\norm{(\sigma_2(0),v_2(0))^{\dagger}}_{(\dot{L}^2(0,2\pi))^2}^2.
\end{equation}
From the expression of $\mc{Y}(t)$, we can get
\begin{align*}
\mc{Z}(t)&=\sum_{n\in\mb{Z}^*\setminus\{\pm n_0\}}a_n^he^{\nu_n^h(T-t)}\mc{B}_{\rho}^*\Phi_n^h\left(1-\frac{\sinh((\nu_n^h-\nu_{n_0})\delta)}{(\nu_n^h-\nu_{n_0})\delta}\right)\\
&\quad\quad+\sum_{n\in\mb{Z}^*\setminus\{\pm n_0\}}a_n^pe^{\nu_n^p(T-t)}\mc{B}_{\rho}^*\Phi_n^p\left(1-\frac{\sinh((\nu_n^p-\nu_{n_0})\delta)}{(\nu_n^p-\nu_{n_0})\delta}\right)
\end{align*}
Since $\inf_{n\in\mb{Z}^*\setminus\{\pm n_0\}}\mod{\nu_n^h-\nu_{n_0}}>0$ and $\inf_{n\in\mb{Z}^*\setminus\{\pm n_0\}}\mod{\nu_n^p-\nu_{n_0}}>0$, we have (for appropriate $\delta>0$) $$\inf_{n\in\mb{Z}^*\setminus\{\pm n_0\}}\mod{1-\frac{\sinh((\nu_n^h-\nu_{n_0})\delta)}{(\nu_n^h-\nu_{n_0})\delta}}>0,\ \text{and}\ \ \inf_{n\in\mb{Z}^*\setminus\{\pm n_0\}}\mod{1-\frac{\sinh((\nu_n^p-\nu_{n_0})\delta)}{(\nu_n^p-\nu_{n_0})\delta}}>0.$$ Since $T>\frac{2\pi}{\bar{u}}$, we can choose $\delta$ small enough such that $T-2\delta>\frac{2\pi}{\bar{u}}$. Applying Ingham-type inequality \eqref{ingham-ineq} (for simple eigenvalues), we obtain
\begin{align*}
\int_{\delta}^{T-\delta}\mod{\mc{Z}(t)}^2dt&\geq C\left[\sum_{n\in\mb{Z^*}\setminus\{\pm n_0\}}\mod{a_n^h}^2e^{2\Re(\nu_n^h)T}\mod{\mc{B}_{\rho}^*\Phi_n^h}^2+\sum_{n\in\mb{Z}^*\setminus\{\pm n_0\}}\mod{a_n^p}^2e^{2\Re(\nu_n^p)T}\mod{\mc{B}_{\rho}^*\Phi_n^p}^2\right]\\
&\geq C\norm{(\sigma_2(0),v_2(0))^{\dagger}}_{(\dot{L}^2(0,2\pi))^2}^2.
\end{align*}
%Note that, since all eigenvalues in the above expression are simple, we can add finitely many terms in the above summation by using the same technique as in \cite{Micu}.
Therefore, using the estimate \eqref{inq_1}, we obtain
\begin{equation}\label{inq_2}
\int_{0}^{T}\mod{\mc{Y}(t)}^2dt\geq C\norm{(\sigma_2(0),v_2(0))^{\dagger}}_{(\dot{L}^2(0,2\pi))^2}^2.
\end{equation}
%In a similar fashion, we can add the term $A_{-n_0}+(T-t)\tilde{A}_{-n_0}$ and finally obtain
%\begin{equation}\label{inq_2}
%\int_{0}^{T}\mod{\bar{u}\sigma(t,2\pi)+\bar{\rho}v(t,2\pi)}^2dt\geq C\norm{(\sigma_2(0),v_2(0))^{\dagger}}_{(\dot{L}^2(0,2\pi))^2}^2
%\end{equation}
Since $T>\frac{2\pi}{\bar{u}}$, we can choose $\epsilon>0$ such that $T-\epsilon>\frac{2\pi}{\bar{u}}$. Therefore we can write
\begin{equation*}
\int_{\epsilon}^{T}\mod{\mc{Y}(t)}^2dt\geq C\norm{(\sigma_2(\epsilon),v_2(\epsilon))^{\dagger}}_{(\dot{L}^2(0,2\pi))^2}^2
\end{equation*}
and thus
\begin{equation}\label{inq_2a}
\int_{0}^{T}\mod{\mc{Y}(t)}^2dt\geq C\int_{\epsilon}^{T}\mod{\mc{Y}(t)}^2dt\geq C\norm{(\sigma_2(\epsilon),v_2(\epsilon))^{\dagger}}_{(\dot{L}^2(0,2\pi))^2}^2.
\end{equation}
Thanks to the well-posedness result (Lemma \ref{well_posed_adj_b}) of the adjoint system \eqref{lcnse_adj_b}, we have
\begin{equation}\label{well_posed_inq_den_b}
\int_{0}^{\epsilon}\mod{\bar{u}\sigma_2(t,2\pi)+\bar{\rho}v_2(t,2\pi)}^2dt\leq C\norm{(\sigma_2(\epsilon),v_2(\epsilon))^{\dagger}}_{(\dot{L}^2(0,2\pi))^2}^2.
\end{equation}
From equations \eqref{inq_2a} and \eqref{well_posed_inq_den_b}, we deduce that
\begin{equation}
\int_{0}^{T}\mod{\mc{Y}(t)}^2dt\geq C\int_{0}^{\epsilon}\mod{\bar{u}\sigma_2(t,2\pi)+\bar{\rho}v_2(t,2\pi)}^2dt
\end{equation}
Using this inequality, we obtain
\begin{align}\label{inq_3}
\int_{0}^{\epsilon}\mod{e^{\nu_{n_0}(T-t)}\left(a_{n_0}\mc{B}^*_{\rho}\Phi_{n_0}+(T-t)\tilde{a}_{n_0}\mc{B}^*_{\rho}\tilde{\Phi}_{n_0}\right)}^2dt&\leq C\int_{0}^{\epsilon}\mod{\mc{Y}(t)}^2dt+C\int_{0}^{\epsilon}\mod{\bar{u}\sigma_2(t,2\pi)+\bar{\rho}v_2(t,2\pi)}^2dt\\
&\leq C\int_{0}^{T}\mod{\mc{Y}(t)}^2dt\notag
\end{align}
We now prove that
\begin{equation}\label{inq_6}
\int_{0}^{\epsilon}\mod{e^{\nu_{n_0}(T-t)}\left(a_{n_0}\mc{B}^*_{\rho}\Phi_{n_0}+(T-t)\tilde{a}_{n_0}\mc{B}^*_{\rho}\tilde{\Phi}_{n_0}\right)}^2dt\geq C\left(\mod{a_{n_0}}^2+\mod{\tilde{a}_{n_0}}^2\right)
\end{equation}
Denote the finite dimensional space
\begin{equation*}
\mc{X}:=\text{span}\left\{\Phi_{n_0},\tilde{\Phi}_{n_0}\right\}
\end{equation*}
and define norms on $\mc{X}$:
\begin{align*}
\norm{(\hat{\sigma}_{T,1},\hat{v}_{T,1})^{\dagger}}_1^2:&=\int_{0}^{\epsilon}\mod{e^{\nu_{n_0}(T-t)}\left(a_{n_0}\mc{B}^*_{\rho}\Phi_{n_0}+(T-t)\tilde{a}_{n_0}\mc{B}^*_{\rho}\tilde{\Phi}_{n_0}\right)}^2dt,\\
\norm{(\hat{\sigma}_{T,1},\hat{v}_{T,1})^{\dagger}}_2^2:&=\norm{(\hat{\sigma}_1(0),\hat{v}_1(0))^{\dagger}}_{(\dot{L}^2(0,2\pi))^2}^2,
\end{align*}
where $(\hat{\sigma}_1,\hat{v}_1)^{\dagger}$ denotes the solution of the adjoint system with terminal data $(\hat{\sigma}_{T,1},\hat{v}_{T,1})^{\dagger}\in\mc{X}$. In fact, $\norm{(\hat{\sigma}_{T,1},\hat{v}_{T,1})^{\dagger}}_1=0$ implies $\mc{B}^*_{\rho}\Phi_{n_0}=\mc{B}^*_{\rho}\tilde{\Phi}_{n_0}=0$. This gives $\Phi_{n_0}=\tilde{\Phi}_{n_0}=0$ (thanks to Lemma \ref{lem_obs_est_b} - Remark \ref{rem_obs_gen_b}) and hence $(\hat{\sigma}_{T,1},\hat{v}_{T,1})=(0,0)$. Also, $(\hat{\sigma}_1(0),\hat{v}_1(0))=(0,0)$ implies $\Phi_{n_0}=\tilde{\Phi}_{n_0}=0$ and consequently $(\hat{\sigma}_1,\hat{v}_1)=(0,0)$.

%This can be seen as follows:
%\begin{equation*}
%\bar{u}\sigma_1(t,2\pi)+\bar{\rho}v_1(t,2\pi)=0\ \ \text{implies}\ (\sigma_1,v_1)\equiv0\ \ \text{in}\ (0,\frac{\epsilon}{2})\times(0,2\pi).
%\end{equation*}
%Indeed, we have
%\begin{equation*}
%\bar{u}\sigma_1(t,2\pi)+\bar{\rho}v_1(t,2\pi)=e^{\nu_{n_0}(T-t)}\left(\mc{B}^*_{\rho}\Phi_{\nu_0}+(T-t)\mc{B}^*_{\rho}\tilde{\Phi}_{\nu_0}\right)
%\end{equation*}
%and therefore $\bar{u}\sigma_1(t,2\pi)+\bar{\rho}v_1(t,2\pi)=0$ implies $\mc{B}^*_{\rho}\Phi_{\nu_0}=\mc{B}^*_{\rho}\tilde{\Phi}_{\nu_0}=0$. Thus, thanks to Lemma \ref{lem_obs_est_b}, we have $\Phi_{\nu_0}=\tilde{\Phi}_{\nu_0}=(0,0)$ and therefore $(\sigma_1(t,x),v_1(t,x))^{\dagger}=(0,0)$ in $(0,\frac{\epsilon}{2})\times(0,2\pi)$.

\smallskip

\noindent Since any two norms in a finite dimensional space are equivalent, we can write
\begin{equation*}
\int_{0}^{\epsilon}\mod{e^{\nu_{n_0}(T-t)}\left(a_{n_0}\mc{B}^*_{\rho}\Phi_{n_0}+(T-t)\tilde{a}_{n_0}\mc{B}^*_{\rho}\tilde{\Phi}_{n_0}\right)}^2dt\geq C\norm{(\hat{\sigma}_1(0),\hat{v}_1(0))^{\dagger}}_{(\dot{L}^2(0,2\pi))^2}^2,
\end{equation*}
proving the inequality \eqref{inq_6}. Hence, using \eqref{inq_3}, we finally obtain
\begin{equation}
\int_{0}^{T}\mod{\mc{Y}(t)}^2dt\geq C\norm{(\hat{\sigma}_1(0),\hat{v}_1(0))^{\dagger}}_{(\dot{L}^2(0,2\pi))^2}^2.
\end{equation}
This inequality, together with \eqref{inq_2} implies
\begin{align*}
\int_{0}^{T}\mod{\mc{Y}(t)}^2dt&\geq C\left[\norm{(\hat{\sigma}_1(0),\hat{v}_1(0))^{\dagger}}_{(\dot{L}^2(0,2\pi))^2}^2+\norm{(\sigma_2(0),v_2(0))^{\dagger}}_{(\dot{L}^2(0,2\pi))^2}^2\right]\\
&\geq C\norm{(\sigma_2(0)+\hat{\sigma}_1(0),v_2(0)+\hat{v}_1(0))^{\dagger}}_{(\dot{L}^2(0,2\pi))^2}^2.
\end{align*}
Proceeding in a similar way, we can add the term $e^{\nu_{-n_0}(T-t)}\left(a_{-n_0}\mc{B}^*_{\rho}\Phi_{-n_0}+(T-t)\tilde{a}_{-n_0}\mc{B}^*_{\rho}\tilde{\Phi}_{-n_0}\right)$ and obtain the desired observability inequality
\begin{equation*}
\int_{0}^{T}\mod{\bar{u}\sigma(t,2\pi)+\bar{\rho}v(t,2\pi)}^2dt\geq C\norm{(\sigma(0),v(0))^{\dagger}}_{(\dot{L}^2(0,2\pi))^2}^2.
\end{equation*}
This completes the proof of Theorem \ref{thm_den_b}-Part \eqref{null_dens_b} in the case of multiple eigenvalues.\qed

\bigskip

\noindent\underline{Control in velocity.} The proof of Theorem \ref{thm_vel_b}-Part \eqref{null_vel_b} (control acting in the velocity component) in the case of multiple eigenvalues can be done in a similar way as above. The only missing part is the following well-posedness result (see the inequality \eqref{well_posed_inq_den_b})
\begin{equation}\label{well_posed_vel_b_1}
\int_{0}^{\epsilon}\mod{b\sigma_2(t,2\pi)+\bar{u}v_2(t,2\pi)+\mu_0(v_2)_x(t,2\pi)}^2dt\leq C\norm{(\sigma_2(\epsilon),v_2(\epsilon))^{\dagger}}_{\dot{H}^{-1}_{\per}(0,2\pi)\times\dot{L}^2(0,2\pi)}.
\end{equation}
The terminal data $(\sigma_2,v_2)\in\dot{H}^{-1}_{\per}(0,2\pi)\times\dot{L}^2(0,2\pi)$ is less regular and so one cannot expect that the observation term $b\sigma_2(t,2\pi)+\bar{u}v_2(t,2\pi)+\mu_0(v_2)_x(t,2\pi)\in L^2(0,\epsilon)$ for some $\epsilon>0$. This is the main difficulty of boundary controllability in comparison with the distributed controllability. In this context, we refer to \cite[Equation (4.43)]{Chowdhury14}, where one can have the well-posedness result due to the internal control. However, in our setup, we can obtain a slightly modified estimate to \eqref{well_posed_vel_b_1} as follows:
\begin{equation}\label{well_posed_vel_b_2}
\int_{0}^{\frac{\epsilon}{2}}\mod{b\sigma_2(t,2\pi)+\bar{u}v_2(t,2\pi)+\mu_0(v_2)_x(t,2\pi)}^2dt\leq C\norm{(\sigma_2(\epsilon),v_2(\epsilon))^{\dagger}}_{\dot{H}^{-1}_{\per}(0,2\pi)\times\dot{L}^2(0,2\pi)}^2.
\end{equation}
Using this inequality \eqref{well_posed_vel_b_2} and proceeding similarly as before, we can obtain the observability inequality \eqref{obs_inq_b_vel} in the presence of multiple eigenvalues. Thus, the only technical part is to prove the inequality \eqref{well_posed_vel_b_2}, which we prove below:

%The inequality \eqref{well_posed_inq_den_b} can be proved directly using the expression of the solution $(\sigma_2,v_2)^{\dagger}$ given by \eqref{sigma_2_b}. Indeed, we have

\smallskip

\noindent Recall the expression of $(\sigma_2,v_2)^{\dagger}$ given by \eqref{sigma_2_b}. We compute
\begin{align*}
&\int_{0}^{\frac{\epsilon}{2}}\mod{b\sigma_2(t,2\pi)+\bar{u}v_2(t,2\pi)+\mu_0(v_2)_x(t,2\pi)}^2dt\\
&\leq\int_{0}^{\frac{\epsilon}{2}}\mod{\sum_{n\in\mb{Z}^*\setminus\{\pm n_0\}}\left(a_n^he^{\nu_n^h(T-t)}\mc{B}^*_u\Phi_n^h+a_n^pe^{\nu_n^p(T-t)}\mc{B}^*_u\Phi_n^p\right)}^2dt\\
&\leq\int_{0}^{\frac{\epsilon}{2}}\mod{\sum_{n\in\mb{Z}^*\setminus\{\pm n_0\}}a_n^he^{\nu_n^h(T-t)}\mc{B}^*_u\Phi_n^h}^2dt+\int_{0}^{\frac{\epsilon}{2}}\mod{\sum_{n\in\mb{Z}^*\setminus\{\pm n_0\}}a_n^pe^{\nu_n^p(T-t)}\mc{B}^*_u\Phi_n^p}^2dt%\\
%&\leq C\sum_{n\in\mb{Z}^*\setminus\{\pm n_0\}}\mod{a_n^h}^2\mod{\mc{B}^*_{\rho}\Phi_n^h}^2+2\sum_{n\in\mb{Z}^*\setminus\{\pm n_0\}}\mod{a_n^p}^2e^{2\Re(\nu_n^p)(T-\epsilon)}\sum_{n\in\mb{Z}^*\setminus\{\pm n_0\}}\mod{\mc{B}^*_{\rho}\Phi_n^p}^2e^{-2\Re(\nu_n^p)(T-\epsilon)}\int_{0}^{\frac{\epsilon}{2}}e^{2\Re(\nu_n^p)(T-t)}dt\\
%&\leq C\sum_{n\in\mb{Z}^*\setminus\{\pm n_0\}}\mod{a_n^h}^2\mod{\mc{B}^*_{\rho}\Phi_n^h}^2+C\sum_{n\in\mb{Z}^*\setminus\{\pm n_0\}}\mod{a_n^p}^2e^{2\Re(\nu_n^p)(T-\epsilon)}\sum_{n\in\mb{Z}^*\setminus\{\pm n_0\}}\frac{\mod{\mc{B}^*_{\rho}\Phi_n^p}^2}{-2\Re(\nu_n^p)T}e^{2\Re(\nu_n^p)\epsilon}\\
%&\leq C\sum_{n\in\mb{Z}^*\setminus\{\pm n_0\}}\mod{a_n^h}^2+C\sum_{n\in\mb{Z}^*\setminus\{\pm n_0\}}\mod{a_n^p}^2e^{2\Re(\nu_n^p)(T-\epsilon)}.
\end{align*}
Note that
\begin{equation}\label{inq_4a}
\int_{0}^{\frac{\epsilon}{2}}\mod{\sum_{n\in\mb{Z}^*\setminus\{\pm n_0\}}a_n^he^{\nu_n^h(T-t)}\mc{B}^*_u\Phi_n^h}^2dt\leq C\sum_{n\in\mb{Z}^*\setminus\{\pm n_0\}}\mod{a_n^h\mc{B}^*_u\Phi_n^h}^2\leq C\sum_{n\in\mb{Z}^*\setminus\{\pm n_0\}}\frac{\mod{a_n^h}^2}{\mod{n}^2}.
\end{equation}
Note that the estimate $\mod{\mc{B}^*_u\Phi_n^h}\leq\frac{C}{\mod{n}}$ follows due to the fact that $\mc{B}^*_u\Phi_n^h=\nu_2^n\alpha_2^n$ for all $n\in\mb{Z}^*$ (see the proof of Lemma \ref{lem_obs_est_b}). For the parabolic part, we write
\begin{align*}
&\int_{0}^{\frac{\epsilon}{2}}\mod{\sum_{n\in\mb{Z}^*\setminus\{\pm n_0\}}a_n^pe^{\nu_n^p(T-t)}\mc{B}^*_u\Phi_n^p}^2dt\\
&\leq\sum_{n\in\mb{Z}^*\setminus\{\pm n_0\}}\mod{a_n^p}^2e^{2\Re(\nu_n^p)(T-\epsilon)}\sum_{n\in\mb{Z}^*\setminus\{\pm n_0\}}\mod{\mc{B}^*_u\Phi_n^p}^2e^{-2\Re(\nu_n^p)(T-\epsilon)}\int_{0}^{\frac{\epsilon}{2}}e^{2\Re(\nu_n^p)(T-t)}dt\\
&\leq\sum_{n\in\mb{Z}^*\setminus\{\pm n_0\}}\mod{a_n^p}^2e^{2\Re(\nu_n^p)(T-\epsilon)}\sum_{n\in\mb{Z}^*\setminus\{\pm n_0\}}\mod{\mc{B}^*_u\Phi_n^p}^2e^{\Re(\nu_n^p)\epsilon}\\
&\leq C\sum_{n\in\mb{Z}^*\setminus\{\pm n_0\}}\mod{a_n^p}^2e^{2\Re(\nu_n^p)(T-\epsilon)},
\end{align*}
as we have $\Re(\nu_n^p)<0$ for all $n\in\mb{Z}^*$. Combining these two estimates, we obtain
\begin{align}\label{inq_4}
\int_{0}^{\frac{\epsilon}{2}}&\mod{b\sigma_2(t,2\pi)+\bar{u}v_2(t,2\pi)+\mu_0(v_2)_x(t,2\pi)}^2dt\\
&\hspace{4cm}\leq C\sum_{n\in\mb{Z}^*\setminus\{\pm n_0\}}\frac{\mod{a_n^h}^2}{\mod{n}^2}+C\sum_{n\in\mb{Z}^*\setminus\{\pm n_0\}}\mod{a_n^p}^2e^{2\Re(\nu_n^p)(T-\epsilon)}\notag
\end{align}
On the other hand (recall the expression given by \eqref{sigma_2_b} and \eqref{exp_e_fns_b}), we have
\begin{align*}
&\norm{(\sigma_2(\epsilon),v_2(\epsilon))^{\dagger}}_{\dot{H}^{-1}_{\per}(0,2\pi)\times\dot{L}^2(0,2\pi)}^2\\
&=\sum_{n\in\mb{Z}^*\setminus\{\pm n_0\}}\left(\frac{b}{\mod{n}^2}\mod{a_n^he^{\nu_n^h(T-\epsilon)}\bar{\rho}+a_n^pe^{\nu_n^p(T-\epsilon)}\frac{\bar{\rho}}{\nu_1^n-\bar{u}}}^2+\bar{\rho}\mod{a_n^he^{\nu_n^h(T-\epsilon)}(\nu_2^n-\bar{u})+a_n^pe^{\nu_n^p(T-\epsilon)}}^2\right)
\end{align*}
Since $\nu_1^n-\bar{u}\sim_{+\infty}n$ and $\nu_2^n-\bar{u}\sim_{+\infty}\frac{1}{n}$, we deduce that for $N$ large enough
\begin{align}\label{inq_5}
\sum_{\mod{n}>N}\left(\frac{b}{\mod{n}^2}\mod{a_n^he^{\nu_n^h(T-\epsilon)}\bar{\rho}+a_n^pe^{\nu_n^p(T-\epsilon)}\frac{\bar{\rho}}{\nu_1^n-\bar{u}}}^2+\bar{\rho}\mod{a_n^he^{\nu_n^h(T-\epsilon)}(\nu_2^n-\bar{u})+a_n^pe^{\nu_n^p(T-\epsilon)}}^2\right)\\
\geq C\sum_{\mod{n}>N}\frac{\mod{a_n^h}^2}{\mod{n}^2}+C\sum_{\mod{n}>N}\mod{a_n^p}^2e^{2\Re(\nu_n^p)(T-\epsilon)}.\notag
\end{align}
\begin{comment}
This can be seen from the following inequality:
\begin{equation}\label{new_inq_b}
\frac{1}{\mod{n}^2}\mod{z_n+\frac{w_n}{n}}^2+\mod{\frac{z_n}{n}+w_n}^2\geq C\left(\frac{\mod{z_n}^2}{\mod{n}^2}+\mod{w_n}^2\right),
\end{equation}
for all complex sequences $(z_n)_{\mod{n}>N}$ and $(w_n)_{\mod{n}>N}$, where $N\in\mb{N}$ is arbitrarily large number. Indeed, if $z_n,w_n\neq0$ for all $\mod{n}>N$, then we can write
\begin{align*}
\frac{1}{\mod{n}^2}\mod{z_n+\frac{w_n}{n}}^2+\mod{\frac{z_n}{n}+w_n}^2=\frac{\mod{z_n}^2}{\mod{n}^2}\mod{1+\frac{w_n}{nz_n}}^2+\mod{w_n}^2\mod{\frac{z_n}{nw_n}+1}^2,\ \ \mod{n}>N.
\end{align*}
If $\inf_{\mod{n}>N}\mod{1+\frac{w_n}{nz_n}}>0$ and $\inf_{\mod{n}>N}\mod{\frac{z_n}{nw_n}+1}>0$, the inequality \eqref{new_inq_b} is obvious. Let us assume $\inf_{\mod{n}>N}\mod{1+\frac{w_n}{nz_n}}=0$, then we can write $1+\frac{w_n}{nz_n}=O(\frac{1}{n^{\delta}})$ for some $\delta>0$. This implies $\frac{w_n}{z_n}=-n+O(n^{1-\delta})$ for infinitely many $\mod{n}>N$ and therefore
\begin{align*}
\mod{\frac{z_n}{n}+w_n}^2=\mod{z_n}^2\mod{\frac{1}{n}+\frac{w_n}{z_n}}^2=\mod{z_n}^2\mod{\frac{1}{n}-n+O(n^{1-\delta})}^2\geq C\mod{n}^2\mod{z_n}^2.
\end{align*}
On the other hand, we have $\mod{n}^2\mod{z_n}^2=\frac{\mod{n}^2}{2}\mod{z_n}^2+\frac{\mod{n}^2}{2}\mod{z_n}^2\geq\frac{\mod{n}^2}{2}\mod{z_n}^2+C\mod{w_n}^2$, proving the inequality \eqref{new_inq_b}. Similarly, we can prove \eqref{new_inq_b} in the case $\inf_{\mod{n}>N}\mod{\frac{z_n}{nw_n}+1}=0$. %Further, this inequality is obvious when $(z_n,w_n)=(0,0)$ for some $n$.
\end{comment}

Now, adding finitely many terms in the estimate \eqref{inq_5} (or, one can include these finitely many terms in $(\sigma_1,v_1)^{\dagger}$ part), we get that
\begin{align}
\norm{(\sigma_2(\epsilon),v_2(\epsilon))^{\dagger}}_{\dot{H}^{-1}_{\per}(0,2\pi)\times\dot{L}^2(0,2\pi)}^2\geq C\left(\sum_{n\in\mb{Z}^*\setminus\{\pm n_0\}}\frac{\mod{a_n^h}^2}{\mod{n}^2}+\sum_{n\in\mb{Z}^*\setminus\{\pm n_0\}}\mod{a_n^p}^2e^{2\Re(\nu_n^p)(T-\epsilon)}\right).
\end{align}
With this, the inequality \eqref{well_posed_vel_b_2} follows.\qed

%In a similar way, we can prove the following estimate when the control is acting in the velocity component:
%\begin{equation}
%\int_{0}^{\frac{\epsilon}{2}}\mod{b\sigma(t,2\pi)+\bar{u}v(t,2\pi)+\mu_0v_x(t,2\pi)}^2dt\leq C\norm{(\sigma(\epsilon),v(\epsilon))^{\dagger}}_{\dot{H}^{-1}_{\per}(0,2\pi)\times\dot{L}^2(0,2\pi)}^2.
%\end{equation}
%This estimate helps us prove the null controllability of \eqref{lcnse_b} (using a control acting in velocity part) in the case of multiple eigenvalues.

}

\bigskip

\subsection{Lack of null controllability for less regular initial states}
{\blue
We first write the following result, the proof of which is standard and so we skip the details (see Theorem \ref{equiv_thm_vel_b}).
\begin{proposition}
Let $0\leq s<1$ and $T>0$ be given. Then, the system \eqref{lcnse_b}-\eqref{in_cd_b}-\eqref{bd_cd_b2} is null controllable at time $T$ in the space $\dot{H}^s_{\per}(0,2\pi)\times\dot{L}^2(0,2\pi)$ if and only if the inequality
\begin{equation}\label{obs_inq_b_vel_s}
\norm{(\sigma(0),v(0))^{\dagger}}_{\dot{H}^{-s}_{\per}(0,2\pi)\times\dot{L}^2(0,2\pi)}^2\leq C\int_{0}^{T}\mod{b\sigma(t,2\pi)+\bar{u}v(t,2\pi)+\mu_0v_x(t,2\pi)}^2dt
\end{equation}
holds for all solutions $(\sigma,v)^{\dagger}$ of the adjoint system \eqref{lcnse_adj_b} with terminal data $(\sigma_T,v_T)^{\dagger}\in\mc{D}(A^*)$.
\end{proposition}
To prove Theorem \ref{thm_vel_b}-Part \eqref{lac_vel_b}, it is enough to find a sequence of terminal data $(\sigma_T^n,v_T^n)_{n\in\mb{Z}^*}\in\mc{D}(A^*)$ for which the observability inequality \eqref{obs_inq_b_vel_s} fails. We will show below that the eigenfunctions corresponding to the hyperbolic branch of eigenvalues helps us disprove this observability inequality.
}
\subsubsection{Proof of Theorem \ref{thm_vel_b}-Part \eqref{lac_vel_b}}
For $\blue(\sigma_T^n,v_T^n)^{\dagger}=\Phi_n^h$, the solution to the adjoint system \eqref{lcnse_adj_b} is
\begin{equation*}
{\blue(\sigma^n(t,x),v^n(t,x))^{\dagger}}=e^{\nu_n^h(T-t)}\Phi_n^h(x),
\end{equation*}
for $(t,x)\in (0,T)\times(0,2\pi)$ {\blue and $n\in\mb{Z}^*$}. Recall the expression of $\Phi_n^h$ from \eqref{exp_e_fns_b}. For all $n\in\mb{Z}^*$, we have the following estimate
\begin{equation*}
\norm{\Phi_n^h}_{\dot{H}^{-s}_{\per}(0,2\pi)\times \dot{L}^2(0,2\pi)}\geq\frac{C}{\mod{n}^s},
\end{equation*}
and therefore
\begin{equation*}
{\blue\norm{(\sigma^n(0),v^n(0))^{\dagger}}_{\dot{H}^{-s}_{\per}(0,2\pi)\times \dot{L}^2(0,2\pi)}^2}\geq \frac{C}{\mod{n}^{2s}}
\end{equation*}
for all $n\in\mb{Z}^*$, since $\Re(\nu_n^h)$ is bounded. On the other hand, we have the upper bound of the observation term
\begin{equation*}
\blue\int_{0}^{T}\mod{b\sigma^n(t,2\pi)+\bar{u}v^n(t,2\pi)+\mu_0v_x^n(t,2\pi)}^2dt\leq\frac{C}{\mod{n}^2},
\end{equation*}
for all $n\in\mb{Z}^*$ {\blue(see \eqref{inq_4a} for instance)}. Thus, if the observability inequality \eqref{obs_inq_b_vel_s} holds, then one must have
\begin{equation*}
\frac{C}{\mod{n}^{2s}}\leq\frac{C}{\mod{n}^2}\implies\mod{n}^{2-2s}\leq C,
\end{equation*}
which is not possible due to our assumption $0\leq s<1$. This completes the proof.\qed

\subsection{Lack of controllability at small time}\label{sec_lac_null_b}
We prove that the system \eqref{lcnse_nb} is not null controllable {\blue in $\dot{L}^2(0,2\pi)$} when the time is small, that is, Theorem \ref{thm_den_b}-Part \eqref{lac_dens_b}. We construct an approximate solution for the corresponding transport equation. The idea of constructing an approximate solution for the transport equation was addressed in \cite{Beauchard20}, where the authors proved a lack of null controllability result at a small time in the case of an interior control (acts in the transport equation). Very recently, in \cite[Section 6]{Chowdhury22}, this approach has been applied to a coupled transport-elliptic system in the case of a boundary control (acts in density). We will follow mainly the proof given in \cite{Chowdhury22} to prove our lack of null controllability result when the time is small.

\subsubsection{Proof of Theorem \ref{thm_den_b}-Part \eqref{lac_dens_b}}
Let $0<T<\frac{2\pi}{\bar{u}}$. {\blue We first consider the transport equation}
\begin{equation}\label{trnsprt_b1}
\begin{dcases}
\tilde{\sigma}_t(t,x)+\bar{u}\tilde{\sigma}_x(t,x)-\frac{b\bar{\rho}}{\mu_0}\tilde{\sigma}(t,x)=0,\ \ (t,x)\in (0,T)\times(0,2\pi),\\
\tilde{\sigma}(t,0)=\tilde{\sigma}(t,2\pi),\ \ t\in (0,T),\\
\tilde{\sigma}(T,x)=\tilde{\sigma}_T(x),\ \ x\in(0,2\pi)
\end{dcases}
\end{equation}
{\blue with $\tilde{\sigma}_T\in \dot{L}^2(0,2\pi)$ .} Since $\bar{u}T<2\pi$, there exists a nontrivial function $\tilde{\sigma}_T\in C^{\infty}(0,2\pi)$ with $\text{supp}(\tilde{\sigma}_T)\subset(\bar{u}T,2\pi)$ such that {\blue the solution $\tilde{\sigma}$ of \eqref{trnsprt_b1} satisfies %$\text{supp}(\tilde{\sigma})\subset\blue\{(t,x)\ :\ \bar{u}t<x<\bar{u}t+(2\pi-\bar{u}T),\ 0<t<T\}$. {\blue In particular, we have 
{$\tilde{\sigma}(t,0)=\tilde{\sigma}(t,2\pi)=0$ but $\tilde{\sigma}$ is not identically zero in $(0,T)\times(0,2\pi)$.}} Let $N>0$ be a fixed integer. We define the polynomial
\begin{equation}\label{pn_b}
P^N(x):=\prod_{\substack{l=-N\\\blue l\neq0}}^{N}(x-l),\ \ x\in(0,2\pi)
\end{equation}
and the function
\begin{equation}\label{new_term_b}
\tilde{\sigma}_T^N:=P^N\left(-i\frac{d}{dx}\right)\tilde{\sigma}_T.
\end{equation}
{\blue Since $\tilde{\sigma}_T\in \dot{L}^2(0,2\pi)$, we can write}
\begin{equation*}
\tilde{\sigma}_T(x):=\sum_{n\in\mb{Z}^*}a_ne^{inx},\ \ x\in(0,2\pi),
\end{equation*}
{\blue where $(a_n)_{n\in\mb{Z}}\in\ell_2$. Using the definition of $P^N$ given by \eqref{pn_b}, we get from \eqref{new_term_b} that}
\begin{equation*}
\tilde{\sigma}_T^N(x)=\sum_{n\in\mb{Z}^*}a_n\prod_{\substack{l=-N\\\blue l\neq0}}^{N}\left(-i\frac{d}{dx}-l\right)e^{inx}=\sum_{n\in\mb{Z}^*}a_n\prod_{\substack{l=-N\\\blue l\neq0}}^{N}\left(n-l\right)e^{inx}=\sum_{n\in\mb{Z}^*}a_nP^N(n)e^{inx},
\end{equation*}
for $(t,x)\in (0,T)\times(0,2\pi)$. Note that $P^N(n)=0$ for all $0<\mod{n}\leq N$ and therefore
\begin{equation*}
\tilde{\sigma}_T^N(x)=\sum_{\mod{n}\geq N+1}a_nP^N(n)e^{inx}.
\end{equation*}
With this $\tilde{\sigma}_T^N$, let us now consider the following system
\begin{equation}\label{trnsprt_b2}
\begin{dcases}
\tilde{\sigma}_t(t,x)+\bar{u}\tilde{\sigma}_x(t,x)-\frac{b\bar{\rho}}{\mu_0}\tilde{\sigma}(t,x)=0,\ \ (t,x)\in (0,T)\times(0,2\pi),\\
\tilde{\sigma}(t,0)=\tilde{\sigma}(t,2\pi),\ \ t\in (0,T),\\
\tilde{\sigma}(T,x)=\tilde{\sigma}_T^N(x),\ \ x\in(0,2\pi).
\end{dcases}
\end{equation}
Since $\text{supp}(\tilde{\sigma}_T^N)\subset\text{supp}(\tilde{\sigma}_T)\subset\blue(\bar{u}T,2\pi)$, the solution {\blue$\tilde{\sigma}^N$ of \eqref{trnsprt_b2}} satisfies $\tilde{\sigma}^N(t,0)=\tilde{\sigma}^N(t,2\pi)=0$. We now consider the following adjoint system
\begin{equation}\label{lcnse_adj_b1}
\begin{dcases}
\sigma_t(t,x)+\bar{u}\sigma_x(t,x)+\bar{\rho}v_x(t,x)=0,\ \ (t,x)\in (0,T)\times(0,2\pi),\\
v_t(t,x)-\mu_0v_{xx}(t,x)+\bar{u}v_x(t,x)+b\sigma_x(t,x)=0,\ \ (t,x)\in (0,T)\times(0,2\pi),\\
\sigma(t,0)=\sigma(t,2\pi),\ \ t\in (0,T),\\
v(t,0)=v(t,2\pi),\ \ v_x(t,0)=v_x(t,2\pi),\ \ t\in (0,T),\\
\sigma(T,x)=\tilde{\sigma}_T^N(x),\ \ v(T,x)=v_T^N(x),\ \ x\in(0,2\pi),
\end{dcases}
\end{equation}
where we choose $v_T^N$ such that
\begin{equation*}
(\tilde{\sigma}_T^N,v_T^N)^{\dagger}=\sum_{\mod{n}\geq N+1}\tilde{a}_n^h\Phi_n^h
\end{equation*}
with $\tilde{a}_n^h\bar{\rho}:=a_nP^N(n)$ for all $\mod{n}\geq N+1$. We write the solutions to the systems \eqref{trnsprt_b2} and \eqref{lcnse_adj_b1} respectively as
\begin{align}
\tilde{\sigma}^N(t,x)&=\sum_{\mod{n}\geq N+1}a_nP^N(n)e^{(\bar{u}in-\frac{b\bar{\rho}}{\mu_0})(T-t)}e^{inx},\\
\sigma^N(t,x)&=\sum_{\mod{n}\geq N+1}a_nP^N(n)e^{\nu_n^h(T-t)}e^{inx},\\
v^N(t,x)&=\sum_{\mod{n}\geq N+1}a_nP^N(n)\frac{\nu_2^n-\bar{u}}{\bar{\rho}}e^{\nu_n^h(T-t)}e^{inx},
\end{align}
for $(t,x)\in[0,T]\times[0,2\pi]$. We prove that the solution component $\sigma^N$ of \eqref{lcnse_adj_b1} approximates the solution $\tilde{\sigma}^N$ of \eqref{trnsprt_b2}. Indeed,
\begin{align*}
&\norm{\sigma^N(\cdot,x)-\tilde{\sigma}^N(\cdot,x)}_{L^2(0,T)}^2\\
&\leq\sum_{\mod{n}\geq N+1}\mod{a_n}^2\mod{P^N(n)}^2\norm{e^{\nu_n^h(T-t)}-e^{\left(\bar{u}in-\frac{b\bar{\rho}}{\mu_0}\right)(T-t)}}_{L^2(0,T)}^2\\
%&\leq\sum_{\mod{n}\geq N+1}\mod{a_n}^2\mod{P^N(n)}^2\norm{e^{\bar{u}in(T-t)}e^{-\frac{\mu_0n}{2}(n-\sqrt{n^2-\frac{4b\bar{\rho}}{\mu_0^2}})(T-t)}-e^{(\bar{u}in-\frac{b\bar{\rho}}{\mu_0})(T-t)}}_{L^2(0,T)}^2\\
&\leq\sum_{\mod{n}\geq N+1}\mod{a_n}^2\mod{P^N(n)}^2\norm{e^{-\frac{\mu_0n}{2}\left(n-\sqrt{n^2-\frac{4b\bar{\rho}}{\mu_0^2}}\right)(T-t)}-e^{-\frac{b\bar{\rho}}{\mu_0}(T-t)}}_{L^2(0,T)}^2\\
&\leq\sum_{\mod{n}\geq N+1}\frac{1}{\mod{n}^2}\mod{a_n}^2\mod{P^N(n)}^2,
\end{align*}
for all $x\in[0,2\pi]$ and therefore
\begin{equation*}
\norm{\sigma^N(\cdot,x)-\tilde{\sigma}^N(\cdot,x)}_{L^2(0,T)}^2\leq\frac{C}{\mod{N}^2}\sum_{\mod{n}\geq N+1}\mod{a_n}^2\mod{P^N(n)}^2,
\end{equation*}
for all $x\in[0,2\pi]$. We also find $L^2$- estimate of the solution component $v^N$. We have for all $x\in[0,2\pi]$
\begin{align*}
\norm{v^N(\cdot,x)}_{L^2(0,T)}^2&\leq\sum_{\mod{n}\geq N+1}\mod{a_n}^2\mod{P^N(n)}^2\frac{\mod{\nu_2^n-\bar{u}}^2}{\bar{\rho}^2}\norm{e^{\nu_n^h(T-t)}}_{\dot{L}^2(0,T)}^2\\
&\leq C\sum_{\mod{n}\geq N+1}\mod{a_n}^2\mod{P^N(n)}^2\frac{1}{\mod{n}^2}\\
&\leq\frac{C}{\mod{N}^2}\sum_{\mod{n}\geq N+1}\mod{a_n}^2\mod{P^N(n)}^2.
\end{align*}
Let us now suppose that the following observability inequality holds
\begin{equation*}
\int_{0}^{T}\mod{\bar{u}\sigma^N(t,2\pi)+\bar{\rho}v^N(t,2\pi)}^2dt\geq C\norm{(\sigma^N(0),v^N(0))}_{{\blue(\dot{L}^2(0,2\pi))^2}}^2.
\end{equation*}
Then, we have
\begin{align*}
&\norm{(\sigma^N(0),v^N(0))}_{{\blue(\dot{L}^2(0,2\pi))^2}}^2\\
&\leq C\int_{0}^{T}\mod{\bar{u}\sigma^N(t,2\pi)+\bar{\rho}v^N(t,2\pi)}^2dt\\
&\leq C\int_{0}^{T}\left(\bar{u}^2\mod{(\sigma^N(t,2\pi)-\tilde{\sigma}^N(t,2\pi))}^2+\bar{u}^2\mod{\tilde{\sigma}^N(t,2\pi)}^2+\bar{\rho}^2\mod{v^N(t,2\pi)}^2\right)dt\\
&\leq\frac{C}{N^2}\sum_{\mod{n}\geq N+1}\mod{a_n}^2\mod{P^N(n)}^2,
\end{align*}
as we have $\tilde{\sigma}^N(t,0)=0=\tilde{\sigma}^N(t,2\pi)$ for all $t\in (0,T)$. Thus we get
\begin{equation*}
\norm{\sigma^N(0)}_{{\blue\dot{L}^2(0,2\pi)}}^2\leq \norm{(\sigma^N(0),v^N(0))^{\dagger}}_{{\blue(\dot{L}^2(0,2\pi))^2}}^2\leq\frac{C}{N^2}\sum_{\mod{n}\geq N+1}\mod{a_n}^2\mod{P^N(n)}^2\leq\frac{C}{N^2}\norm{\sigma^N(0)}_{{\blue\dot{L}^2(0,2\pi)}}^2,
\end{equation*}
since $\Re(\nu_n^h)$ is bounded. Therefore, $1\leq\frac{C}{N^2}$ for all $N$ and hence the above inequality cannot hold. This is a contradiction and the proof is complete.\qed

{\blue
\subsection{Lack of approximate controllability}
In this section, we prove that the system \eqref{lcnse_b} is not approximately controllable at any time $T>0$ in $(L^2(0,2\pi))^2$ when we have the restriction on the coefficients $\frac{2\sqrt{b\bar{\rho}-\bar{u}^2}}{\mu_0}\in\mb{N}$ (that is, Proposition \ref{prop_lac_app_b}). We present the proof of Proposition \ref{prop_lac_app_b} in the case when there is a boundary control acting in density component. The proof will be similar for the velocity control case and so we omit the details.

\subsubsection{Proof of Proposition \ref{prop_lac_app_b}}
Let $T>0$ be given and $\frac{2\sqrt{b\bar{\rho}-\bar{u}^2}}{\mu_0}\in\mb{N}$. To prove this result (in the density case), it is enough to find a terminal data $(\sigma_T,v_T)\in\mc{D}(A^*)$ such that the associated solution $(\sigma,v)$ of \eqref{lcnse_adj_b} fails to satisfy the following unique continuation property:
\begin{equation*}
\bar{u}\sigma(t,2\pi)+\bar{\rho}v(t,2\pi)=0\ \ \text{implies}\ \ (\sigma,v)=(0,0),
\end{equation*}
see for instance \cite[Theorem 2.43]{Coron07}. Let us denote $n_1:=\frac{2\sqrt{b\bar{\rho}-\bar{u}^2}}{\mu_0}$ and the eigenvalue $\nu_{n_1}^p=\nu_{-n_1}^p=:\nu_{n_1}$. The eigenfunctions of $A^*$ corresponding to this multiple eigenvalue $\nu_{n_1}$ are $\Phi_{n_1}^p=\vector{\frac{\bar{\rho}}{\nu_1^{n_1}-\bar{u}}}{1}e^{in_1x}$ and $\Phi_{-n_1}^p=\vector{\frac{\bar{\rho}}{\nu_1^{-n_1}-\bar{u}}}{1}e^{-in_1x}$ (see \eqref{exp_e_fns_b} in Lemma \ref{lemma_eigen_b}). We now choose the terminal data as
\begin{equation*}
(\sigma_T,v_T)^{\dagger}=C\Phi_{n_1}^p+D\Phi_{-n_1}^p,
\end{equation*}
where $C,D$ are (complex) constants that will be chosen later. The solution of \eqref{lcnse_adj_b} is then given by
\begin{equation*}
(\sigma(t),v(t))^{\dagger}=e^{\nu_{n_1}(T-t)}\left(C\Phi_{n_1}^p+D\Phi_{-n_1}^p\right).
\end{equation*}
Therefore
\begin{equation*}
\bar{u}\sigma(t,2\pi)+\bar{\rho}v(t,2\pi)=e^{\nu_{n_1}(T-t)}\left(C\mc{B}^*_{\rho}\Phi_{n_1}^p+D\mc{B}^*_{\rho}\Phi_{-n_1}^p\right).
\end{equation*}
If we take $C=-\mc{B}^*_{\rho}\Phi_{-n_1}^p$ and $D=\mc{B}^*_{\rho}\Phi_{n_1}^p$, then $C,D\neq0$ (thanks to Lemma \ref{lem_obs_est_b}) and for these choice of $C,D$, we have $\bar{u}\sigma(t,2\pi)+\bar{\rho}v(t,2\pi)=0$ but $(\sigma,v)\neq(0,0)$. This completes the proof.\qed
}
%%%%%%%%%%%%%%%%%%%%%%%%%%%%%%%%%%%%%%%%%%%%%%%%%%%%%%%%%%%%%%%%%%%%%%%%%
\section{Controllability of linearized compressible Navier-Stokes system (non-barotropic)}
\subsection{Functional setting}
{\blue Recall} the positive constants (equation \eqref{diff_coeff_nb})
\begin{equation*}
\lambda_0:=\frac{\lambda+2\mu}{\bar{\rho}},\ \ \kappa_0:=\frac{\kappa}{\bar{\rho}c_{\nu}},
\end{equation*}
and from now on-wards, we re-denote $c_{\nu}$ by $c_0$ to distinguish it from the eigenvalue $\nu$.

We define the inner product in the space $(L^2(0,2\pi))^3$ as follows
\begin{equation*}
\ip{\cv{f_1}{g_1}{h_1}}{\cv{f_2}{g_2}{h_2}}_{\blue L^2\times L^2\times L^2}:=R\bar{\theta}\int_{0}^{2\pi}f_1(x)\overline{f_2(x)}dx+\bar{\rho}^2\int_{0}^{2\pi}g_1(x)\overline{g_2(x)}dx+\frac{\bar{\rho}^2c_0}{\bar{\theta}}\int_{0}^{2\pi}h_1(x)\overline{h_2(x)}dx,
\end{equation*}
for $f_i,g_i,h_i\in L^2(0,2\pi), i=1,2,3$. {\blue From now on-wards, the notation $\ip{\cdot}{\cdot}_{L^2\times L^2\times L^2}$ means the above inner product in $L^2\times L^2\times L^2$.} We write the system \eqref{lcnse_nb} in abstract differential equation
\begin{equation}\label{lcnse_abs_nb}
U^{\prime}(t)=AU(t),\ \ U(0)=U_0,\ \ t\in (0,T),
\end{equation}
where $U:=(\rho,u,\theta)^{\dagger}, U_0:=(\rho_0,u_0,\theta_0)^{\dagger}$ and the operator $A$ is given by
\begin{equation*}
A:=\begin{pmatrix}-\bar{u}\partial_x&-\bar{\rho}\partial_x&0\\[4pt]-\frac{R\bar{\theta}}{\bar{\rho}}\partial_x&\lambda_0\partial_{xx}-\bar{u}\partial_x&-R\partial_x\\[4pt]0&-\frac{R\bar{\theta}}{c_0}\partial_x&\kappa_0\partial_{xx}-\bar{u}\partial_x\end{pmatrix}
\end{equation*}
with the domain
\begin{equation}%\label{domain_A_nb}
\mc{D}(A):=H^1_{\per}(0,2\pi)\times(H^2_{\per}(0,2\pi))^2.
\end{equation}
The adjoint of the operator $\blue A$ is given by
\begin{equation}\label{op_A^*_nb}
A^*:=\begin{pmatrix}\bar{u}\partial_x&\bar{\rho}\partial_x&0\\[4pt]\frac{R\bar{\theta}}{\bar{\rho}}\partial_x&\lambda_0\partial_{xx}+\bar{u}\partial_x&R\partial_x\\[4pt]0&\frac{R\bar{\theta}}{c_0}\partial_x&\kappa_0\partial_{xx}+\bar{u}\partial_x\end{pmatrix}
\end{equation}
with the same domain $\mc{D}(A^*)=\mc{D}(A)$. The adjoint system is given by
\begin{equation}\label{lcnse_adj_nb}
\begin{dcases}
-\sigma_t-\bar{u}\sigma_x-\bar{\rho}v_x=0,\ &\text{in}\ (0,T)\times(0,2\pi),\\
-v_t-\lambda_0v_{xx}-\frac{R\bar{\theta}}{\bar{\rho}}\sigma_x-\bar{u}v_x-R\vphi_x=0,\ &\text{in}\ (0,T)\times(0,2\pi),\\
-\vphi_t-\kappa_0\vphi_{xx}-\frac{R\bar{\theta}}{c_0}v_x-\bar{u}\vphi_x=0,\ &\text{in}\ (0,T)\times(0,2\pi),\\
\sigma(t,0)=\sigma(t,2\pi),\ \ v(t,0)=v(t,2\pi),\ \ v_x(t,0)=v_x(t,2\pi),\ \ &t\in (0,T),\\
\vphi(t,0)=\vphi(t,2\pi),\ \ \vphi_x(t,0)=\vphi_x(t,2\pi),\ \ &t\in (0,T),\\
\sigma(T,x)=\sigma_T(x),\ \ v(T,x)=v_T(x),\ \ \vphi(T,x)=\vphi_T(x),\ \ &x\in(0,2\pi),
\end{dcases}
\end{equation}
{\blue where} $(\sigma_T,v_T,\vphi_T)$ is a terminal state. We also write the following system with source terms $f,g$, and $h$.
\begin{equation}\label{lcnse_adj_nb_s}
\begin{dcases}
-\sigma_t-\bar{u}\sigma_x-\bar{\rho}v_x=f,\ &\text{in}\ (0,T)\times(0,2\pi),\\
-v_t-\lambda_0v_{xx}-\frac{R\bar{\theta}}{\bar{\rho}}\sigma_x-\bar{u}v_x-R\vphi_x=g,\ &\text{in}\ (0,T)\times(0,2\pi),\\
-\vphi_t-\kappa_0\vphi_{xx}-\frac{R\bar{\theta}}{c_0}v_x-\bar{u}\vphi_x=h,\ &\text{in}\ (0,T)\times(0,2\pi),\\
\sigma(t,0)=\sigma(t,2\pi),\ \ v(t,0)=v(t,2\pi),\ \ v_x(t,0)=v_x(t,2\pi),\ \ &t\in (0,T),\\
\vphi(t,0)=\vphi(t,2\pi),\ \ \vphi_x(t,0)=\vphi_x(t,2\pi),\ \ &t\in (0,T),\\
\sigma(T,x)=\sigma_T(x),\ \ v(T,x)=v_T(x),\ \ \vphi(T,x)=\vphi_T(x),\ \ &x\in(0,2\pi).
\end{dcases}
\end{equation}

\subsection{Well-posedness of the systems}
{\blue We first state the following well-posedness result of the system \eqref{lcnse_nb} when there is no control input.}
\begin{lemma}\label{lem_ex_sg_nb}
The operator $A$ (resp. $A^*$) generates a $\mc{C}^0$-semigroup of contractions on $(L^2(0,2\pi))^3$. Moreover, for every $U_0\in(L^2(0,2\pi))^3$ the system \eqref{lcnse_abs_b} admits a unique solution $U$ in $\mc{C}^0([0,T];(L^2(0,2\pi))^3)$ and
\begin{equation*}
\norm{U(t)}_{(L^2(0,2\pi))^3}\leq C\norm{U_0}_{(L^2(0,2\pi))^3}
\end{equation*}
for all $t\geq0$.
\end{lemma}
\noindent{\blue For the sake of completeness, we give a proof of this result in Appendix \ref{app_well_p_nb}.}
\begin{lemma}\label{well_posedness_adj_nb}
{\blue	The following statements hold:}
\begin{enumerate}
\item {\blue For any given $(f,g,h)\in L^2(0,T;(L^2(0,2\pi))^3)$ and $(\sigma_T,v_T,\vphi_T)\in(L^2(0,2\pi))^3$, the adjoint system \eqref{lcnse_adj_nb_s} has a unique solution $(\sigma,v,\vphi)$ in the space}
\begin{equation*}
\mc{C}^0([0,T];L^2(0,2\pi))\times[\mc{C}^0([0,T];L^2(0,2\pi))\cap L^2(0,T;H^1_{\per}(0,2\pi))]^2.%\times[\mc{C}^0([0,T];L^2(0,2\pi))\cap L^2(0,T;H^1_{\per}(0,2\pi))].
\end{equation*}
{\blue Moreover, we have the hidden regularity property $\sigma(\cdot,2\pi)\in L^2(0,T)$.}
\item {\blue For any given $(f,g,h)\in L^2(0,T;H^1_{\per}(0,2\pi)\times(L^2(0,2\pi))^2)$ and $(\sigma_T,v_T,\vphi_T)\in H^{-1}_{\per}(0,2\pi)\times(L^2(0,2\pi))^2$, the adjoint system \eqref{lcnse_adj_nb_s} has a unique solution $(\sigma,v,\vphi)$ in
\begin{equation*}
\mc{C}^0([0,T];H^{-1}_{\per}(0,2\pi)\times(L^2(0,2\pi))^2).
\end{equation*}

In particular, when $(\sigma_T,v_T,\vphi_T)=(0,0,0)$, the solution $(\sigma,v,\vphi)$ belong to the space $$\mc{C}^0([0,T];H^1_{\per}(0,2\pi))\times[\mc{C}^0([0,T];H^1_{\per}(0,2\pi))\cap L^2(0,T;H^2_{\per}(0,2\pi))]^2.$$
}
\end{enumerate}
\end{lemma}
{\blue The proof of this lemma can be proved similarly as in the barotropic case (Lemma \ref{well_posed_adj_b}), see for instance \cite{Debayan15}.}
\begin{definition}\label{def_sol_nb}
We give the following definitions of solutions based on the act of the controls.
\begin{itemize}
\item For any given initial state $(\rho_0,u_0,\theta_0)\in(L^2(0,2\pi))^3$ and boundary control $p\in L^2(0,T)$, we say $(\rho,u,\theta)\in(L^2(0,2\pi))^3$ is a solution to the system \eqref{lcnse_nb}-\eqref{in_cd_nb}-\eqref{bd_cd_nb1} if for any $(f,g,h)\in L^2(0,T;(L^2(0,2\pi))^3)$, the following identity holds.
\begin{align*}
\blue \int_{0}^{T}\ip{(\rho(t,\cdot),u(t,\cdot),\theta(t,\cdot))^{\dagger}}{(f(t,\cdot),g(t,\cdot),h(t,\cdot))^{\dagger}}_{L^2\times L^2\times L^2}\\
=\ip{(\rho_0(\cdot),u_0(\cdot),\theta_0(\cdot))^{\dagger}}{(\sigma(0,\cdot),v(0,\cdot),\vphi(0,\cdot))^{\dagger}}_{L^2\times L^2\times L^2}+R\bar{\theta}\int_{0}^{T}\left[\bar{u}\overline{\sigma(t,2\pi)}+\bar{\rho}\overline{v(t,2\pi)}\right]p(t)dt,
\end{align*}
where $(\sigma,v,\vphi)$ is the solution to the adjoint system \eqref{lcnse_adj_nb_s} with $(\sigma_T,v_T,\vphi_T)=(0,0,0)$.
\item For any given initial state $(\rho_0,u_0,\theta_0)\in H^1_{\per}(0,2\pi)\times(L^2(0,2\pi))^2$ and boundary control $q\in L^2(0,T)$, we say $(\rho,u,\theta)\in L^2(0,T;H^{-1}_{\blue\per}(0,2\pi))\times L^2(0,T;(L^2(0,2\pi))^2)$ is a solution to the system \eqref{lcnse_nb}-\eqref{in_cd_nb}-\eqref{bd_cd_nb2} if for any $(f,g,h)\in L^2(0,T;H^1_{\per}(0,2\pi))\times L^2(0,T;(L^2(0,2\pi))^2)$, the following identity holds.
\begin{align*}
\blue\int_{0}^{T}\ip{(\rho(t,\cdot),u(t,\cdot),\theta(t,\cdot))^{\dagger}}{(f(t,\cdot),g(t,\cdot),h(t,\cdot))^{\dagger}}_{H^{-1}_{\per}\times L^2\times L^2,H^1_{\per}\times L^2\times L^2}dt\\
=\ip{(\rho_0(\cdot),u_0(\cdot),\theta_0(\cdot))^{\dagger}}{(\sigma(0,\cdot),v(0,\cdot),\vphi(0,\cdot))^{\dagger}}_{L^2\times L^2\times L^2}\\
\blue+\bar{\rho}\int_{0}^{T}\left[R\bar{\theta}\overline{\sigma(t,2\pi)}+\lambda_0\bar{\rho}\overline{v_x(t,2\pi)}+\bar{\rho}\bar{u}\overline{v(t,2\pi)}+R\bar{\rho}\overline{\vphi(t,2\pi)}\right]q(t)dt,
\end{align*}
where $(\sigma,v,\vphi)$ is the solution to the adjoint system \eqref{lcnse_adj_nb_s} with $(\sigma_T,v_T,\vphi_T)=(0,0,0)$.
\item For any given initial state $(\rho_0,u_0,\theta_0)\in H^1_{\per}(0,2\pi)\times(L^2(0,2\pi))^2$ and boundary control $r\in L^2(0,T)$, we say $(\rho,u,\theta)\in L^2(0,T;H^{-1}_{\per}(0,2\pi))\times L^2(0,T;(L^2(0,2\pi))^2)$ is a solution to the system \eqref{lcnse_nb}-\eqref{in_cd_nb}-\eqref{bd_cd_nb2} if for any $(f,g,h)\in L^2(0,T;H^1_{\per}(0,2\pi))\times L^2(0,T;(L^2(0,2\pi))^2)$, the following identity holds.
\begin{align*}
\blue\int_{0}^{T}\ip{(\rho(t,\cdot),u(t,\cdot),\theta(t,\cdot))^{\dagger}}{(f(t,\cdot),g(t,\cdot),h(t,\cdot))^{\dagger}}_{H^{-1}_{\per}\times L^2\times L^2,H^1_{\per}\times L^2\times L^2}dt\\
=\ip{(\rho_0(\cdot),u_0(\cdot),\theta_0(\cdot))^{\dagger}}{(\sigma(0,\cdot),v(0,\cdot),\vphi(0,\cdot))^{\dagger}}_{L^2\times L^2\times L^2}\\
\blue+\bar{\rho}^2\int_{0}^{T}\left[R\overline{v(t,2\pi)}+\frac{c_0\bar{u}}{\bar{\theta}}\overline{\vphi(t,2\pi)}+\frac{c_0\kappa_0}{\bar{\theta}}\overline{\vphi_x(t,2\pi)}\right]r(t)dt,
\end{align*}
where $(\sigma,v,\vphi)$ is the solution to the adjoint system \eqref{lcnse_adj_nb_s} with $(\sigma_T,v_T,\vphi_T)=(0,0,0)$.
\end{itemize}
\end{definition}
\begin{proposition}\label{prop_ex_sol_nb1}
For any given initial state $(\rho_0,u_0,\theta_0)\in (L^2(0,2\pi))^3$ and boundary control $p\in L^2(0,T)$, the system \eqref{lcnse_nb}-\eqref{in_cd_nb}-\eqref{bd_cd_nb1} admits a unique solution $(\rho,u,\theta)$ in the space
\begin{equation*}
\mc{C}^0([0,T];L^2(0,2\pi))\times[\mc{C}^0([0,T];L^2(0,2\pi))\cap L^2(0,T;H^1_{\per}(0,2\pi))]^2.%\times[\mc{C}^0([0,T];L^2(0,2\pi))\cap L^2(0,T;H^1_{\per}(0,2\pi))].
\end{equation*}
\end{proposition}
\begin{proposition}\label{prop_ex_sol_nb2}
For any given initial state $(\rho_0,u_0,\theta_0)\in\blue H^1_{\per}(0,2\pi)\times(L^2(0,2\pi))^2$ and boundary control $q\in L^2(0,T)$, the system \eqref{lcnse_nb}-\eqref{in_cd_nb}-\eqref{bd_cd_nb2} admits a unique solution $(\rho,u,\theta)$ in the space
\begin{equation*}
\blue\mc{C}^0([0,T];H^{-1}_{\per}(0,2\pi))\times[\mc{C}^0([0,T];H^{-1}_{\per}(0,2\pi))\cap L^2(0,T;(L^2(0,2\pi)))]^2.
\end{equation*}
%Moreover, the operator $q\mapsto(\rho,u,\theta)$ is linear and continuous from $L^2(0,T)$ into $L^2(0,T;(H^1_{\per}(0,2\pi))^{\prime})\times L^2(0,T;(L^2(0,2\pi))^2)$.
\end{proposition}
\begin{proposition}\label{prop_ex_sol_nb3}
For any given initial state $(\rho_0,u_0,\theta_0)\in\blue H^1_{\per}(0,2\pi)\times(L^2(0,2\pi))^2$ and boundary control $r\in L^2(0,T)$, the system \eqref{lcnse_nb}-\eqref{in_cd_nb}-\eqref{bd_cd_nb3} admits a unique solution $(\rho,u,\theta)$ in the space
\begin{equation*}
\blue\mc{C}^0([0,T];H^{-1}_{\per}(0,2\pi))\times[\mc{C}^0([0,T];H^{-1}_{\per}(0,2\pi))\cap L^2(0,T;(L^2(0,2\pi)))]^2.
\end{equation*}
\end{proposition}
The proofs of \Cref{prop_ex_sol_nb1}, \Cref{prop_ex_sol_nb2} and \Cref{prop_ex_sol_nb3} can be done in a similar way (\cite[Theorem 2.4]{Bhandari22} and {\cite{Chowdhury13,Girinon08}}) like the barotropic case and so we skip the proofs.

\subsection{Spectral Analysis of $A^*$}
{\blue Let $\sigma(A^*)$ denotes the spectrum of the operator $A^*$.} We first write the following lemma.
\begin{lemma}\label{lemma_eigen_nb}
The following statements hold.
\begin{enumerate}[(i)]
\item $\ker(A^*)=\text{span}\left\{\cv{-1}{1}{1},\cv{1}{-1}{1},\cv{1}{1}{-1}\right\}$.
\item $\sup\left\{\Re(\nu)\ :\ \nu\in\sigma(A^*),\ \nu\neq0\right\}<0$.
\item The spectrum of $A^*$ consists of the eigenvalue $0$ and three branches of complex eigenvalues $$\blue\{\nu_n^h,\nu_n^{p_1},\nu_n^{p_2}\}_{n\in\mb{Z}^*}$$ with the asymptotic expressions given as
\begin{align}
\nu_n^h&=\bar{u}in-\bar{\omega}+O(\mod{n}^{-2}),\\%\label{exp_ev_h_nb}\\
\nu_n^{p_1}&=-\lambda_0n^2+\bar{u}in+O(1),\\%\label{exp_ev_p1_nb}\\
\nu_n^{p_2}&=-\kappa_0n^2+\bar{u}in+O(1),%\label{exp_ev_p2_nb},
\end{align}
for all $\mod{n}$ large, where $\bar{\omega}=\frac{R\bar{\theta}}{\lambda_0}$.
\item The eigenfunctions of $A^*$ corresponding to $\nu_n^h$ and $\nu_n^{p_1},\nu_n^{p_2}$ are respectively
\begin{equation}
\Phi_n^h=\cv{\xi_n^h}{\eta_n^h}{\zeta_n^h}=\cv{\alpha_1^n}{\alpha_2^n}{\alpha_3^n}e^{inx},\quad \Phi_n^{p_1}=\cv{\xi_n^{p_1}}{\eta_n^{p_1}}{\zeta_n^{p_1}}=\cv{\beta_1^n}{\beta_2^n}{\beta_3^n}e^{inx},\quad \Phi_n^{p_2}=\cv{\xi_n^{p_2}}{\eta_n^{p_2}}{\zeta_n^{p_2}}=\cv{\gamma_1^n}{\gamma_2^n}{\gamma_3^n}e^{inx},
\end{equation}
for all $n\in\mb{Z}^*$, with the constants $\alpha_i^n$, $\beta_i^n$ and $\gamma_i^n$ ($i=1,2,3$) given as
\begin{equation}
\begin{cases}
\alpha_1^n=R\bar{\rho},\ \ \alpha_2^n=-R(\bar{u}-\nu_3^n),\ \ \alpha_3^n=(\lambda_0in+\bar{u}-\nu_3^n)(\bar{u}-\nu_3^n)-R\bar{\theta}\\[4pt]
\beta_1^n=-\frac{R\bar{\rho}}{\bar{u}-\nu_1^n},\ \ \beta_2^n=R,\ \ \beta_3^n=\frac{1}{\bar{u}-\nu_1^n}[R\bar{\theta}-(\lambda_0in+\bar{u}-\nu_1^n)(\bar{u}-\nu_1^n)]\\[4pt]
\gamma_1^n=(\lambda_0in+\bar{u}-\nu_2^n)(\kappa_0in+\bar{u}-\nu_2^n)-\frac{R^2\bar{\theta}}{c_0},\ \ \gamma_2^n=-\frac{R\bar{\theta}}{\bar{\rho}}(\kappa_0in+\bar{u}-\nu_2^n),\ \ \gamma_3^n=\frac{R^2\bar{\theta}^2}{\bar{\rho}c_0},
\end{cases}
\end{equation}
for all $n\in\mb{Z}^*$, where $\nu_1^n,\nu_2^n$ and $\nu_3^n$ are roots of the cubic polynomial
\begin{align}
\nu^3-[(\lambda_0+\kappa_0)in+3\bar{u}]\nu^2-[\lambda_0\kappa_0n^2-2(\lambda_0+\kappa_0)\bar{u}in-3\bar{u}^2+\frac{R^2\bar{\theta}}{c_0}+R\bar{\theta}]\nu\\
+\lambda_0\kappa_0\bar{u}n^2-(\lambda_0+\kappa_0)\bar{u}^2in-\bar{u}^3+\frac{R^2\bar{\theta}}{c_0}\bar{u}+R\bar{\theta}\kappa_0in+R\bar{\theta}\bar{u}=0,\notag
\end{align}
for all $n\in\mb{Z}^*$.
%\item The eigenfunctions $\{\Phi_n^h,\ \Phi_n^{p_1},\ \Phi_n^{p_2}\ : \ n\in\mb{Z}^*\}$ of $A^*$ forms a Riesz basis of $(\dot{L}^2(0,2\pi))^3$.
\end{enumerate}
\end{lemma}
\begin{remark}\label{rem_exp_coef}
We have the asymptotic expressions of $\alpha_i^n,\beta_i^n,\gamma_i^n$, $i=1,2,3$ as follows.
\begin{equation}\label{asy_exp_coef}
\begin{cases}
\alpha_1^n\sim_{+\infty}1,\ \ \alpha_2^n\sim_{+\infty}\frac{1}{\mod{n}},\ \ \alpha_3^n\sim_{+\infty}\frac{1}{\mod{n}},\\[4pt]
\beta_1^n\sim_{+\infty}\frac{1}{\mod{n}},\ \ \beta_2^n\sim_{+\infty}1,\ \ \beta_3^n\sim_{+\infty}\frac{1}{\mod{n}},\\[4pt]
\gamma_1^n\sim_{+\infty}\frac{1}{\mod{n}},\ \ \gamma_2^n\sim_{+\infty}\frac{1}{\mod{n}},\ \ \gamma_3^n\sim_{+\infty}1.
\end{cases}
\end{equation}
\end{remark}
\begin{proof}
{\blue We will prove each parts separately.}

\smallskip

\begin{enumerate}
\item[\blue Part-(i).] Follows immediately from the fact that $A^*(\xi,\eta,\zeta)^{\dagger}=0$ implies $(\xi,\eta,\zeta)=$ constant.
\item[\blue Part-(ii).] Let $\Phi=(\xi,\eta,\zeta)^{\dagger}\in\mc{D}(A^*)$ be the eigenfunction of $A^*$ corresponding to the eigenvalue $\nu\neq0$. Then, we have
\begin{equation*}
\ip{A^*\cv{\xi}{\eta}{\zeta}}{\cv{\xi}{\eta}{\zeta}}_{\blue L^2\times L^2\times L^2}=\ip{\nu\cv{\xi}{\eta}{\zeta}}{\cv{\xi}{\eta}{\zeta}}_{\blue L^2\times L^2\times L^2},
\end{equation*}
that is,
\begin{align*}
&R\bar{\theta}\bar{u}\int_{0}^{2\pi}\overline{\xi(x)}\xi_x(x)dx+R\bar{\theta}\bar{\rho}\int_{0}^{2\pi}\overline{\xi(x)}\eta_x(x)dx+\lambda_0\bar{\rho}^2\int_{0}^{2\pi}\overline{\eta(x)}\eta_{xx}(x)dx\\
&\quad\quad+\bar{\rho}^2\bar{u}\int_{0}^{2\pi}\overline{\eta(x)}\eta_x(x)dx+R\bar{\theta}\bar{\rho}\int_{0}^{2\pi}\xi_x(x)\overline{\eta(x)}dx+R\bar{\rho}^2\int_{0}^{2\pi}\overline{\eta(x)}\zeta_x(x)dx\\
&\quad\quad+\frac{\bar{\rho}^2c_0}{\bar{\theta}}\kappa_0\int_{0}^{2\pi}\overline{\zeta(x)}\zeta_{xx}(x)dx+\frac{\bar{\rho}^2c_0}{\bar{\theta}}\bar{u}\int_{0}^{2\pi}\overline{\zeta(x)}\zeta_x(x)dx+R\bar{\rho}^2\int_{0}^{2\pi}\eta_x(x)\overline{\zeta(x)}dx\\
&\quad\quad\quad\quad=\blue\nu R\bar{\theta}\int_{0}^{2\pi}\mod{\xi(x)}^2dx+\nu\bar{\rho}^2\int_{0}^{2\pi}\mod{\eta(x)}^2dx+\nu\frac{\bar{\rho}^2c_0}{\bar{\theta}}\int_{0}^{2\pi}\mod{\zeta(x)}^2dx.
\end{align*}
An integration by parts yields
\begin{equation*}
\blue\Re(\nu)=-\frac{\lambda_0\bar{\rho}^2\norm{\eta_x}_{L^2(0,2\pi)}^2+\frac{\bar{\rho}^2c_0}{\bar{\theta}}\kappa_0\norm{\zeta_x}_{L^2(0,2\pi)}^2}{R\bar{\theta}\norm{\xi}_{L^2(0,2\pi)}^2+\bar{\rho}^2\norm{\eta}_{L^2(0,2\pi)}^2+\frac{\bar{\rho}^2c_0}{\bar{\theta}}\norm{\zeta}_{L^2(0,2\pi)}^2}<0,
\end{equation*}
which proves part (ii).

\item[\blue Parts (iii)-(iv).] We denote
\begin{equation*}
\vphi_n(x):=e^{inx},\ \ n\in\mb{Z}.
\end{equation*}
Then the set $\left\{\cv{\vphi_n}{0}{0},\cv{0}{\vphi_n}{0},\cv{0}{0}{\vphi_n}\right\}$ forms an orthogonal basis of $(L^2(0,2\pi))^3$. Let us define
\begin{equation*}
E_n:=\begin{pmatrix}\vphi_n&0&0\\[4pt]0&\vphi_n&0\\0&0&\vphi_n\end{pmatrix},\ \ \text{and}\ \Phi_n:=(\xi_n,\eta_n,\zeta_n)^{\dagger},
\end{equation*}
for all $n\in\mb{Z}$. Then, we have the following relation
\begin{equation}
A^*E_n\Phi_n=inE_nR_n\Phi_n,\ \ n\in\mb{Z},
\end{equation}
where
\begin{equation}
R_n:=\begin{pmatrix}\bar{u}&\bar{\rho}&0\\[4pt]\frac{R\bar{\theta}}{\bar{\rho}}&\lambda_0in+\bar{u}&R\\[4pt]0&\frac{R\bar{\theta}}{c_0}&\kappa_0in+\bar{u}\end{pmatrix},\ \ n\in\mb{Z}.
\end{equation}
Thus, if $(\alpha_n,\nu_n)$ is an eigenpair of $R_n$, then $(E_n\alpha_n,in\nu_n)$ will be an eigenpair of $A^*$. Therefore, it's remains to find the eigenvalues and eigenvectors of the matrix $R_n$ for $n\in\mb{Z}$. The characteristics equation of $R_n$ is
\begin{align}\label{ch_pol_rn_nb}
\nu^3-[(\lambda_0+\kappa_0)in+3\bar{u}]\nu^2-[\lambda_0\kappa_0n^2-2(\lambda_0+\kappa_0)\bar{u}in-3\bar{u}^2+\frac{R^2\bar{\theta}}{c_0}+R\bar{\theta}]\nu\\
+\lambda_0\kappa_0\bar{u}n^2-(\lambda_0+\kappa_0)\bar{u}^2in-\bar{u}^3+\frac{R^2\bar{\theta}}{c_0}\bar{u}+R\bar{\theta}\kappa_0in+R\bar{\theta}\bar{u}=0,\notag
\end{align}
for all $n\in\mb{Z}$. 

\smallskip

\noindent\textit{Claim 1.} $0$ cannot be a root of the polynomial \eqref{ch_pol_rn_nb} for any $n\in\mb{Z}$.

\smallskip

\noindent\textit{Proof of Claim 1.} Let $\nu=0$ be a root of \eqref{ch_pol_rn_nb}. Then, there exists some $n\in\mb{Z}$ such that
\begin{equation*}
\lambda_0\kappa_0\bar{u}n^2-(\lambda_0+\kappa_0)\bar{u}^2in-\bar{u}^3+\frac{R^2\bar{\theta}}{c_0}\bar{u}+R\bar{\theta}\kappa_0in+R\bar{\theta}\bar{u}=0,
\end{equation*}
which implies
\begin{equation*}
\lambda_0\kappa_0n^2-\bar{u}^2+\frac{R^2\bar{\theta}}{c_0}+R\bar{\theta}=0,\ \text{and}\ (\lambda_0+\kappa_0)\bar{u}^2=R\bar{\theta}\kappa_0.
\end{equation*}
We then have
\begin{equation*}
\lambda_0\kappa_0n^2=\bar{u}^2-\frac{R^2\bar{\theta}}{c_0}-R\bar{\theta}=\bar{u}^2-\frac{R^2\bar{\theta}}{c_0}-\left(\frac{\lambda_0}{\kappa_0}+1\right)\bar{u}^2=-\frac{R^2\bar{\theta}}{c_0}-\frac{\lambda_0}{\kappa_0}\bar{u}^2<0,
\end{equation*}
a contradiction. This proves our first claim.

\smallskip

\noindent\textit{Claim 2.} $\bar{u}$ cannot be a root of the polynomial \eqref{ch_pol_rn_nb} for any $n\in\mb{Z}^*$.

\smallskip

\noindent\textit{Proof of Claim 2.} Observe that $\bar{u}$ is a root of \eqref{ch_pol_rn_nb} if and only if $R\bar{\theta}\kappa_0in=0$. Thus, for all $n\in\mb{Z}^*$, $\bar{u}$ cannot be a root of \eqref{ch_pol_rn_nb}, which proves our second claim.

\smallskip

\noindent For fixed $n\in\mb{Z}^*$, let $\nu_1^n,\nu_2^n$ and $\nu_3^n$ be the roots of this cubic polynomial. The relation between roots and coefficients are
\begin{equation*}
\begin{cases}
\nu_1^n+\nu_2^n+\nu_3^n=(\lambda_0+\kappa_0)in+3\bar{u}\\[4pt]
\nu_1^n\nu_2^n+\nu_2^n\nu_3^n+\nu_3^n\nu_1^n=-[\lambda_0\kappa_0n^2-2(\lambda_0+\kappa_0)\bar{u}in-3\bar{u}^2+\frac{R^2\bar{\theta}}{c_0}+R\bar{\theta}]\\[4pt]
\nu_1^n\nu_2^n\nu_3^n=-[\lambda_0\kappa_0\bar{u}n^2-(\lambda_0+\kappa_0)\bar{u}^2in-\bar{u}^3+\frac{R^2\bar{\theta}}{c_0}\bar{u}+R\bar{\theta}\kappa_0in+R\bar{\theta}\bar{u}].
\end{cases}
\end{equation*}
We will find the asymptotic expressions of roots of the cubic polynomial \eqref{ch_pol_rn_nb} for large values of $\mod{n}$. The first relation between roots and coefficients tells us that $\bar{u}$ is present in at least one of the roots of the cubic polynomial \eqref{ch_pol_rn_nb}. Thus, using the transformation
\begin{equation}%\label{transformation}
\nu=\bar{u}+\epsilon_n,
\end{equation}
it is enough to find the roots of the transformed cubic equation in $\epsilon_n$
\begin{equation}\label{tr_cubic_pol}
\epsilon_n^3-(\lambda_0+\kappa_0)in\epsilon_n^2-\left(\lambda_0\kappa_0n^2+\frac{R^2\bar{\theta}}{c_0}+R\bar{\theta}\right)\epsilon_n+R\bar{\theta}\kappa_0in=0
\end{equation}
for all $n\in\mb{Z}^*$. We use the transformation $\epsilon_n=in\tilde{\epsilon}_n$, for $n\in\mb{Z}^*$, to simplify the above equation and we get
\begin{equation}%\label{tr_cubic_pol_1}
\tilde{\epsilon}_n^3-(\lambda_0+\kappa_0)\tilde{\epsilon}_n^2+\left(\lambda_0k_0+\frac{1}{n^2}\left(\frac{R^2\bar{\theta}}{c_0}+R\bar{\theta}\right)\right)\tilde{\epsilon}_n-\frac{R\bar{\theta}\kappa_0}{n^2}=0
\end{equation}
for all $n\in\mb{Z}^*$. We now use Rouche's Theorem to find the roots of this polynomial. Let us first state the Rouch\'e's Theorem, the proof of which can be found in \cite{Conway78}.
\begin{theorem}[Rouch\'e's Theorem]
Let $\Omega\subset\mb{C}$ be an open connected set and $f,g:\Omega\to\mb{C}$ be holomorphic on $\Omega$. Suppose there exists $a\in\Omega$ and $R>0$ such that $\overline{B(a,R)}\subset\Omega$ and $$\mod{g(z)-f(z)}<\mod{g(z)}\ \text{for all}\ z\in\partial B(a,R),$$ then $f$ and $g$ have the same number of zeros inside $B(a,R)$.
\end{theorem}
\noindent Let $n\in\mb{Z}^*$. We define the functions $f,g:\mb{C}\to\mb{C}$ by
\begin{equation*}
f(z):=z^3-(\lambda_0+\kappa_0)z^2+\left(\lambda_0k_0+\frac{1}{n^2}\left(\frac{R^2\bar{\theta}}{c_0}+R\bar{\theta}\right)\right)z-\frac{R\bar{\theta}\kappa_0}{n^2}
\end{equation*}
and
\begin{equation*}
g(z):=z^3-(\lambda_0+\kappa_0)z^2+\lambda_0k_0z
\end{equation*}
for all $z\in\mb{C}$. %Note that, $0$ cannot be a root of the function $f$. 
The roots of $g$ are $0,\lambda_0$ and $\kappa_0$. We choose ${\blue R_0}:=\frac{1}{2}\min\{\lambda_0,\kappa_0,\mod{\lambda_0-\kappa_0}\}$. Then, we have the following estimates
\begin{align*}
\mod{g(z)-f(z)}&=\mod{\frac{1}{n^2}\left(\frac{R^2\bar{\theta}}{c_0}+R\bar{\theta}\right)z-\frac{R\bar{\theta}\kappa_0}{n^2}}\\
&\leq\frac{C}{n^2}(\mod{z}+1)\begin{cases}=\frac{C}{n^2}({\blue R_0}+1),\ \ \text{for all}\ z\in\partial B(0,{\blue R_0}),\\[4pt]\leq\frac{C}{n^2}(\lambda_0+{\blue R_0}+1),\ \ \text{for all}\ z\in\partial B(\lambda_0,{\blue R_0}),\\[4pt]\leq\frac{C}{n^2}(\kappa_0+{\blue R_0}+1),\ \ \text{for all}\ z\in\partial B(\kappa_0,{\blue R_0}),\end{cases}
\end{align*}
for all $n\in\mb{Z}^*$. On the other hand, the choice of $\blue R_0$ tells us that the function $g$ does not have any root on the sets $\partial B(0,{\blue R_0})$, $\partial B(\lambda_0,{\blue R_0})$ and $\partial B(\kappa_0,{\blue R_0})$. Therefore, $\inf_{\mod{z}={\blue R_0}}\mod{g(z)}>0$,$\inf_{\mod{z-\lambda_0}={\blue R_0}}\mod{g(z)}>0$ and $\inf_{\mod{z-\kappa_0}={\blue R_0}}\mod{g(z)}>0$. This implies, for $\mod{n}$ large enough, we have
\begin{align*}
\mod{g(z)-f(z)}&<\mod{g(z)}\ \text{for all}\ z\in\partial B(0,{\blue R_0})\cup\partial B(\lambda_0,{\blue R_0})\cup\partial B(\kappa_0,{\blue R_0}).
\end{align*}
Thus, the function $f$ has a unique root inside each of the sets $B(0,{\blue R_0})$, $B(\lambda_0,{\blue R_0})$ and $B(\kappa_0,{\blue R_0})$. We denote these roots by $z_1^n$, $z_2^n$ and $z_3^n$ respectively. We now find asymptotic expressions of these roots.

\smallskip

\noindent\underline{\textit{Asymptotic expression of $z_1^n$.}} Since $z_1^n\in B(0,{\blue R_0})$, we have
\begin{equation*}
z_1^n=\frac{1}{(z_1^n-\lambda_0)(z_1^n-\kappa_0)}\left(\frac{R\bar{\theta}\kappa_0}{n^2}-\frac{1}{n^2}\left(\frac{R^2\bar{\theta}}{c_0}+R\bar{\theta}\right)z_1^n\right)
\end{equation*}
and therefore
\begin{equation*}
\mod{z_1^n}\leq\frac{1}{\mod{z_1^n-\lambda_0}\mod{z_1^n-\kappa_0}}\left(\mod{\frac{R\bar{\theta}\kappa_0}{n^2}}+\mod{\frac{1}{n^2}\left(\frac{R^2\bar{\theta}}{c_0}+R\bar{\theta}\right)z_1^n}\right)\leq\frac{C}{\mod{n}^2}
\end{equation*}
for $\mod{n}$ large enough. To find the asymptotic expression of $z_1^n$, we write $f(z_1^n)=0$ in the following way
\begin{align*}
z_1^n&=\frac{R\bar{\theta}\kappa_0}{n^2}\left(\lambda_0\kappa_0-(\lambda_0+\kappa_0)z_1^n+(z_1^n)^2+\frac{1}{n^2}\left(\frac{R\bar{\theta}}{c_0}+R\bar{\theta}\right)\right)^{-1}\\
&=\frac{R\bar{\theta}\kappa_0}{n^2}\frac{1}{\lambda_0\kappa_0}\left(1-\frac{(\lambda_0+\kappa_0)}{\lambda_0\kappa_0}z_1^n+\frac{1}{\lambda_0\kappa_0n^2}\left(\frac{R\bar{\theta}}{c_0}+R\bar{\theta}\right)+O(\mod{n}^{-4})\right)^{-1}\\
&=\frac{\bar{\omega}}{n^2}\left(1+\frac{(\lambda_0+\kappa_0)}{\lambda_0\kappa_0}z_1^n-\frac{1}{\lambda_0\kappa_0n^2}\left(\frac{R\bar{\theta}}{c_0}+R\bar{\theta}\right)+O(\mod{n}^{-4})\right)\\
&=\frac{\bar{\omega}}{n^2}+O(\mod{n}^{-4}),
\end{align*}
since $\mod{z_1^n}\leq\frac{C}{n^2}$ for all $\mod{n}$ large.

\smallskip

\noindent\underline{\textit{Asymptotic expression of $z_2^n$.}} Since $z_2^n\in B(\lambda_0,{\blue R_0})$, we have
\begin{equation*}
z_2^n-\lambda_0=\frac{1}{z_2^n(z_2^n-\kappa_0)}\left(\frac{R\bar{\theta}\kappa_0}{n^2}-\frac{1}{n^2}\left(\frac{R^2\bar{\theta}}{c_0}+R\bar{\theta}\right)z_1^n\right)
\end{equation*}
and therefore
\begin{equation*}
\mod{z_2^n-\lambda_0}\leq\frac{1}{\mod{z_2^n}\mod{z_2^n-\kappa_0}}\left(\mod{\frac{R\bar{\theta}\kappa_0}{n^2}}+\mod{\frac{1}{n^2}\left(\frac{R^2\bar{\theta}}{c_0}+R\bar{\theta}\right)z_1^n}\right)\leq\frac{C}{\mod{n}^2}
\end{equation*}
for $\mod{n}$ large enough. Thus, we can write
\begin{equation*}
z_2^n=\lambda_0+O(\mod{n}^{-2})
\end{equation*}
for all $\mod{n}$ large. 

\smallskip

\noindent\underline{\textit{Asymptotic expression of $z_3^n$.}} Following the similar approach as mentioned above, we can get
\begin{equation*}
z_3^n=\kappa_0+O(\mod{n}^{-2})
\end{equation*}
for all $\mod{n}$ large. 

\smallskip 

\noindent Combining all of the above, we obtain the asymptotic expressions of the roots of \eqref{tr_cubic_pol} as
\begin{align*}
\epsilon_1^n:=\lambda_0in+O(\mod{n}^{-1}),\\
\epsilon_2^n:=\kappa_0in+O(\mod{n}^{-1}),\\
\epsilon_3^n:=-\frac{\bar{\omega}}{in}+O(\mod{n}^{-3})
\end{align*}
for all $\mod{n}$ large. Therefore, eigenvalues of the matrix $R_n$ are $\nu_1^n,\nu_2^n$ and $\nu_3^n$ with the asymptotic expressions
\begin{align}
\nu_1^n&=\lambda_0in+\bar{u}+O(\mod{n}^{-1}),\label{exp_nu1n}\\
\nu_2^n&=\kappa_0in+\bar{u}+O(\mod{n}^{-1}),\label{exp_nu2n}\\
\nu_3^n&=\bar{u}-\frac{\bar{\omega}}{in}+O(\mod{n}^{-3}),\label{exp_nu3n}
\end{align}
for all $\mod{n}$ large.

\smallskip

\noindent To find the eigenvectors of the matrix $R_n$, we now consider the equation
\begin{equation*}
R_n\alpha_n=\nu_3^n\alpha_n,\ \ n\in\mb{Z}^*,
\end{equation*}
where $\alpha_n=(\alpha_1^n,\alpha_2^n,\alpha_3^n)^{\dagger}$, that is,
\begin{align}
(\bar{u}-\nu_3^n)\alpha_1^n+\bar{\rho}\alpha_2^n&=0,\\%\label{ev_rn_nb_11}\\
\frac{R\bar{\theta}}{\bar{\rho}}\alpha_1^n+(\lambda_0in+\bar{u}-\nu_3^n)\alpha_2^n+R\alpha_3^n&=0,\\%\label{ev_rn_nb_12}\\
\frac{R\bar{\theta}}{c_0}\alpha_2^n+(\kappa_0in+\bar{u}-\nu_3^n)\alpha_3^n&=0,%\label{ev_rn_nb_13},
\end{align}
for all $n\in\mb{Z}^*$. One solution is given by
\begin{equation*}
\alpha_1^n=R\bar{\rho},\ \ \alpha_2^n=-R(\bar{u}-\nu_3^n),\ \ \alpha_3^n=(\lambda_0in+\bar{u}-\nu_3^n)(\bar{u}-\nu_3^n)-R\bar{\theta},\ \ n\in\mb{Z}^*.
\end{equation*}
We next consider the equation
\begin{equation*}
R_n\beta_n=\nu_1^n\beta_n,\ \ n\in\mb{Z}^*,
\end{equation*}
where $\beta_n=(\beta_1^n,\beta_2^n,\beta_3^n)^{\dagger}$, that is,
\begin{align}
(\bar{u}-\nu_1^n)\beta_1^n+\bar{\rho}\beta_2^n&=0,\\%\label{ev_rn_nb_21}\\
\frac{R\bar{\theta}}{\bar{\rho}}\beta_1^n+(\lambda_0in+\bar{u}-\nu_1^n)\beta_2^n+R\beta_3^n&=0,\\%\label{ev_rn_nb_22}\\
\frac{R\bar{\theta}}{c_0}\beta_2^n+(\kappa_0in+\bar{u}-\nu_1^n)\beta_3^n&=0,%\label{ev_rn_nb_23},
\end{align}
for all $n\in\mb{Z}^*$. One solution is given by
\begin{equation*}
\beta_1^n=-\frac{R\bar{\rho}}{\bar{u}-\nu_1^n},\ \ \beta_2^n=R,\ \ \beta_3^n=\frac{1}{\bar{u}-\nu_1^n}[R\bar{\theta}-(\lambda_0in+\bar{u}-\nu_1^n)(\bar{u}-\nu_1^n)],\ \ n\in\mb{Z}^*.
\end{equation*}
We finally consider the equation
\begin{equation*}
R_n\gamma_n=\nu_2^n\gamma_n,\ \ n\in\mb{Z}^*,
\end{equation*}
where $\gamma_n=(\gamma_1^n,\gamma_2^n,\gamma_3^n)^{\dagger}$, that is,
\begin{align}
(\bar{u}-\nu_2^n)\gamma_1^n+\bar{\rho}\gamma_2^n&=0,\\%\label{ev_rn_nb_31}\\
\frac{R\bar{\theta}}{\bar{\rho}}\gamma_1^n+(\lambda_0in+\bar{u}-\nu_2^n)\gamma_2^n+R\gamma_3^n&=0,\\%\label{ev_rn_nb_32}\\
\frac{R\bar{\theta}}{c_0}\gamma_2^n+(\kappa_0in+\bar{u}-\nu_2^n)\gamma_3^n&=0,%\label{ev_rn_nb_33},
\end{align}
for all $n\in\mb{Z}^*$. One solution is given by
\begin{equation*}
\quad\gamma_1^n=(\lambda_0in+\bar{u}-\nu_2^n)(\kappa_0in+\bar{u}-\nu_2^n)-\frac{R^2\bar{\theta}}{c_0},\ \ \gamma_2^n=-\frac{R\bar{\theta}}{\bar{\rho}}(\kappa_0in+\bar{u}-\nu_2^n),\ \ \gamma_3^n=\frac{R^2\bar{\theta}^2}{\bar{\rho}c_0},
\end{equation*}
for $n\in\mb{Z}^*$. Therefore, the eigenvectors of $R_n$ corresponding to the eigenvalues $\nu_3^n$, $\nu_1^n$ and $\nu_2^n$ are respectively $\alpha_n,\beta_n$ and $\gamma_n$, where
\begin{equation*}
\alpha_n=\cv{\alpha_1^n}{\alpha_2^n}{\alpha_3^n},\ \ \beta_n=\cv{\beta_1^n}{\beta_2^n}{\beta_3^n},\ \ \gamma_n=\cv{\gamma_1^n}{\gamma_2^n}{\gamma_3^n},
\end{equation*}
for all $n\in\mb{Z}^*$. Hence, the eigenvalues of the operator $A^*$ are $\nu_n^h:=in\nu_n^3,\nu_n^{p_1}:=in\nu_n^2$ and $\nu_n^{p_2}:=in\nu_n^1$ for all $n\in\mb{Z}^*$ with the asymptotic expressions
\begin{align*}
\nu_n^h&=\bar{u}in-\bar{\omega}+O(\mod{n}^{-1}),\\
\nu_n^{p_1}&=-\lambda_0n^2+\bar{u}in+O(1),\\
\nu_n^{p_2}&=-\kappa_0n^2+\bar{u}in+O(1),
\end{align*}
for $\mod{n}$ large enough and the corresponding eigenfunctions are
\begin{equation*}
\quad\quad\Phi_n^h(x):=E_n(x)\alpha_n=\alpha_ne^{inx},\ \ \Phi_n^{p_1}(x):=E_n(x)\beta_n=\beta_ne^{inx},\ \ \Phi_n^{p_2}(x):=E_n(x)\gamma_n=\gamma_ne^{inx},
\end{equation*}
for all $n\in\mb{Z}^*$ and $x\in(0,2\pi)$. 
\end{enumerate}
This completes the proof.
\end{proof}
\begin{remark}\label{rem_ev_nb}
Note that, all the eigenvalues of $A^*$ are simple at least for $\mod{n}$ large enough. Depending on the constants $\bar{\rho},\bar{u},\bar{\theta},\lambda_0,\kappa_0, R$ and $c_0$, there may be multiple eigenvalues, but that would be only finitely many of them. {\blue For example, if we take $\bar{\rho}=\bar{u}=\lambda_0=1$ and $R\bar{\theta}=\frac{R^2\bar{\theta}}{c_0}=\frac{1}{2},\kappa_0=2$, then the characteristics equation \eqref{ch_pol_rn_nb} of $R_n$ (with $n=1$) becomes
\begin{equation*}
\nu^3-(3i+3)\nu^2+6in\nu+2-2i=0
\end{equation*}
and therefore $\nu=1+i$ is a root of multiplicity $3$, and consequently $-1+i$ is an eigenvalue of $A^*$ with algebraic multiplicity $3$. In this case, the proof of null controllability of the system \eqref{lcnse_nb} will be similar to the barotropic case (Section \ref{sec_multi_ev_b}) and for the sake of completeness, we will give a brief proof in this (non-barotropic) case also.% at the end of this section.

Furthermore, there can exist (finitely many) multiple eigenvalues for different values of $n$. For example, if we take $\bar{\rho}=\bar{u}=\lambda_0=1=R\bar{\theta}=\frac{R^2\bar{\theta}}{c_0}=1$ and $\kappa_0=2$, then $\nu=-1$ is an eigenvalue of $A^*$ with $n=1$ and $n=-1$, that is, $\nu_1=\nu_{-1}=-1$. Indeed, the polynomial equation \eqref{ch_pol_rn_nb} for $n=1$ and $n=-1$ becomes
\begin{align*}
&\nu^3-(3i+3)\nu^2-(1-6i)\nu+3-i=0,\\
&\nu^3-(-3i+3)\nu^2-(1+6i)\nu+3+i=0,
\end{align*}
and the root of which are $i$ and $-i$ respectively. In this case, as mentioned in the barotropic case, we have two independent eigenfunctions corresponding to this eigenvalue; as a consequence, the adjoint system \eqref{lcnse_adj_nb} fails to satisfy the unique continuation property; see Section \ref{sec_lac_app_nb} for more details.
}

%Therefore, in this case also, we assume that all the eigenvalues of $A^*$ are simple.
\end{remark}
{\blue
We denote the set of generalized eigenfunctions of $A^*$ by $\left\{\Phi_j,\ \tilde{\Phi}_j\right\}$ corresponding to the eigenvalue $\nu_j$. Then, one can have the following result:
\begin{proposition}\label{prop_rb_nb}
There exists a positive integer $N$ (large enough) such that the set of generalized eigenfunctions
$$\mc{E}(A^*):=\left\{\Phi_n^h,\ \Phi_n^{p_1},\ \Phi_n^{p_2}\ :\ \mod{n}>N;\ \Phi_j,\tilde{\Phi}_j\ :\ 0\leq\mod{j}\leq N\right\}$$
forms a Riesz basis in $(L^2(0,2\pi))^3$. In particular, if all the eigenvalues of $A^*$ are simple, then the set of eigenfunctions
$$\left\{\Phi_n^h,\ \Phi_n^{p_1},\ \Phi_n^{p_2}\ :\ n\in\mb{Z}^*\right\}$$
forms a Riesz basis in $(\dot{L}^2(0,2\pi))^3$.
\end{proposition}
\begin{proof}
In view of the proof of Proposition \ref{prop_rb_b}, it is enough to find an orthogonal basis of $(L^2(0,2\pi))^3$ that is quadratically close to the set of generalized eigenfunctions of $A^*$. One obvious choice is the following orthogonal basis 
$$\left\{\Psi_n(x):=\cv{R\bar{\rho}}{0}{0}e^{inx},\ \tilde{\Psi}_n(x):=\cv{0}{R}{0}e^{inx},\ \tilde{\tilde{\Psi}}_n(x):=\cv{0}{0}{\frac{R^2\bar{\theta}^2}{\bar{\rho}c_0}}e^{inx}\ :\ n\in\mb{Z}\right\}.$$ Indeed, we have for large enough $N\in\mb{N}$
\begin{equation*}
\sum_{\mod{n}>N}\left(\norm{\Phi_n^h-\Psi_n}_{(L^2(0,2\pi))^3}+\norm{\Phi_n^{p_1}-\tilde{\Psi}_n}_{(L^2(0,2\pi))^3}+\norm{\Phi_n^{p_2}-\tilde{\tilde{\Psi}}_n}_{(L^2(0,2\pi))^3}\right)\leq C\sum_{\mod{n}>N}\frac{1}{\mod{n}^2}<\infty,
\end{equation*}
thanks to Remark \ref{rem_exp_coef}. This completes the proof.
\end{proof}

}
\subsection{Observation estimates}
As mentioned in the barotropic case (Section \ref{sec_obs_est_b}), we need lower bound estimates of certain observation terms associated to the system \eqref{lcnse_nb}. First, 
%For any eigenvalue $\nu$, let us denote the corresponding eigenfunction of $A^*$ by $\Phi_{\nu}$, {\blue the generalized eigenfunction by $\tilde{\Phi}_{\nu}$} and let $\mc{E}(A^*)$ be the set of all generalized eigenfunctions of $A^*$. 
we define the observation operator corresponding to the system \eqref{lcnse_nb} as follows {\blue (see the Definition \ref{def_sol_nb})}:
\begin{itemize}
\item The observation operator $\mc{B}^*_{\rho}:\mc{D}(A^*)\to\mb{C}$ to the system \eqref{lcnse_nb}-\eqref{in_cd_nb}-\eqref{bd_cd_nb1} is defined by
\begin{equation}\label{obs_op_den_nb}
\mc{B}_{\rho}^*\Phi:=\bar{u}\xi(2\pi)+\bar{\rho}\eta(2\pi),\ \ \text{for}\ \Phi=(\xi,\eta)\in\mc{D}(A^*).
\end{equation}
\item The observation operator $\mc{B}^*_{\rho}:\mc{D}(A^*)\to\mb{C}$ to the system \eqref{lcnse_nb}-\eqref{in_cd_nb}-\eqref{bd_cd_nb1} is defined by
\begin{equation}\label{obs_op_vel_nb}
\mc{B}_u^*\Phi:=R\bar{\theta}\xi(2\pi)+\bar{\rho}\bar{u}\eta(2\pi)+\lambda_0\bar{\rho}\eta_x(2\pi)+R\bar{\rho}\zeta(2\pi),\ \ \text{for}\ \Phi=(\xi,\eta)\in\mc{D}(A^*).
\end{equation}
\item The observation operator $\mc{B}^*_{\rho}:\mc{D}(A^*)\to\mb{C}$ to the system \eqref{lcnse_nb}-\eqref{in_cd_nb}-\eqref{bd_cd_nb1} is defined by
\begin{equation}
\mc{B}_{\theta}^*\Phi:=R\eta(2\pi)+\frac{\bar{c_0}\bar{u}}{\bar{\theta}}\zeta(2\pi)+\frac{c_0\kappa_0}{\bar{\theta}}\zeta_x(2\pi),\ \ \text{for}\ \Phi=(\xi,\eta)\in\mc{D}(A^*).
\end{equation}
\end{itemize}

%\begin{align}
%&\mc{B}_{\rho}^*\Phi:=\bar{u}\xi(2\pi)+\bar{\rho}\eta(2\pi),\label{obs_op_den_nb}\\
%&\mc{B}_u^*\Phi:=R\bar{\theta}\xi(2\pi)+\bar{\rho}\bar{u}\eta(2\pi)+\lambda_0\bar{\rho}\eta_x(2\pi)+R\bar{\rho}\zeta(2\pi),\label{obs_op_vel_nb}\\
%&\mc{B}_{\theta}^*\Phi:=R\eta(2\pi)+\frac{\bar{c_0}\bar{u}}{\bar{\theta}}\zeta(2\pi)+\frac{c_0\kappa_0}{\bar{\theta}}\zeta_x(2\pi),%\label{obs_op_temp_nb}
%\end{align}
%for all $\Phi=(\xi,\eta,\zeta)^{\dagger}\in\mc{E}(A^*)$. 
\noindent Recall that $\mc{E}(A^*)$ denotes the set of all (generalized) eigenfunctions of $A^*$. Then, we have the following estimates:
\begin{lemma}\label{lem_obs_est_nb}
For all eigenfunction $\Phi_{\blue\nu}\in\mc{E}(A^*)\setminus\{\Phi_0,\blue\tilde{\Phi}_{0,1},\tilde{\Phi}_{0,2}\}$, the observation operators satisfies $\mc{B}_{\rho}^*\Phi_{\blue\nu}\neq0$ and $\mc{B}_u^*\Phi_{\blue\nu}\neq0$. Moreover, we have the following estimates
\begin{align}
&\mod{\mc{B}_{\rho}^*\Phi_n^h}\geq C,\ \ \mod{\mc{B}_{\rho}^*\Phi_n^{p_1}}\geq C,\ \ \mod{\mc{B}_{\rho}^*\Phi_n^{p_2}}\geq C,\label{obs_est_d_nb}\\
&\mod{\mc{B}_u^*\Phi_n^h}\geq \frac{C}{\mod{n}},\ \ \mod{\mc{B}_u^*\Phi_n^{p_1}}\geq C\mod{n},\mod{\mc{B}_u^*\Phi_n^{p_2}}\geq C,\label{obs_est_v_nb}\\
&\mod{\mc{B}_{\theta}^*\Phi_n^h}\geq \frac{C}{\mod{n}},\ \ \mod{\mc{B}_u^*\Phi_n^{p_2}}\geq C\ \ \mod{\mc{B}_{\rho}^*\Phi_n^{p_2}}\geq C\mod{n},%\label{obs_est_t_nb},
\end{align}
for some $C>0$ and all $n\in\mb{Z}^*$.
\end{lemma}
\begin{proof}
Recall from the proof of Lemma \ref{lemma_eigen_nb} that $\nu_1^n,\nu_2^n,\nu_3^n\neq0$ (Claim 1) for all $n\in\mb{Z}^*$ {\blue and the eigenvectors $(\alpha_1^n,\alpha_2^n,\alpha_3^n)^{\dagger}$, $(\beta_1^n,\beta_2^n,\beta_3^n)^{\dagger}$ and $(\gamma_1^n,\gamma_2^n,\gamma_3^n)^{\dagger}$ of $R_n$ satisfies the following relations:
\begin{align}
&(\bar{u}-\nu_3^n)\alpha_1^n+\bar{\rho}\alpha_2^n=0,\ \ (\bar{u}-\nu_1^n)\beta_1^n+\bar{\rho}\beta_2^n=0,\ \ (\bar{u}-\nu_2^n)\gamma_1^n+\bar{\rho}\gamma_2^n=0;\label{ev_rn_nb_1}\\
&\frac{R\bar{\theta}}{\bar{\rho}}\alpha_1^n+(\lambda_0in+\bar{u}-\nu_3^n)\alpha_2^n+R\alpha_3^n=0,\ \ \frac{R\bar{\theta}}{\bar{\rho}}\beta_1^n+(\lambda_0in+\bar{u}-\nu_1^n)\beta_2^n+R\beta_3^n=0,\label{ev_rn_nb_2}\\
&\hspace{3cm}\frac{R\bar{\theta}}{\bar{\rho}}\gamma_1^n+(\lambda_0in+\bar{u}-\nu_2^n)\gamma_2^n+R\gamma_3^n=0;\notag\\
&\frac{R\bar{\theta}}{c_0}\alpha_2^n+(\kappa_0in+\bar{u}-\nu_3^n)\alpha_3^n=0,\ \ \frac{R\bar{\theta}}{c_0}\beta_2^n+(\kappa_0in+\bar{u}-\nu_1^n)\beta_3^n=0,\label{ev_rn_nb_3}\\
&\hspace{3cm}\frac{R\bar{\theta}}{c_0}\gamma_2^n+(\kappa_0in+\bar{u}-\nu_2^n)\gamma_3^n=0,\notag
\end{align}
for all $n\in\mb{Z}^*$.}

\smallskip

\noindent We now consider the following cases:

\smallskip

\noindent\textit{Case 1. (Control acts in density)} We have

\begin{align*}
\mc{B}_{\rho}^*\Phi_n^h&=\bar{u}\xi_n^h(2\pi)+\bar{\rho}\eta_n^h(2\pi)=\bar{u}\alpha_1^n+\bar{\rho}\alpha_2^n=\nu_3^n\alpha_1^n\neq0,\\
\mc{B}_{\rho}^*\Phi_n^{p_1}&=\bar{u}\xi_n^{p_1}(2\pi)+\bar{\rho}\eta_n^{p_1}(2\pi)=\bar{u}\beta_1^n+\bar{\rho}\beta_2^n=\nu_1^n\beta_1^n\neq0,\\
\mc{B}_{\rho}^*\Phi_n^{p_2}&=\bar{u}\xi_n^{p_2}(2\pi)+\bar{\rho}\eta_n^{p_2}(2\pi)=\bar{u}\gamma_1^n+\bar{\rho}\gamma_2^n=\nu_2^n\gamma_1^n\neq0,
\end{align*}
for all $n\in\mb{Z}^*$, thanks to the equation \eqref{ev_rn_nb_1}.

\smallskip

\noindent\textit{Case 2. (Control acts in velocity)} We have
\begin{align*}
\mc{B}_u^*\Phi_n^h&=R\bar{\theta}\xi_n^h(2\pi)+\lambda_0\bar{\rho}(\eta_n^h)_x(2\pi)+\bar{\rho}\bar{u}\eta_n^h(2\pi)+R\bar{\rho}\zeta_n^h(2\pi)\\
&=R\bar{\theta}\alpha_1^n+\lambda_0\bar{\rho}in\alpha_2^n+\bar{\rho}\bar{u}\alpha_2^n+R\bar{\rho}\alpha_3^n=\bar{\rho}\nu_3^n\alpha_2^n\neq0,\\
\mc{B}_u^*\Phi_n^{p_1}&=R\bar{\theta}\xi_n^{p_1}(2\pi)+\lambda_0\bar{\rho}(\eta_n^{p_1})_x(2\pi)+\bar{\rho}\bar{u}\eta_n^{p_1}(2\pi)+R\bar{\rho}\zeta_n^{p_1}(2\pi)\\
&=R\bar{\theta}\beta_1^n+\lambda_0\bar{\rho}in\beta_2^n+\bar{\rho}\bar{u}\beta_2^n+R\bar{\rho}\beta_3^n=\bar{\rho}\nu_1^n\beta_2^n\neq0,\\
\mc{B}_u^*\Phi_n^{p_2}&=R\bar{\theta}\xi_n^{p_2}(2\pi)+\lambda_0\bar{\rho}(\eta_n^{p_2})_x(2\pi)+\bar{\rho}\bar{u}\eta_n^{p_2}(2\pi)+R\bar{\rho}\zeta_n^{p_2}(2\pi)\\
&=R\bar{\theta}\gamma_1^n+\lambda_0\bar{\rho}in\gamma_2^n+\bar{\rho}\bar{u}\gamma_2^n+R\bar{\rho}\gamma_3^n=\bar{\rho}\nu_2^n\gamma_2^n\neq0,
\end{align*}
for all $n\in\mb{Z}^*$, thanks to the equation \eqref{ev_rn_nb_2}.

\smallskip

\noindent\textit{Case 3. (Control acts in temperature)} We have
\begin{align*}
\mc{B}_{\theta}^*\Phi_n^h&=R\eta_n^h(2\pi)+\frac{c_0\bar{u}}{\bar{\theta}}\zeta_n^h(2\pi)+\frac{c_0\kappa_0}{\bar{\theta}}(\zeta_n^h)_x(2\pi)\\
&=R\alpha_2^n+\frac{c_0}{\bar{\theta}}(\bar{u}+\kappa_0in)\alpha_3^n=\frac{c_0}{\bar{\theta}}\nu_3^n\alpha_3^n\neq0,\\
\mc{B}_{\theta}^*\Phi_n^{p_1}&=R\eta_n^{p_1}(2\pi)+\frac{c_0\bar{u}}{\bar{\theta}}\zeta_n^{p_1}(2\pi)+\frac{c_0\kappa_0}{\bar{\theta}}(\zeta_n^{p_1})_x(2\pi)\\
&=R\beta_2^n+\frac{c_0}{\bar{\theta}}(\bar{u}+\kappa_0in)\beta_3^n=\frac{c_0}{\bar{\theta}}\nu_1^n\beta_3^n\neq0,\\
\mc{B}_{\theta}^*\Phi_n^{p_2}&=R\eta_n^{p_2}(2\pi)+\frac{c_0\bar{u}}{\bar{\theta}}\zeta_n^{p_2}(2\pi)+\frac{c_0\kappa_0}{\bar{\theta}}(\zeta_n^{p_2})_x(2\pi)\\
&=R\gamma_2^n+\frac{c_0}{\bar{\theta}}(\bar{u}+\kappa_0in)\gamma_3^n=\frac{c_0}{\bar{\theta}}\nu_2^n\gamma_3^n\neq0,
\end{align*}
for all $n\in\mb{Z}^*$, thanks to the equation \eqref{ev_rn_nb_3}.

\smallskip
{\blue
\noindent The estimates on the observation terms follows directly from the asymptotic expressions \eqref{exp_nu1n}-\eqref{exp_nu2n}-\eqref{exp_nu3n} and Remark \ref{rem_exp_coef}.}
\end{proof}

{\blue 
\begin{remark}\label{rem_obs_gen_nb}
Similar to the barotropic case, we can choose the (finitely many) generalized eigenfunctions $\tilde{\Phi}_{\nu}\in\mc{E}(A^*)$ in such a way that $\mc{B}^*_{\rho}\tilde{\Phi}_{\nu},\ \mc{B}^*_u\tilde{\Phi}_{\nu}\neq0$.
\end{remark}
}

\subsection{Observability inequalities}
{\blue As mentioned in the barotropic case, we will write the observability inequalities in this case also, that will help us prove the null controllability of the system \eqref{lcnse_nb}. The proof is similar and so we skip the details.
\begin{theorem}\label{equiv_thm_den_nb}
Let $T>0$ be given. Then, the system \eqref{lcnse_nb}-\eqref{in_cd_nb}-\eqref{bd_cd_nb1} is null controllable at time $T$ in the space $(\dot{L}^2(0,2\pi))^3$ if and only if the inequality
\begin{equation}\label{obs_inq_nb_den}
\norm{(\sigma(0),v(0),\vphi(0))^{\dagger}}_{(\dot{L}^2(0,2\pi))^3}^2\leq C\int_{0}^{T}\mod{\bar{u}\sigma(t,2\pi)+\bar{\rho}v(t,2\pi)}^2dt
\end{equation}
holds for all solutions $(\sigma,v,\vphi)^{\dagger}$ of the adjoint system \eqref{lcnse_adj_nb} with terminal data $(\sigma_T,v_T,\vphi_T)^{\dagger}\in\mc{D}(A^*)$.
\end{theorem}
\begin{theorem}\label{equiv_thm_vel_nb}
Let $T>0$ be given. Then, the system \eqref{lcnse_nb}-\eqref{in_cd_nb}-\eqref{bd_cd_nb2} is null controllable at time $T$ in the space $\dot{H}^1_{\per}(0,2\pi)\times(\dot{L}^2(0,2\pi))^2$ if and only if the inequality
\begin{align}\label{obs_inq_nb_vel}
&\norm{(\sigma(0),v(0),\vphi(0))^{\dagger}}_{\dot{H}^{-1}_{\per}(0,2\pi)\times(\dot{L}^2(0,2\pi))^2}^2\\
&\quad\quad\quad\quad\quad\quad\leq C\int_{0}^{T}\mod{R\bar{\theta}\sigma(t,2\pi)+\lambda_0\bar{\rho}v_x(t,2\pi)+\bar{\rho}\bar{u}v(t,2\pi)+R\bar{\rho}\vphi(t,2\pi)}^2dt\notag
\end{align}
holds for all solutions $(\sigma,v,\vphi)^{\dagger}$ of the adjoint system \eqref{lcnse_adj_nb} with terminal data $(\sigma_T,v_T,\vphi_T)^{\dagger}\in\mc{D}(A^*)$.
\end{theorem}
\begin{theorem}%\label{equiv_thm_temp_nb}
Let $T>0$ be given. Then, the system \eqref{lcnse_nb}-\eqref{in_cd_nb}-\eqref{bd_cd_nb3} is null controllable at time $T$ in the space $\dot{H}^1_{\per}(0,2\pi)\times(\dot{L}^2(0,2\pi))^2$ if and only if the inequality
\begin{equation}\label{obs_inq_nb_temp}
\norm{(\sigma(0),v(0),\vphi(0))^{\dagger}}_{\dot{H}^{-1}_{\per}(0,2\pi)\times(\dot{L}^2(0,2\pi))^2}^2\leq C\int_{0}^{T}\mod{Rv(t,2\pi)+\frac{c_0\bar{u}}{\bar{\theta}}\vphi(t,2\pi)+\frac{c_0\kappa_0}{\bar{\theta}}\vphi_x(t,2\pi)}^2dt
\end{equation}
holds for all solutions $(\sigma,v,\vphi)^{\dagger}$ of the adjoint system \eqref{lcnse_adj_nb} with terminal data $(\sigma_T,v_T,\vphi_T)^{\dagger}\in\mc{D}(A^*)$.
\end{theorem}

To prove these inequalities, we require lower bound estimates of the observation terms (given in the right hand side of \eqref{obs_inq_nb_den},\eqref{obs_inq_nb_vel} and \eqref{obs_inq_nb_temp}) and to obtain these bounds, we will use the Ingham-type inequality \eqref{ingham-ineq}. So,} we first rewrite the eigenvalues as $\{\nu_n^h,\nu_n^p\}_{n\in\mb{Z}^*}$, where
\begin{equation*}
\nu_n^p=\begin{cases}\nu_k^{p_1},\ \ \text{if}\ n=2k-1,\ k\in\mb{Z}\\\nu_k^{p_2},\ \ \text{if}\ n=2k,\ k\in\mb{Z}^*,\end{cases}
\end{equation*}
for all $n\in\mb{Z}^*$ and $\nu_n^h$ is as defined earlier. We also denote {\blue the eigenfunction
\begin{equation*}
\Phi_n^p=
\begin{cases}
\Phi_k^{p_1},\ \ \text{if}\ n=2k-1,\ k\in\mb{Z},\\
\Phi_k^{p_2},\ \ \text{if}\ n=2k,\ k\in\mb{Z}^*,
\end{cases}
\end{equation*}
}
and the observation term
\begin{equation*}
\mc{B}^*\Phi_n^p=\begin{cases}\mc{B}^*\Phi_k^{p_1},\ \ \text{if}\ n=2k-1,\ k\in\mb{Z},\\\mc{B}^*\Phi_k^{p_2},\ \ \text{if}\ n=2k,\ k\in\mb{Z}^*,\end{cases}
\end{equation*}
for all $n\in\mb{Z}^*$. {\blue Also,} recall that we have defined the set
\begin{equation*}
\mc{S}:=\left\{(\lambda_0,\kappa_0)\ : \ \sqrt{\frac{\lambda_0}{\kappa_0}}\notin\mb{Q}\right\}.
\end{equation*}
%Then, the eigenvalues $(\nu_n^h)_{n\in\mb{Z}^*}$ satisfies hypotheses (H1)-(H2) with $\tau=\bar{u}, \beta=-\bar{\omega}$ and for all $(\lambda_0,\kappa_0)\in\mc{S}$, the sequence $(\nu_n^p)_{n\in\mb{Z^*}}$ satisfies hypotheses (P1)-(P2)-(P3)-(P4) of \Cref{lem_ingham}.

\smallskip

{\blue
Similar to the barotropic case, we first prove the null controllability results when all eigenvalues of $A^*$ are simple. The case when there exist generalized eigenfunctions corresponding to the finitely many multiple eigenvalues will be presented at the end of this section. Throughout the proof of null controllability of the system \eqref{lcnse_nb}, we will assume that all the eigenvalues of $A^*$ have geometric multiplicity $1$.

\subsubsection{The case of simple eigenvalues} Let $(\sigma_T,v_T,\vphi_T)^{\dagger}\in(\dot{L}^2(0,2\pi))^3$.} Since the {\blue set of} eigenfunctions $\blue\left\{\Phi_n^h,\Phi_n^{p_1},\Phi_n^{p_2}\ :\ n\in\mb{Z}^*\right\}$ of $A^*$ form a Riesz basis of $(\dot{L}^2(0,2\pi))^3$ {\blue(see Proposition \ref{prop_rb_nb})}, therefore any $(\sigma_T,v_T,\vphi_T)^{\dagger}\in (\dot{L}^2(0,2\pi))^3$ can be written as
\begin{equation*}
(\sigma_T,v_T,\vphi_T)^{\dagger}=\sum_{n\in\mb{Z}^*}a_n^h\Phi_n^h+\sum_{n\in\mb{Z}^*}a_n^{p_1}\Phi_n^{p_1}+\sum_{n\in\mb{Z}^*}a_n^{p_2}\Phi_n^{p_2},
\end{equation*}
for some $(a_n^h)_{n\in\mb{Z}^*},(a_n^{p_1})_{n\in\mb{Z}^*}, (a_n^{p_2})_{n\in\mb{Z}^*}\in\ell_2$. Then, the solution to the adjoint system \eqref{lcnse_adj_nb} is
\begin{equation*}
(\sigma(t,x),v(t,x),\vphi(t,x))^{\dagger}=\sum_{n\in\mb{Z}^*}a_n^he^{\nu_n^h(T-t)}\Phi_n^h(x)+\sum_{n\in\mb{Z}^*}a_n^{p_1}e^{\nu_n^{p_1}(T-t)}\Phi_n^{p_1}(x)+\sum_{n\in\mb{Z}^*}a_n^{p_2}e^{\nu_n^{p_2}(T-t)}\Phi_n^{p_2}(x),
\end{equation*}
for $(t,x)\in(0,T)\times(0,2\pi)$, that is,
\begin{align*}
\sigma(t,x)=\sum_{n\in\mb{Z}^*}a_n^he^{\nu_n^h(T-t)}\alpha_1^ne^{inx}+\sum_{n\in\mb{Z}^*}a_n^{p_1}e^{\nu_n^{p_1}(T-t)}\beta_1^ne^{inx}+\sum_{n\in\mb{Z}^*}a_n^{p_2}e^{\nu_n^{p_2}(T-t)}\gamma_1^ne^{inx},\\
v(t,x)=\sum_{n\in\mb{Z}^*}a_n^he^{\nu_n^h(T-t)}\alpha_2^ne^{inx}+\sum_{n\in\mb{Z}^*}a_n^{p_1}e^{\nu_n^{p_1}(T-t)}\beta_2^ne^{inx}+\sum_{n\in\mb{Z}^*}a_n^{p_2}e^{\nu_n^{p_2}(T-t)}\gamma_2^ne^{inx},\\
\vphi(t,x)=\sum_{n\in\mb{Z}^*}a_n^he^{\nu_n^h(T-t)}\alpha_3^ne^{inx}+\sum_{n\in\mb{Z}^*}a_n^{p_1}e^{\nu_n^{p_1}(T-t)}\beta_3^ne^{inx}+\sum_{n\in\mb{Z}^*}a_n^{p_2}e^{\nu_n^{p_2}(T-t)}\gamma_3^ne^{inx},
\end{align*}
for $(t,x)\in (0,T)\times(0,2\pi)$. {\blue We denote
\begin{equation*}
a_n^p=
\begin{cases}
a_k^{p_1},\ \ \text{if}\ n=2k-1,\ k\in\mb{Z},\\
a_k^{p_2},\ \ \text{if}\ n=2k,\ k\in\mb{Z}^*.
\end{cases}
\end{equation*}
Then, we can write
\begin{equation*}
(\sigma(t,x),v(t,x),\vphi(t,x))^{\dagger}=\sum_{n\in\mb{Z}^*}a_n^he^{\nu_n^h(T-t)}\Phi_n^h(x)+\sum_{n\in\mb{Z}^*}a_n^pe^{\nu_n^p(T-t)}\Phi_n^p(x),
\end{equation*}
for $(t,x)\in (0,T)\times(0,2\pi)$.

\smallskip

\noindent\underline{Estimates on the norms of $(\sigma(0),v(0),\vphi(0))^{\dagger}$:}
We have
\begin{align}\label{est_norm_nb1}
\norm{(\sigma(0),v(0),\vphi(0))^{\dagger}}_{(\dot{L}^2(0,2\pi))^3}^2&\leq C\left[\sum_{n\in\mb{Z}^*}\mod{a_n^h}^2e^{2\Re(\nu_n^h)T}\norm{\Phi_n^h}_{(\dot{L}^2(0,2\pi))^3}^2+\sum_{n\in\mb{Z}^*}\mod{a_n^p}^2e^{2\Re(\nu_n^p)T}\norm{\Phi_n^p}_{(\dot{L}^2(0,2\pi))^3}^2\right]\\
&\leq C\left[\sum_{n\in\mb{Z}^*}\mod{a_n^h}^2+\sum_{n\in\mb{Z}^*}\mod{a_n^p}^2e^{2\Re(\nu_n^p)T}\right],\notag
\end{align}
thanks to the asymptotic expressions \eqref{asy_exp_coef}. We also have
\begin{align}\label{est_norm_nb2}
&\norm{(\sigma(0),v(0),\vphi(0))^{\dagger}}_{\dot{H}^{-1}_{\per}(0,2\pi)\times (\dot{L}^2(0,2\pi))^2}^2\\
&\leq C\left[\sum_{n\in\mb{Z}^*}\mod{a_n^h}^2e^{2\Re(\nu_n^h)T}\norm{\Phi_n^h}_{\dot{H}^{-1}_{\per}(0,2\pi)\times (\dot{L}^2(0,2\pi))^2}^2+\sum_{n\in\mb{Z}^*}\mod{a_n^p}^2e^{2\Re(\nu_n^p)T}\norm{\Phi_n^p}_{\dot{H}^{-1}_{\per}(0,2\pi)\times (\dot{L}^2(0,2\pi))^2}^2\right]\notag\\
&\leq C\left[\sum_{n\in\mb{Z}^*}\mod{a_n^h}^2\frac{1}{\mod{n}^{2}}+\sum_{n\in\mb{Z}^*}\mod{a_n^p}^2e^{2\Re(\nu_n^p)T}\right]\notag,
\end{align}
thanks to the asymptotic expressions \eqref{asy_exp_coef}.}

\smallskip

\noindent{\blue
To prove our null controllability results (Theorem \ref{thm_null_nb}), we will use the Ingham-type inequality \eqref{ingham-ineq} and for that, we need to prove that the eigenvalues $(\nu_n^h)_{n\in\mb{Z}^*}$ and $(\nu_n^p)_{n\in\mb{Z}^*}$ satisfy all the hypotheses of Lemma \ref{lem_ingham}. Recall the asymptotic expressions of the eigenvalues, given by Lemma \ref{lemma_eigen_nb}:
\begin{align*}
\nu_n^h&=\bar{u}in-\bar{\omega}+O(\mod{n}^{-2}),\\
\nu_n^{p_1}&=-\lambda_0n^2+\bar{u}in+O(1),\\
\nu_n^{p_2}&=-\kappa_0n^2+\bar{u}in+O(1).
\end{align*}
\begin{itemize}
\item Due to our assumption on the eigenvalues, we have $\nu_n^h\neq\nu_l^h$, $\nu_n^p\neq\nu_l^p$ for all $n,l\in\mb{Z}^*$ with $n\neq l$ and $\{\nu_n^h\ ; \ n\in\mb{Z}^*\}\cap\{\nu_n^p\ ;\ n\in\mb{Z}^*\}=\emptyset$.
\item From the expression of $\nu_n^h$, it is easy to see that the family $(\nu_n^h)_{n\in\mb{Z}^*}$ satisfies hypothesis (H2) of Lemma \ref{lem_ingham} with $\beta=-\bar{\omega}, \tau=\bar{u}$ and $e_n=O(\mod{n}^{-2})$ for $\mod{n}$ large enough.
\item On the other hand, we have
\begin{equation*}
\frac{-\Re(\nu_n^p)}{\mod{\Im(\nu_n^p)}}=
\begin{cases}
\frac{\lambda_0k^2+O(1)}{\bar{u}\mod{k}},\ \ \text{if}\ n=2k-1,\ k\in\mb{Z},\\
\frac{\kappa_0k^2+O(1)}{\bar{u}\mod{k}},\ \ \text{if}\ n=2k,\ k\in\mb{Z}^*,
\end{cases}
\end{equation*}
and therefore $\frac{-\Re(\nu_n^p)}{\mod{\Im(\nu_n^p)}}\geq\min(\frac{\lambda_0}{\bar{u}},\frac{\kappa_0}{\bar{u}})$ for $\mod{n}$ large enough, which verifies hypothesis (P2) of Lemma \ref{lem_ingham}.
\item We also have for $\mod{n}$ large
\begin{equation*}
\lambda_0n^2\leq\mod{\nu_n^{p_1}}\leq(\lambda_0+\bar{u})n^2,\ \text{and}\ \kappa_0n^2\leq\mod{\nu_n^{p_1}}\leq(\kappa_0+\bar{u})n^2,
\end{equation*}
and therefore $(\nu_n^p)$ satisfies hypothesis (P4) of Lemma \ref{lem_ingham} for large enough $\mod{n}$.
\end{itemize}
The family $(\nu_n^h)$ satisfy hypotheses (H1)-(H2) of Lemma \ref{lem_ingham} for $\mod{n}$ large enough, and therefore one can have the hyperbolic Ingham inequality \eqref{hyper_ingham}. On the other hand, the parabolic branch $(\nu_n^p)_{n\in\mb{Z}^*}$ satisfy hypotheses (P1)-(P2) and (P4), but does not necessarily satisfy the gap condition (Hypothesis (P3) of Lemma \ref{lem_ingham}) when $\mod{n}$ is large enough. However, we can prove the existence of a biorthogonal family to $(e^{\nu_n^pt})_{n\in\mb{Z}^*}$ under the stronger assumption \eqref{cond_coeff_nb} on the coefficients $\lambda_0$ and $\kappa_0$; as a consequence we have the parabolic Ingham inequality \eqref{para_ingham} (thanks to Remark \ref{rem_ingham}).

\begin{lemma}%\label{lem_biortho}
Let us assume that all eigenvalues $(\nu_n^p)_{n\in\mb{Z}^*}$ are distinct. Then, under the assumption of Theorem \ref{thm_null_nb} and given $\epsilon>0$, there exists a sequence $(q_n)_{n\in\mb{Z}^*}\subset L^2(0,\infty)$ biorthogonal to the family $(e^{\nu_n^pt})_{n\in\mb{Z}^*}$ with the following estimate
\begin{equation}%\label{est_biortho}
\norm{q_n}_{L^2(0,\infty)}\leq K(\epsilon)e^{\Re(\nu_n^p)\epsilon}
\end{equation}
for all $n\in\mb{Z}^*$.
\end{lemma}
The proof of this Lemma can be done in a similar way as \cite[Lemma 3.1]{Cara10} and \cite[Lemma 2]{Luca13}, so we omit the details. Indeed, an easy calculation yields that
\begin{align*}
\mod{\nu_n^{p_1}-\nu_j^{p_1}}\geq C\mod{n^2-j^2},\ \ \mod{\nu_n^{p_1}-\nu_j^{p_2}}\geq C\mod{\lambda_0n^2-\kappa_0j^2},\\
\mod{\nu_n^{p_2}-\nu_j^{p_1}}\geq C\mod{\kappa_0n^2-\lambda_0j^2},\ \ \mod{\nu_n^{p_2}-\nu_j^{p_2}}\geq C\mod{n^2-j^2},
\end{align*}
for some $C>0$. With the help of this Lemma and the hyperbolic Ingham inequality \eqref{hyper_ingham}, we can have the combined Ingham-type inequality \eqref{ingham-ineq} (as mentioned in Remark \ref{rem_ingham}). With this, we are now ready to prove null controllability results of the system \eqref{lcnse_nb} in the case of simple eigenvalues.
}

\bigskip

\noindent\underline{Proof of Theorem \ref{thm_null_nb}-Part \eqref{null_dens_nb}}
{\blue
Let $T>\frac{2\pi}{\bar{u}}$. Recall from Theorem \ref{equiv_thm_den_nb} that it is enough to prove the observability inequality \eqref{obs_inq_nb_den}, that is,}
\begin{equation*}
\int_{0}^{T}\mod{\bar{u}\sigma(t,2\pi)+\bar{\rho}v(t,2\pi)}^2dt\geq C\norm{(\sigma(0),v(0),\vphi(0))^{\dagger}}_{(\dot{L}^2(0,2\pi))^3}^2,
\end{equation*}
for all $(\sigma_T,v_T,\vphi_T)^{\dagger}\in\blue\mc{D}(A^*)$. {\blue Also, recall the observation operator $\mc{B}^*_{\rho}$ given by \eqref{obs_op_den_nb}. Then,} we have the observation term
\begin{align*}
&\int_{0}^{T}\mod{\bar{u}\sigma(t,2\pi)+\bar{\rho}v(t,2\pi)}^2dt\\
&=\int_{0}^{T}\mod{\sum_{n\in\mb{Z}^*}a_n^h\mc{B}_{\rho}^*\Phi_n^he^{\nu_n^h(T-t)}+\sum_{n\in\mb{Z}^*}a_n^{p_1}\mc{B}_{\rho}^*\Phi_n^{p_1}e^{\nu_n^{p_1}(T-t)}+\sum_{n\in\mb{Z}^*}a_n^{p_2}\mc{B}_{\rho}^*\Phi_n^{p_2}e^{\nu_n^{p_2}(T-t)}}^2dt\\
&=\int_{0}^{T}\mod{\sum_{n\in\mb{Z}^*}a_n^h\mc{B}_{\rho}^*\Phi_n^he^{\nu_n^h(T-t)}+\sum_{n\in\mb{Z}^*}a_n^{p}\mc{B}_{\rho}^*\Phi_n^{p}e^{\nu_n^{p}(T-t)}}^2dt.
\end{align*}
Using the combined parabolic-hyperbolic Ingham type inequality \eqref{ingham-ineq} (Lemma \ref{lem_ingham}), we have
\begin{align*}
\int_{0}^{T}\mod{\bar{u}\sigma(t,2\pi)+\bar{\rho}v(t,2\pi)}^2dt&\geq C\left(\sum_{n\in\mb{Z}^*}\mod{a_n^h}^2\mod{\mc{B}_{\rho}^*\Phi_n^h}^2e^{2\Re(\nu_n^h)(T-t)}+\sum_{n\in\mb{Z}^*}\mod{a_n^{p}}^2\mod{\mc{B}_{\rho}^*\Phi_n^{p}}^2e^{2\Re(\nu_n^{p})T}\right)\\
&\geq C\left(\sum_{n\in\mb{Z}^*}\mod{a_n^h}^2+\sum_{n\in\mb{Z}^*}\mod{a_n^{p}}^2e^{2\Re(\nu_n^{p})T}\right),
\end{align*}
thanks to the estimate \ref{obs_est_d_nb}. This estimate together with the norm estimate \eqref{est_norm_nb1}, the observability inequality \eqref{obs_inq_nb_den} {\blue follows}. This completes the proof {\blue in the case of simple eigenvalues}.

\bigskip

\noindent\underline{Proof of Theorem \ref{thm_null_nb}-Part \eqref{null_vel_temp_nb}.}
Let $T>\frac{2\pi}{\bar{u}}$. {\blue We will consider only the velocity control case. The case when a control acts in temperature \eqref{bd_cd_nb3}, the proof will be similar (as we have similar lower bounds on the observation term $\mc{B}^*_u\Phi_n^h$ and $\mc{B}^*_{\theta}\Phi_n^h$, see Lemma \ref{lem_obs_est_nb} for instance), and so we omit the details. Thanks to Theorem \ref{equiv_thm_vel_nb}, it is enough to prove the observability inequality \eqref{obs_inq_nb_vel}, that is,}
\begin{equation*}
\int_{0}^{T}|R\bar{\theta}\sigma(t,2\pi)+\lambda_0\bar{\rho}v_x(t,2\pi)+\bar{\rho}\bar{u}v(t,2\pi)+R\bar{\rho}\vphi(t,2\pi)|^2dt\geq C\norm{(\sigma(0),v(0),\vphi(0))^{\dagger}}_{{\blue\dot{H}^{-1}_{\per}(0,2\pi)}\times (\dot{L}^2(0,2\pi))^2}^2
\end{equation*}
for all $(\sigma_T,v_T,\vphi_T)^{\dagger}\in\blue\mc{D}(A^*)$.
We have
\begin{align*}
&\int_{0}^{T}\mod{R\bar{\theta}\sigma(t,2\pi)+\lambda_0\bar{\rho}v_x(t,2\pi)+\bar{\rho}\bar{u}v(t,2\pi)+R\bar{\rho}\vphi(t,2\pi)}^2dt\\
&=\int_{0}^{T}\mod{\sum_{n\in\mb{Z^*}}a_n^h\mc{B}_u^*\Phi_n^he^{\nu_n^h(T-t)}+\sum_{n\in\mb{Z}^*}a_n^{p_1}\mc{B}_u^*\Phi_n^{p_1}e^{\nu_n^{p_1}(T-t)}+\sum_{n\in\mb{Z}^*}a_n^{p_2}\mc{B}_u^*\Phi_n^{p_2}e^{\nu_n^{p_2}(T-t)}}^2dt\\
&=\int_{0}^{T}\mod{\sum_{n\in\mb{Z}^*}a_n^h\mc{B}_u^*\Phi_n^he^{\nu_n^h(T-t)}+\sum_{n\in\mb{Z}^*}a_n^{p}\mc{B}_u^*\Phi_n^{p}e^{\nu_n^{p}(T-t)}}^2dt
\end{align*}
{\blue where $\mc{B}^*_u$ is defined in \eqref{obs_op_vel_nb}.} Using the combined parabolic-hyperbolic Ingham type inequality (Lemma \ref{lem_ingham}) and the observation estimates \eqref{obs_est_v_nb}, we obtain
\begin{align*}
&\int_{0}^{T}\mod{R\bar{\theta}\sigma(t,2\pi)+\lambda_0\bar{\rho}v_x(t,2\pi)+\bar{\rho}\bar{u}v(t,2\pi)+R\bar{\rho}\vphi(t,2\pi)}^2dt\\
&\geq C\left(\sum_{n\in\mb{Z}^*}\mod{a_n^h}^2\mod{\mc{B}_u^*\Phi_n^h}^2e^{2\Re(\nu_n^h)(T-t)}+\sum_{n\in\mb{Z}^*}\mod{a_n^{p}}^2\mod{\mc{B}_u^*\Phi_n^{p}}^2e^{2\Re(\nu_n^{p})T}\right)\\
&\geq C\left(\sum_{n\in\mb{Z}^*}\mod{a_n^h}^2\frac{1}{\mod{n}^2}+\sum_{n\in\mb{Z}^*}\mod{a_n^{p}}^2e^{2\Re(\nu_n^{p})T}\right).
\end{align*}
{\blue Thus, we obtain that
\begin{align*}
&\int_{0}^{T}\mod{R\bar{\theta}\sigma(t,2\pi)+\lambda_0\bar{\rho}v_x(t,2\pi)+\bar{\rho}\bar{u}v(t,2\pi)+R\bar{\rho}\vphi(t,2\pi)}^2dt\\
%&\geq C\left(\sum_{n\in\mb{Z}^*}\mod{a_n^h}^2\frac{1}{\mod{n}^{2s}}+\sum_{n\in\mb{Z}^*}\mod{a_n^{p}}^2e^{2\Re(\nu_n^{p})T}\right)\\
&\geq C\norm{(\sigma(0),v(0),\vphi(0))^{\dagger}}_{\dot{H}^{-1}_{\per}(0,2\pi)\times (\dot{L}^2(0,2\pi))^2}^2,
\end{align*}}
thanks to the estimate \eqref{est_norm_nb1}. This proves the observability inequality \eqref{obs_inq_nb_vel} {\blue and the proof is complete for simple eigenvalues}.

{\blue
\subsubsection{The case of multiple eigenvalues}
Throughout the proof, we assume that all the eigenvalues of $A^*$ have geometric multiplicity $1$. By changing the indices, let us assume (without loss of generality) that $\nu_j$ be the eigenvalues of $A^*$ with multiplicity $N_j$ for $j=1,2,\dots,m$ and for all $\mod{n}>m$, the eigenvalues $\nu_n$ of $A^*$ are simple. We denote the set of generalized eigenfunctions corresponding to $\nu_n$ (for $n=1,2,\dots,m$) as $\left\{\Phi_{n,1}=(\xi_{n,1},\eta_{n,1},\zeta_{n,1})\ ;\ \tilde{\Phi}_{n,j}=(\tilde{\xi}_{n,j},\tilde{\eta}_{n,j},\tilde{\zeta}_{n,j})\ :\ j=1,2\dots,N_n-1\right\}$. The proof of null controllability of the system \eqref{lcnse_nb} in the presence of multiple eigenvalue will be similar to the barotropic case, so we give a brief proof of Theorem \ref{thm_null_nb} in each cases (control acting in density, velocity and temperature).

\smallskip

\noindent\underline{Control in density.} Let $(\sigma_T,v_T,\vphi_T)^{\dagger}\in (\dot{L}^2(0,2\pi))^3$. We decompose it as follows:
\begin{equation}
(\sigma_T,v_T,\vphi_T)^{\dagger}=(\sigma_{T,1},v_{T,1},\vphi_{T,1})^{\dagger}+(\sigma_{T,2},v_{T,2},\vphi_{T,2})^{\dagger},
\end{equation}
with
\begin{equation*}
(\sigma_{T,1},v_{T,1},\vphi_{T,1})^{\dagger}=\sum_{n=1}^{m}\left(a_{n,1}\Phi_{n,1}+\sum_{j=1}^{N_n-1}\tilde{a}_{n,j}\tilde{\Phi}_{n,j}\right)
\end{equation*}
and
\begin{equation*}
(\sigma_{T,2},v_{T,2},\vphi_{T,2})^{\dagger}=\sum_{\mod{n}>m}\left(a_n^h\Phi_n^h+a_n^{p_1}\Phi_n^{p_1}+a_n^{p_2}\Phi_n^{p_2}\right).
\end{equation*}
Let $(\sigma_1,v_1,\vphi_1)$ and $(\sigma_2,v_2,\vphi_2)$ denote the solutions of the adjoint system \eqref{lcnse_adj_nb} associated to the terminal data $(\sigma_{T,1},v_{T,1},\vphi_{T,1})$ and $(\sigma_{T,2},v_{T,2},\vphi_{T,2})$ respectively. Then, we can write the solutions as
\begin{align}
(\sigma_1,v_1,\vphi_1)^{\dagger}=\sum_{n=1}^{m}e^{\nu_n(T-t)}\left(a_{n,1}\Phi_{n,1}+\sum_{j=1}^{N_n-1}(T-t)^j\tilde{a}_{n,j}\tilde{\Phi}_{n,j}\right),\\
(\sigma_2,v_2,\vphi_2)^{\dagger}=\sum_{\mod{n}>m}\left(a_n^he^{\nu_n^h}\Phi_n^h+a_n^{p_1}e^{\nu_n^{p_1}(T-t)}\Phi_n^{p_1}+a_n^{p_2}e^{\nu_n^{p_2}}\Phi_n^{p_2}\right).\label{exp_nb_2}
\end{align}
Using the observability inequality \eqref{obs_inq_nb_den} in the case of simple eigenvalues, we get
\begin{equation}%\label{inq_nb_1}
\int_{0}^{T}\mod{\bar{u}\sigma_2(t,2\pi)+\bar{\rho}v_2(t,2\pi)}^2dt\geq C\norm{(\sigma_2(0),v_2(0),\vphi_2(0))^{\dagger}}_{(\dot{L}^2(0,2\pi))^3}.
\end{equation}
Note that
\begin{equation}
\bar{u}\sigma_1(t,2\pi)+\bar{\rho}v_1(t,2\pi)=\sum_{n=1}^{m}e^{\nu_n(T-t)}\left(a_{n,1}\mc{B}^*_{\rho}\Phi_{n,1}+\sum_{j=1}^{N_n-1}\tilde{a}_{n,j}\mc{B}^*\tilde{\Phi}_{n,j}\right).
\end{equation}
Proceeding similarly as in the barotropic case and using the well-posedness result (Lemma \ref{well_posedness_adj_nb})
\begin{equation*}
\int_{0}^{\epsilon}\mod{\bar{u}\sigma_2(t,2\pi)+\bar{\rho}v_2(t,2\pi)}^2dt\leq C\norm{(\sigma_2(\epsilon),v_2(\epsilon),\vphi_2(\epsilon))^{\dagger}}_{(\dot{L}^2(0,2\pi))^3}^2,
\end{equation*}
($\epsilon>0$) and finite dimensional norm equivalence (thanks to Lemma \ref{lem_obs_est_nb}-Remark \ref{rem_obs_gen_nb}), we can add these finitely many terms in the above observability inequality to obtain
\begin{equation}\label{inq_nb_2}
\int_{0}^{T}\mod{\bar{u}\sigma(t,2\pi)+\bar{\rho}v(t,2\pi)}^2dt\geq C\norm{(\sigma_2(0),v_2(0),\vphi_2(0))^{\dagger}}_{(\dot{L}^2(0,2\pi))^3}^2.
\end{equation}
and
\begin{equation}\label{inq_nb_3}
\int_{0}^{T}\mod{\bar{u}\sigma(t,2\pi)+\bar{\rho}v(t,2\pi)}^2dt\geq C\norm{(\sigma_1(0),v_1(0),\vphi_1(0))^{\dagger}}_{(\dot{L}^2(0,2\pi))^3}^2.
\end{equation}
Combining these two inequalities \eqref{inq_nb_2} and \eqref{inq_nb_3}, we obtain the desired observability inequality \eqref{obs_inq_nb_den}, proving Theorem \ref{thm_null_nb}-Part (i) in the case of multiple eigenvalues. 

\bigskip

\noindent\underline{Control in velocity.} As mentioned in the barotropic case, it is enough to prove the following inequality:
\begin{align}\label{well_posed_vel_nb}
&\int_{0}^{\frac{\epsilon}{2}}\mod{R\bar{\theta}\sigma(t,2\pi)+\lambda_0\bar{\rho}v_x(t,2\pi)+\bar{\rho}\bar{u}v(t,2\pi)+R\bar{\rho}\vphi(t,2\pi)}^2dt\\
&\hspace{5cm}\leq C\norm{(\sigma_2(\epsilon),v_2(\epsilon),\vphi_2(\epsilon))^{\dagger}}_{\dot{H}^{-1}_{\per}(0,2\pi)\times(\dot{L}^2(0,2\pi))^2}^2.\notag
\end{align}
Recall that (equation \eqref{exp_nb_2})
\begin{equation*}
(\sigma_2(t),v_2(t),\vphi_2(t))^{\dagger}=\sum_{\mod{n}>m}\left(a_n^he^{\nu_n^h(T-t)}\Phi_n^h+a_n^{p_1}e^{\nu_n^{p_1}(T-t)}\Phi_n^{p_1}+a_n^{p_2}e^{\nu_n^{p_2}(T-t)}\Phi_n^{p_2}\right).
\end{equation*}
Since the observation term $\mc{B}^*_u\Phi_n^h$ have similar upper bound (of order $\frac{1}{n}$), proceeding similarly as in the barotropic case, we can obtain
\begin{align}\label{inq_nb_4}
\int_{0}^{\frac{\epsilon}{2}}\mod{R\bar{\theta}\sigma(t,2\pi)+\lambda_0\bar{\rho}v_x(t,2\pi)+\bar{\rho}\bar{u}v(t,2\pi)+R\bar{\rho}\vphi(t,2\pi)}^2dt\\
\leq C\sum_{\mod{n}>m}\frac{\mod{a_n^h}^2}{\mod{n}^2}+C\sum_{\mod{n}>m}\mod{a_n^p}^2e^{2\Re(\nu_n^p)(T-\epsilon)},\notag
\end{align}
see for instance the inequality \eqref{inq_4}. On the other hand, we compute
\begin{align*}
&\norm{(\sigma_2(\epsilon),v_2(\epsilon),\vphi_2(\epsilon))}_{\dot{H}^{-1}_{\per}(0,2\pi)\times(\dot{L}^2(0,2\pi))^2}\\
&=\sum_{\mod{n}>m}\bigg(\frac{R\bar{\theta}}{\mod{n}^2}\mod{a_n^he^{\nu_n^h(T-\epsilon)}\alpha_1^n+a_n^{p_1}e^{\nu_n^{p_1}(T-\epsilon)}\beta_1^n+a_n^{p_2}e^{\nu_n^{p_2}(T-\epsilon)}\gamma_1^n}^2\\
&\quad\quad\quad\quad\quad\quad+\bar{\rho}^2\mod{a_n^he^{\nu_n^h(T-\epsilon)}\alpha_2^n+a_n^{p_1}e^{\nu_n^{p_1}(T-\epsilon)}\beta_2^n+a_n^{p_2}e^{\nu_n^{p_2}(T-\epsilon)}\gamma_2^n}^2\\
&\quad\quad\quad\quad\quad\quad+\frac{\bar{\rho}^2c_0}{\bar{\theta}}\mod{a_n^he^{\nu_n^h(T-\epsilon)}\alpha_3^n+a_n^{p_1}e^{\nu_n^{p_1}(T-\epsilon)}\beta_3^n+a_n^{p_2}e^{\nu_n^{p_2}(T-\epsilon)}\gamma_3^n}^2\bigg).
\end{align*}
Thanks to Remark \ref{rem_exp_coef}, we can write for $\mod{n}>N$ (with $N$ large) that
\begin{align}%\label{inq_nb_5}
&\norm{(\sigma_2(\epsilon),v_2(\epsilon),\vphi_2(\epsilon))}_{\dot{H}^{-1}_{\per}(0,2\pi)\times(\dot{L}^2(0,2\pi))^2}\\
&\geq C\left(\sum_{\mod{n}>N}\frac{\mod{a_n^h}^2}{\mod{n}^2}+\sum_{\mod{n}>N}\mod{a_n^{p_1}}^2e^{2\Re(\nu_n^{p_1})(T-\epsilon)}+\sum_{\mod{n}>N}\mod{a_n^{p_2}}^2e^{2\Re(\nu_n^{p_2})(T-\epsilon)}\right)\notag\\
&=C\left(\sum_{\mod{n}>N}\frac{\mod{a_n^h}^2}{\mod{n}^2}+\sum_{\mod{n}>N}\mod{a_n^p}^2e^{2\Re(\nu_n^p)(T-\epsilon)}\right).\notag
\end{align}
Comparing this inequality with \eqref{inq_nb_4}, we deduce that
\begin{align*}
&\int_{0}^{\frac{\epsilon}{2}}\mod{R\bar{\theta}\sigma(t,2\pi)+\lambda_0\bar{\rho}v_x(t,2\pi)+\bar{\rho}\bar{u}v(t,2\pi)+R\bar{\rho}\vphi(t,2\pi)}^2dt\\
&\hspace{5cm}\leq C\norm{(\sigma_2(\epsilon),v_2(\epsilon),\vphi_2(\epsilon))}_{\dot{H}^{-1}_{\per}(0,2\pi)\times(\dot{L}^2(0,2\pi))^2},
\end{align*}
proving the required inequality \eqref{well_posed_vel_nb}.

\bigskip

\noindent\underline{Control in Temperature.} The proof will be similar to the velocity case (due to the similar bounds on the observation terms) and so we skip the details.

This concludes the proof of Theorem \ref{thm_null_nb} in the case of multiple eigenvalues.\qed
}

\subsection{Lack of null controllability for less regular initial states}
{\blue
Similar to the barotropic case, we first write the following result:
\begin{proposition}
Let $0\leq s<1$ and $T>0$ be given. Then, 
\begin{itemize}
\item the system \eqref{lcnse_nb}-\eqref{in_cd_nb}-\eqref{bd_cd_nb2} is null controllable at time $T$ in the space $\dot{H}^s_{\per}(0,2\pi)\times(\dot{L}^2(0,2\pi))^2$ if and only if the inequality
\begin{align}\label{obs_inq_nb_vel_s}
&\norm{(\sigma(0),v(0),\vphi(0))^{\dagger}}_{\dot{H}^{-s}_{\per}(0,2\pi)\times(\dot{L}^2(0,2\pi))^2}^2\\
&\quad\quad\quad\quad\leq C\int_{0}^{T}\mod{R\bar{\theta}\sigma(t,2\pi)+\lambda_0\bar{\rho}v_x(t,2\pi)+\bar{\rho}\bar{u}v(t,2\pi)+R\bar{\rho}\vphi(t,2\pi)}^2dt\notag
\end{align}
holds for all $(\sigma,v,\vphi)^{\dagger}$ of the adjoint system \eqref{lcnse_adj_nb} with $(\sigma_T,v_T,\vphi_T)^{\dagger}\in\mc{D}(A^*)$.
\item the system \eqref{lcnse_nb}-\eqref{in_cd_nb}-\eqref{bd_cd_nb3} is null controllable at time $T$ in the space $\dot{H}^s_{\per}(0,2\pi)\times(\dot{L}^2(0,2\pi))^2$ if and only if the inequality
\begin{equation}\label{obs_inq_nb_temp_s}
\norm{(\sigma(0),v(0),\vphi(0))^{\dagger}}_{\dot{H}^{-s}_{\per}(0,2\pi)\times(\dot{L}^2(0,2\pi))^2}^2\leq C\int_{0}^{T}\mod{Rv(t,2\pi)+\frac{c_0\bar{u}}{\bar{\theta}}\vphi(t,2\pi)+\frac{c_0\kappa_0}{\bar{\theta}}\vphi_x(t,2\pi)}^2dt
\end{equation}
holds for all solutions $(\sigma,v,\vphi)^{\dagger}$ of \eqref{lcnse_adj_nb} with $(\sigma_T,v_T,\vphi_T)^{\dagger}\in\mc{D}(A^*)$.
\end{itemize}

\end{proposition}
}
\subsubsection{Proof of Proposition \ref{prop_lac_null_nb}- Part \eqref{lac_vel_temp_nb}}
{\blue We will present the proof for velocity case only; the temperature case will be exactly similar, because the observation terms $B^*_{\theta}\Phi_n^h$ and $B^*_{u}\Phi_n^h$ have same upper bounds {(see \blue Lemma \ref{lem_obs_est_nb})}.} For ${\blue(\sigma_T^n,v_T^n,\vphi_T^n)^{\dagger}}=\Phi_n^h$, the solution to the adjoint system is
\begin{equation*}
{\blue(\sigma^n(t,x),v^n(t,x),\vphi^n(t,x))^{\dagger}}=e^{\nu_n^h(T-t)}\Phi_n^h(x),
\end{equation*}
for $(t,x)\in (0,T)\times(0,2\pi)$ {\blue and $n\in\mb{Z}^*$}. For large $n$, we have the following estimate
\begin{equation*}
\norm{\Phi_n^h}_{H^{-s}_{\per}(0,2\pi)\times (L^2(0,2\pi))^2}\geq\frac{C}{\mod{n}^s},
\end{equation*}
and therefore
\begin{equation*}
\norm{\blue(\sigma^n(0),v^n(0),\vphi^n(0))^{\dagger}}_{H^{-s}_{\per}(0,2\pi)\times (L^2(0,2\pi))^2}^2\geq \frac{C}{\mod{n}^{2s}}.
\end{equation*} 
We also have
\begin{equation*}
\blue\int_{0}^{T}\mod{R\bar{\theta}\sigma^n(t,2\pi)+\lambda_0\bar{\rho}v^n_x(t,2\pi)+\bar{\rho}\bar{u}v^n(t,2\pi)+R\bar{\rho}\vphi^n(t,2\pi)}^2dt\leq\frac{C}{\mod{n}^2},
\end{equation*}
{\blue for all $n\in\mb{Z}^*$ (see equation \eqref{inq_nb_4} for instance).} Thus if the observability inequality \eqref{obs_inq_nb_vel_s} holds, then we get
\begin{equation*}
\frac{C}{\mod{n}^{2s}}\leq\frac{C}{\mod{n}^2}\implies\mod{n}^{2-2s}\leq C,
\end{equation*}
which is not possible since $0\leq s<1$. This completes the proof.\qed

%\begin{remark}
%Since , proof of Proposition \ref{prop_lac_null_nb}{\blue-Part (ii)} will be similar to above, so we omit the details.
%\end{remark}

\subsection{Lack of controllability at small time}
The proof will be similar to the barotropic case, that is, the proof of {\blue Theorem \ref{thm_den_b}- Part (ii)}. For the sake of completeness, we give the proof below.
\subsubsection{Proof of Proposition \ref{prop_lac_null_nb}-Part \eqref{lac_dens_nb}}
Let $0<T<\frac{2\pi}{\bar{u}}$. Following the notations in the proof of {\blue Theorem \ref{thm_den_b}- Part (ii) (Section \ref{sec_lac_null_b})}, we consider the system
\begin{equation}\label{trnsprt_nb2}
\begin{dcases}
\tilde{\sigma}_t(t,x)+\bar{u}\tilde{\sigma}_x(t,x)-\bar{\omega}\tilde{\sigma}(t,x)=0,\ \ (t,x)\in (0,T)\times(0,2\pi),\\
\tilde{\sigma}(t,0)=\tilde{\sigma}(t,2\pi),\ \ t\in (0,T),\\
\tilde{\sigma}(T,x)=\tilde{\sigma}_T^N(x),\ \ x\in(0,2\pi).
\end{dcases}
\end{equation}
Since $\text{supp}(\tilde{\sigma}_T^N)\subset\text{supp}(\tilde{\sigma}_T)\subset(T,2\pi)$, the solution satisfies $\tilde{\sigma}^N(t,0)=\tilde{\sigma}^N(t,2\pi)=0$. We now consider the adjoint to our main system
\begin{equation}\label{lcnse_adj_nb1}
\begin{dcases}
-\sigma_t(t,x)-\bar{u}\sigma_x(t,x)-\bar{\rho}v_x(t,x)=0,\ &\text{in}\ (0,T)\times(0,2\pi),\\
-v_t(t,x)-\lambda_0v_{xx}(t,x)-\frac{R\bar{\theta}}{\bar{\rho}}\sigma_x(t,x)-\bar{u}v_x(t,x)-R\vphi_x(t,x)=0,\ &\text{in}\ (0,T)\times(0,2\pi),\\
-\vphi_t(t,x)-\kappa_0\vphi_{xx}(t,x)-\frac{R\bar{\theta}}{c_0}v_x(t,x)-\bar{u}\vphi_x(t,x)=0,\ &\text{in}\ (0,T)\times(0,2\pi),\\
\sigma(t,0)=\sigma(t,2\pi),\ \ v(t,0)=v(t,2\pi),\ \ v_x(t,0)=v_x(t,2\pi),\ \ &t\in (0,T),\\
\vphi(t,0)=\vphi(t,2\pi),\ \ \vphi_x(t,0)=\vphi_x(t,2\pi),\ \ &t\in (0,T),\\
\sigma(T,x)=\sigma_T^N(x),\ \ v(T,x)=v_T^N(x),\ \ \vphi(T,x)=\vphi_T^N(x),\ \ &x\in(0,2\pi),
\end{dcases}
\end{equation}
where we choose $v_T^N$ and $\vphi_T^N$ such that
\begin{equation*}
(\tilde{\sigma}_T^N,v_T^N,\vphi_T^N)^{\dagger}=\sum_{\mod{n}\geq N+1}\tilde{a}_n^h\Phi_n^h
\end{equation*}
with $\tilde{a}_n^h\alpha_1^n:=a_nP^N(n)$ for all $\mod{n}\geq N+1$. We write the solutions to the systems \eqref{trnsprt_nb2} and \eqref{lcnse_adj_nb1} respectively as
\begin{align}
\tilde{\sigma}^N(t,x)&=\sum_{\mod{n}\geq N+1}a_nP^N(n)e^{(\bar{u}in-\bar{\omega})(T-t)}e^{inx},\\
\sigma^N(t,x)&=\sum_{\mod{n}\geq N+1}a_nP^N(n)e^{\nu_n^h(T-t)}e^{inx},\\
v^N(t,x)&=\sum_{\mod{n}\geq N+1}a_nP^N(n)\frac{\beta_1^n}{\alpha_1^n}e^{\nu_n^h(T-t)}e^{inx},\\
\vphi^N(t,x)&=\sum_{\mod{n}\geq N+1}a_nP^N(n)\frac{\gamma_1^n}{\alpha_1^n}e^{\nu_n^h(T-t)}e^{inx},
\end{align}
for all $(t,x)\in[0,T]\times[0,2\pi]$. Similar to the barotropic case, we prove that the solution component $\sigma^N$ approximates the solution $\tilde{\sigma}^N$. Indeed, we have
\begin{align*}
\norm{\sigma^N(\cdot,x)-\tilde{\sigma}^N(\cdot,x)}_{L^2(0,T)}^2&\leq\sum_{\mod{n}\geq N+1}\mod{a_n}^2\mod{P^N(n)}^2\norm{e^{\nu_n^h(T-t)}-e^{(\bar{u}in-\bar{\omega})(T-t)}}_{L^2(0,T)}^2\\
&\leq\sum_{\mod{n}\geq N+1}\mod{a_n}^2\mod{P^N(n)}^2\norm{e^{(\bar{u}in-\bar{\omega})(T-t)}e^{O(\mod{n}^{-1})(T-t)}-1}_{L^2(0,T)}^2\\
&\leq\frac{C}{\mod{n}^2}\sum_{\mod{n}\geq N+1}\mod{a_n}^2\mod{P^N(n)}^2,
\end{align*}
for all $x\in[0,2\pi]$. We also have for all $x\in[0,2\pi]$
\begin{align*}
\norm{v^N(\cdot,x)}_{L^2(0,T)}^2&\leq\sum_{\mod{n}\geq N+1}\mod{a_n}^2\mod{P^N(n)}^2\frac{\mod{\beta_1^n}^2}{\mod{\alpha_1^n}^2}\norm{e^{\nu_n^h(T-t)}}_{L^2(0,T)}^2\\
&\leq C\sum_{\mod{n}\geq N+1}\mod{a_n}^2\mod{P^N(n)}^2\frac{1}{\mod{n}^2}\\
&\leq\frac{C}{\mod{N}^2}\sum_{\mod{n}\geq N+1}\mod{a_n}^2\mod{P^N(n)}^2
\end{align*}
We suppose that the following observability inequality holds
\begin{equation*}
\int_{0}^{T}\mod{\bar{u}\sigma^N(t,2\pi)+\bar{\rho}v^N(t,2\pi)}^2dt\geq C\norm{(\sigma^N(0),v^N(0),\vphi^N(0))^{\dagger}}_{{\blue(\dot{L}^2(0,2\pi))^3}}^2.
\end{equation*}
Then, we have
\begin{align*}
&\norm{(\sigma^N(0),v^N(0),\vphi^N(0))^{\dagger}}_{{\blue(\dot{L}^2(0,2\pi))^3}}^2\\
&\leq C\int_{0}^{T}\mod{\bar{u}\sigma^N(t,2\pi)+\bar{\rho}v^N(t,2\pi)}^2dt\\
&\leq C\int_{0}^{T}\left(\bar{u}^2\mod{(\sigma^N(t,2\pi)-\tilde{\sigma}^N(t,2\pi))}^2+\bar{u}^2\mod{\tilde{\sigma}^N(t,2\pi)}^2+\bar{\rho}^2\mod{v^N(t,2\pi)}^2\right)dt\\
&\leq\frac{C}{N^2}\sum_{\mod{n}\geq N+1}\mod{a_n}^2\mod{P^N(n)}^2,
\end{align*}
since $\tilde{\sigma}^N(t,0)=0=\tilde{\sigma}^N(t,2\pi)$ for all $t\in (0,T)$. Thus, we get
\begin{align*}
\norm{\sigma^N(0)}_{{\blue\dot{L}^2(0,2\pi)}}^2&\leq \norm{(\sigma^N(0),v^N(0),\vphi^N(0))^{\dagger}}_{{\blue(\dot{L}^2(0,2\pi))^3}}^2\\
&\leq\frac{C}{N^2}\sum_{\mod{n}\geq N+1}\mod{a_n}^2\mod{P^N(n)}^2\leq\frac{C}{N^2}\norm{\sigma^N(0)}_{{\blue\dot{L}^2(0,2\pi)}}^2,
\end{align*}
since $\Re(\nu_n^h)$ is bounded. Therefore, $1\leq\frac{C}{N^2}$ for all $N$ and hence the above inequality cannot hold. This is a contradiction and the proof is complete.\qed

{\blue
\subsection{Lack of approximate controllability}\label{sec_lac_app_nb}
In this section, we find existence of certain coefficients $\bar{\rho},\bar{u},\bar{\theta},\lambda_0,\kappa_0,R,c_0$ such that the system \eqref{lcnse_nb} is not approximately controllable at any time $T>0$ in $(L^2(0,2\pi))^2$ (that is, Proposition \ref{prop_lac_null_app_nb}). Full characterization of these coefficients is very difficult due to the cubic polynomial \eqref{ch_pol_rn_nb}. We present the proof of Proposition \ref{prop_lac_null_app_nb} in the case when there is a boundary control acting in density component. The proof will be similar in other cases (that is, when the control is acting in the velocity or temperature components) and so we omit the details.
	
\subsubsection{Proof of Proposition \ref{prop_lac_null_app_nb}}
Let $T>0$ be given and let us choose the coefficients
\begin{equation*}
\bar{\rho}=\bar{u}=\lambda_0=1=R\bar{\theta}=\frac{R^2\bar{\theta}}{c_0}=1,\ \kappa_0=2.
\end{equation*}
To prove this result (in the density case), it is enough to find a terminal data $(\sigma_T,v_T,\vphi_T)\in\mc{D}(A^*)$ such that the associated solution $(\sigma,v,\vphi)$ of \eqref{lcnse_adj_nb} fails to satisfy the following unique continuation property:
\begin{equation*}
\bar{u}\sigma(t,2\pi)+\bar{\rho}v(t,2\pi)=0\ \ \text{implies}\ \ (\sigma,v,\vphi)=(0,0,0).
\end{equation*}
Thanks to Remark \ref{rem_ev_nb}, $A^*$ has an eigenvalue $\nu_1=\nu_{-1}=-1$ for $n=1$ and $n=-1$ respectively. Let $\Phi_1:=\cv{\alpha_1}{\alpha_2}{\alpha_3}e^{ix}$ and $\Phi_{-1}:=\cv{\beta_1}{\beta_2}{\beta_3}e^{-ix}$ (for some $\alpha_i,\beta_i\in\mb{C}, i=1,2,3$) denote the independent eigenfunctions of $A^*$ corresponding to this multiple eigenvalue $-1$. We now choose the terminal data as
\begin{equation*}
(\sigma_T,v_T,\vphi_T)^{\dagger}=C\Phi_1+D\Phi_{-1},
\end{equation*}
where $C,D$ are complex constants that will be chosen later. The solution of \eqref{lcnse_adj_nb} is then given by
\begin{equation*}
(\sigma(t),v(t),\vphi(t))^{\dagger}=e^{-(T-t)}\left(C\Phi_1+D\Phi_{-1}\right).
\end{equation*}
Therefore
\begin{equation*}
\bar{u}\sigma(t,2\pi)+\bar{\rho}v(t,2\pi)=e^{-(T-t)}\left(C\mc{B}^*_{\rho}\Phi_1+D\mc{B}^*_{\rho}\Phi_{-1}\right).
\end{equation*}
If we take $C=-\mc{B}^*_{\rho}\Phi_{-1}$ and $D=\mc{B}^*_{\rho}\Phi_1$, then $C,D\neq0$ (thanks to Lemma \ref{lem_obs_est_nb}) and for these choice of $C,D$, we have $\bar{u}\sigma(t,2\pi)+\bar{\rho}v(t,2\pi)=0$ but $(\sigma,v,\vphi)\neq(0,0,0)$. This completes the proof.\qed
}
%%%%%%%%%%%%%%%%%%%%%%%%%%%%%%%%%%%%%%%%%%%%%%%%%%%%%%%%%%%%%%%%%%%%%%%%%
\section{Further comments and conclusions}
\subsection{Controllability results using Neumann boundary conditions}
We consider the system \eqref{lcnse_b} with the initial state \eqref{in_cd_b} and the boundary conditions
\begin{align}\label{bd_cd_b3}
\rho(t,0)=\rho(t,2\pi),\ \ u(t,0)=u(t,2\pi),\ \ u_x(t,0)=u_x(t,2\pi)+q_1(t),\ \ t\in (0,T),
\end{align}
where $q_1$ is a boundary control that acts on the velocity through Neumann conditions. Since the observation terms satisfies similar estimates as in \eqref{obs_est_v_b}, following the proof of Theorem \ref{thm_vel_b}, we can obtain the null controllability of the system \eqref{lcnse_b}-\eqref{in_cd_b}-\eqref{bd_cd_b3} at time $T>\frac{2\pi}{\bar{u}}$ in the space $\blue\dot{H}^1_{\per}(0,2\pi)\times \dot{L}^2(0,2\pi)$, and the null controllability fails in the space $\dot{H}^s_{\per}(0,2\pi)\times \dot{L}^2(0,2\pi)$ for $0\leq s<1$. In this case also, null controllability of the system \eqref{lcnse_b}-\eqref{in_cd_b}-\eqref{bd_cd_b3} is inconclusive when the time is small $(0<T\leq\frac{2\pi}{\bar{u}})$.

Similar to the barotropic case, we consider the system \eqref{lcnse_nb} with the initial state \eqref{in_cd_nb} and the boundary conditions
\begin{align}\label{bd_cd_nb4}
\rho(t,0)=\rho(t,2\pi),\ \ u(t,0)=u(t,2\pi),\ \ u_x(t,0)=u_x(t,2\pi)+q_2(t),\\ \theta(t,0)=\theta(t,2\pi),\ \ \theta_x(t,0)=\theta_x(t,2\pi),\ \ t\in (0,T).\notag
\end{align}
In this case also, following the proof of {\blue Theorem \ref{thm_null_nb}-Part (ii)} and {\blue Proposition \ref{prop_lac_null_nb}-Part (ii)}, we get null controllability of the system \eqref{lcnse_nb}-\eqref{in_cd_nb}-\eqref{bd_cd_nb4} at time $T>\frac{2\pi}{\bar{u}}$ in the space $\blue\dot{H}^1_{\per}(0,2\pi)\times (\dot{L}^2(0,2\pi))^2$, and null controllability fails in the space $\dot{H}^s_{\per}(0,2\pi)\times (\dot{L}^2(0,2\pi))^2$ for $0\leq s<1$. 
	
\smallskip
	
\noindent We next consider the system \eqref{lcnse_nb} with the initial state \eqref{in_cd_nb} and the boundary conditions
\begin{align}\label{bd_cd_nb5}
\rho(t,0)=\rho(t,2\pi),\ \ u(t,0)=u(t,2\pi),\ \ u_x(t,0)=u_x(t,2\pi),\\ \theta(t,0)=\theta(t,2\pi),\ \ \theta_x(t,0)=\theta_x(t,2\pi)+q_3(t),\ \ t\in (0,T).\notag
\end{align}
Similar to the previous case, following the proof of {\blue Theorem \ref{thm_null_nb}-Part (ii)} and {\blue Proposition \ref{prop_lac_null_nb}-Part (ii)}, we get null controllability of the system \eqref{lcnse_nb}-\eqref{in_cd_nb}-\eqref{bd_cd_nb5} at time $T>\frac{2\pi}{\bar{u}}$ in the space $\blue\dot{H}^1_{\per}(0,2\pi)\times (\dot{L}^2(0,2\pi))^2$, and null controllability fails in the space $\dot{H}^s_{\per}(0,2\pi)\times (\dot{L}^2(0,2\pi))^2$ for $0\leq s<1$.

\smallskip
	
\noindent For both systems \eqref{lcnse_nb}-\eqref{in_cd_nb}-\eqref{bd_cd_nb4} and \eqref{lcnse_nb}-\eqref{in_cd_nb}-\eqref{bd_cd_nb5}, null controllability is inconclusive for a small time $0<T\leq\frac{2\pi}{\bar{u}}$.

{\blue 
\subsection{Backward uniqueness property}\label{back_uniq}
The backward uniqueness property of the system \eqref{lcnse_b} or \eqref{lcnse_nb} is itself an interesting question from the mathematical point of view. It says that, when there is no control acting on the system and the solution vanishes at time $T>0$, then the solution must vanish identically at all time $t\in[0,T]$. Our system \eqref{lcnse_b} with the initial condition \eqref{in_cd_b} and boundary condition \eqref{bd_cd_b1} (with $p=0$) or \eqref{bd_cd_b2} (with $q=0$) satisfies the backward uniqueness property, more precisely, $(\rho(T),u(T))=(0,0)$ implies $(\rho,u)=0$ for all $t\in[0,T]$. This can be seen easily as the eigenfunctions of $A^*$, and hence of $A$, form a complete set in $(L^2(0,2\pi))^2$. We can similarly conclude that the non-barotropic system \eqref{lcnse_nb} with the initial condition \eqref{in_cd_nb} and boundary condition \eqref{bd_cd_nb1} (with $p=0$) or \eqref{bd_cd_nb2} (with $q=0$) or \eqref{bd_cd_nb3} (with $r=0$) satisfies the backward uniqueness property.

\smallskip

If a system has backward uniqueness property, then null controllability of the system at some time $T>0$ will give approximate controllability at that time $T$. This can be seen easily, because the observability inequality (for null controllability) and the backward uniqueness implies the unique continuation property for the corresponding adjoint system. Thus, using a boundary control in density, our systems \eqref{lcnse_b} and \eqref{lcnse_nb} are approximately controllable at time $T>\frac{2\pi}{\bar{u}}$ in the spaces $(\dot{L}^2(0,2\pi))^2$ and $(\dot{L}^2(0,2\pi))^3$ respectively (thanks to Theorem \ref{thm_den_b} and Theorem \ref{thm_null_nb}). Similarly, when a boundary control is acting in the velocity or in temperature (for the non-barotropic case), the systems \eqref{lcnse_b} and \eqref{lcnse_nb} are approximately controllable at time $T>\frac{2\pi}{\bar{u}}$ in the spaces $\dot{H}^1_{\per}(0,2\pi)\times\dot{L}^2(0,2\pi)$ and $\dot{H}^1_{\per}(0,2\pi)\times(\dot{L}^2(0,2\pi))^2$ respectively (thanks to Theorem \ref{thm_vel_b} and Theorem \ref{thm_null_nb}).

%``If $(\rho,u)$ is a solution of \eqref{lcnse_b} satisfying $(\rho(T),u(T))=(0,0)$ for a.e.  $x\in(0,2\pi)$, then $(\rho,u)=(0,0)$."

%Since the eigenfunctions of $A^*$, defined by \eqref{lemma_eigen_b}, form a Riesz basis and hence a complete set in $(\dot{L}^2(0,2\pi))^2$, the above property follows directly. The same is true in the non-barotropic case also; as a consequence, we have backward uniqueness property for both barotropic and non-barotropic systems.

In this context, we must mention that proving the backward uniqueness property might be difficult (in general) when the associated operator do not have complete set of eigenfunctions; see for instance \cite{Renardy15}, where the author proved backward uniqueness of the linearized compressible Navier-Stokes system \eqref{lcnse_b} under Dirichlet boundary conditions $\rho(t,0)=u(t,0)=u(t,1)=0$ ($t\in(0,T)$), by proving injectivity of the semigroup.
}

\subsection{More number of controls}
Adding controls in both velocity and temperature components %through periodic boundary conditions 
does not improve the null controllability result of the system \eqref{lcnse_nb} with respect to the regularity of the initial states. Estimates of the observation terms remain the same as in the control acts in velocity or temperature.

\subsection{Controllability under Dirichlet boundary conditions}
Let us consider the system \eqref{lcnse_b} in the interval $(0,1)$ with the initial state \eqref{in_cd_b} and the following boundary conditions
\begin{equation}\label{bd_cd_b4}
\rho(t,0)=p(t),\ \ u(t,0)=0,\ \ u(t,1)=q(t),\ \ t\in(0,T),
\end{equation}
where $p$ and $q$ are boundary controls. It is known in \cite{Bhandari22} that the system \eqref{lcnse_b}-\eqref{in_cd_b}-\eqref{bd_cd_b4} (with $q=0$) is null controllability at a large time $T$ using only one boundary control $p\in L^2(0,T)$ provided the initial states belong to the space $\blue L^2(0,1)\times L^2(0,1)$. Null controllability of the system \eqref{lcnse_b} using a boundary control acts only in velocity through Dirichlet conditions (that is, $p=0$ in \eqref{bd_cd_b4}) is still an open problem.

\smallskip

\noindent In the case of non-barotropic fluids, null controllability of the system \eqref{lcnse_nb} at large time $T$ using only one boundary control acts either in density, velocity or temperature through Dirichlet conditions is also an open problem.
%%%%%%%%%%%%%%%%%%%%%%%%%%%%%%%%%%%%%%%%%%%%%%%%%%%%%%%%%%%%%%%%%%%%%%%%%
\appendix
\section{Proof of the well-posedness results}
\subsection{Existence of  semigroup: proof of Lemma \ref{lem_ex_sg_nb}} \label{app_well_p_nb}
The proof is divided  into several parts.

\smallskip 

\textbf{Part 1.} {\em The operator $A$ is dissipative.} Indeed, for all $(\xi,\eta,\zeta)^{\dagger}\in\mc{D}(A)$
\begin{align*}
&\Re\ip{A\mathbf{U}}{\mathbf{U}}_{(L^2(0,2\pi))^3}=\Re\ip{\cv{-\bar{u}\xi_x-\bar{\rho}\eta_x}{-\frac{R\bar{\theta}}{\bar{\rho}}\xi_x+\lambda_0\eta_{xx}-\bar{u}\eta_x-R\zeta_x}{-\frac{R\bar{\theta}}{c_0}\eta_x+\kappa_0\zeta_{xx}-\bar{u}\zeta_x}}{\cv{\xi}{\eta}{\zeta}}_{(L^2(0,2\pi))^3}\\
&=\Re\left(-R\bar{\theta}\bar{u}\int_{0}^{2\pi}\bar{\xi}\xi_xdx-R\bar{\theta}\bar{\rho}\int_{0}^{2\pi}\bar{\xi}\eta_xdx-R\bar{\theta}\bar{\rho}\int_{0}^{2\pi}\xi_x\bar{\eta}dx+\lambda_0\bar{\rho}^2\int_{0}^{2\pi}\bar{\eta}\eta_{xx}dx-\bar{\rho}^2\bar{u}\int_{0}^{2\pi}\bar{\eta}\eta_xdx\right.\\
&\quad\quad\quad\quad\quad\left.-R\bar{\rho}^2\int_{0}^{2\pi}\bar{\eta}\zeta_xdx-R\bar{\rho}^2\int_{0}^{2\pi}\eta_x\bar{\zeta}dx+\kappa_0\frac{\bar{\rho}^2c_0}{\bar{\theta}}\int_{0}^{2\pi}\bar{\zeta}\zeta_{xx}dx-\bar{u}\frac{\bar{\rho}^2c_0}{\bar{\theta}}\int_{0}^{2\pi}\bar{\zeta}\zeta_xdx\right)\\
&=-\frac{R\bar{\theta}\bar{u}}{2}\int_{0}^{2\pi}\frac{d}{dx}(\mod{\xi}^2)dx-\lambda_0\bar{\rho}^2\int_{0}^{2\pi}\bar{\eta}_x\eta_xdx-\frac{\bar{\rho}^2\bar{u}}{2}\int_{0}^{2\pi}\frac{d}{dx}(\mod{\eta}^2)dx-\kappa_0\frac{\bar{\rho}^2c_0}{\bar{\theta}}\int_{0}^{2\pi}\bar{\zeta}_x\zeta_xdx\\
&\quad\quad\quad-\frac{\bar{u}}{2}\frac{\bar{\rho}^2c_0}{\bar{\theta}}\int_{0}^{2\pi}\frac{d}{dx}(\mod{\zeta}^2)dx-\lambda_0\bar{\rho}^2\int_{0}^{2\pi}\mod{u_x}^2dx-\kappa_0\frac{\bar{\rho}^2c_0}{\bar{\theta}}\int_{0}^{2\pi}\mod{\zeta_x}^2dx\leq0.
\end{align*}
\textbf{Part 2.} {\em The operator $A$ is maximal.} This is equivalent to the following. For any $\nu>0$ and any $\cv{f}{g}{h}\in(L^2(0,2\pi))^3$, we can find a $\cv{\xi}{\eta}{\zeta}\in\mc{D}(A)$ such that
\begin{equation*}
(\nu I-A)\cv{\xi}{\eta}{\zeta}=\cv{f}{g}{h},
\end{equation*}
that is,
\begin{align*}
\nu\xi+\bar{u}\xi_x+\bar{\rho}\eta_x=f,\\
\nu\eta+\frac{R\bar{\theta}}{\bar{\rho}}\xi_x-\lambda_0\eta_{xx}+\bar{u}\eta_x+R\zeta_x=g,\\
\nu\zeta+\frac{R\bar{\theta}}{c_0}\eta_x-\kappa_0\zeta_{xx}+\bar{u}\zeta_x=h.
\end{align*}
Let $\epsilon>0$. Instead of solving the above problem, we will solve the following regularized problem
\begin{equation}\label{reg_prb}
\begin{cases}
\nu\xi+\bar{u}\xi_x-\epsilon\xi_{xx}+\bar{\rho}\eta_x=f,\\[4pt]
\nu\eta+\frac{R\bar{\theta}}{\bar{\rho}}\xi_x-\lambda_0\eta_{xx}+\bar{u}\eta_x+R\zeta_x=g,\\[4pt]
\nu\zeta+\frac{R\bar{\theta}}{c_0}\eta_x-\kappa_0\zeta_{xx}+\bar{u}\zeta_x=h.
\end{cases}
\end{equation}
with the following boundary conditions
\begin{equation*}
\xi(0)=\xi(2\pi),\ \ \xi_x(0)=\xi_x(2\pi), \ \ \eta(0)=\eta(2\pi),\ \ \eta_x(0)=\eta_x(2\pi), \ \ \zeta(0)=\zeta(2\pi),\ \ \zeta_x(0)=\zeta_x(2\pi).
\end{equation*}
We now proceed through the following steps.

\smallskip

\noindent\textit{Step 1.} Using Lax-Milgram theorem, we first prove that the system \eqref{reg_prb} has a unique solution in $(H^1_{\per}(0,2\pi))^3$. Define the operator $B:(H^1_{\per}(0,2\pi))^3\times(H^1_{\per}(0,2\pi))^3\to\mathbb{C}$ by
\begin{align*}
&B\left(\cv{\xi}{\eta}{\zeta},\cv{\xi_1}{\eta_1}{\zeta_1}\right)\\
&=\epsilon\int_{0}^{2\pi}\xi_x(\bar{\xi_1})_xdx+\bar{\rho}\int_{0}^{2\pi}\eta_x\bar{\xi_1}dx+\bar{u}\int_{0}^{2\pi}\xi_x\bar{\xi_1}dx+\nu\int_{0}^{2\pi}\xi\bar{\xi_1}dx\\
&\quad+\lambda_0\int_{0}^{2\pi}\eta_x(\bar{\eta_1})_xdx+\bar{u}\int_{0}^{2\pi}\eta_x\bar{\eta_1}dx+\frac{R\bar{\theta}}{\bar{\rho}}\int_{0}^{1}\xi_x\bar{\eta_1}dx+R\int_{0}^{2\pi}\zeta_x\bar{\eta_1}dx+\nu\int_{0}^{2\pi}\eta\bar{\eta_1}dx\\
&\quad+\kappa_0\int_{0}^{2\pi}\zeta_x(\bar{\zeta_1})_xdx+\bar{u}\int_{0}^{2\pi}\zeta_x\bar{\zeta_1}dx+\frac{R\bar{\theta}}{c_0}\int_{0}^{2\pi}\eta_x\bar{\zeta_1}dx+\nu\int_{0}^{2\pi}\zeta\bar{\zeta_1}dx,
\end{align*}
for all $\cv{\xi}{\eta}{\zeta},\cv{\xi_1}{\eta_1}{\zeta_1}\in(H^1_{\per}(0,2\pi))^3$. Then, one can show that $B$ is continuous and coercive. Thus, by Lax-Milgram theorem, for every $\epsilon>0$, there exists a unique solution $(\xi^{\epsilon},\eta^{\epsilon},\zeta^{\epsilon})^{\dagger}\in(H^1_{\per}(0,2\pi))^3$ such that
\begin{equation*}
B\left(\cv{\xi^{\epsilon}}{\eta^{\epsilon}}{\zeta^{\epsilon}},\cv{\xi}{\eta}{\zeta}\right)=F\left(\cv{\xi}{\eta}{\zeta}\right), \quad \forall \cv{\xi}{\eta}{\zeta}\in(H^1_{\per}(0,2\pi))^3,
\end{equation*}
where $F:(H^1_{\per}(0,2\pi))^3\to\mathbb{C}$ is the linear functional given by
\begin{equation*}
F\left(\cv{\xi}{\eta}{\zeta}\right):=\int_{0}^{2\pi}f\bar{\xi}dx+\int_{0}^{2\pi}g\bar{\eta}dx+\int_{0}^{2\pi}h\bar{\zeta}dx.
\end{equation*}

\smallskip

\noindent\textit{Step 2.}
Observe that 
\begin{align*}
\Re\left(B\left(\cv{\xi^{\epsilon}}{\eta^{\epsilon}}{\zeta^{\epsilon}},\cv{\xi^{\epsilon}}{\eta^{\epsilon}}{\zeta^{\epsilon}}\right)\right)\leq\int_{0}^{2\pi}\mod{f\overline{\xi^{\epsilon}}}dx+\int_{0}^{2\pi}\mod{g\overline{\eta^{\epsilon}}}dx+\int_{0}^{2\pi}\mod{h\overline{\zeta^{\epsilon}}}dx\\
\leq\frac{1}{2}\int_{0}^{2\pi}\left(\mod{f}^2+\mod{g}^2+\mod{h}^2\right)dx+\frac{1}{2}\int_{0}^{2\pi}\left(\mod{\overline{\xi^{\epsilon}}}^2+\mod{\overline{\eta^{\epsilon}}}^2+\mod{\overline{\zeta^{\epsilon}}}^2\right)dx,
\end{align*}
which yields 
\begin{equation*}
\epsilon\int_{0}^{2\pi}\mod{\xi^{\epsilon}_x}^2+\frac{\nu}{2}\int_{0}^{2\pi}\mod{\xi^{\epsilon}}^2+\lambda_0\int_{0}^{2\pi}\mod{\eta^{\epsilon}_x}^2+\frac{\nu}{2}\int_{0}^{2\pi}\mod{\eta^{\epsilon}}^2+\kappa_0\int_{0}^{2\pi}\mod{\zeta^{\epsilon}_x}^2+\frac{\nu}{2}\int_{0}^{2\pi}\mod{\zeta^{\epsilon}}^2\leq\frac{1}{2}\int_{0}^{2\pi}(\mod{f}^2+\mod{g}^2+\mod{h}^2)
\end{equation*}
This shows that the sequences $(\eta^{\epsilon})$ and $(\zeta^{\epsilon})$ are bounded in $H^1(0,2\pi)$ and the sequences  $(\xi^{\epsilon})$ and $(\sqrt{\epsilon}\xi^{\epsilon}_x)$ are bounded in $L^2(0,2\pi)$. Since the spaces $H^1(0,2\pi)$ and $L^2(0,2\pi)$ are reflexive, there exist subsequences, still denoted by $(\eta^{\epsilon}), (\zeta^{\epsilon})$, $(\xi^{\epsilon})$, and functions $\xi\in L^2(0,2\pi)$ and $\eta\in H^1(0,2\pi)$ such that
\begin{equation*}
\eta^{\epsilon}\rightharpoonup\eta\text{ in } H^1(0,2\pi),\text{ and } \xi^{\epsilon}\rightharpoonup\xi\text{ in }L^2(0,2\pi).
\end{equation*}
Furthermore, we have
\begin{equation*}
\blue\int_{0}^{2\pi}\mod{\epsilon\xi^{\epsilon}_x}^2=\epsilon\int_{0}^{1}\mod{\sqrt{\epsilon}\xi^{\epsilon}_x}^2\to0,\text{ as }\epsilon\to0.
\end{equation*}
Now, since $B\left(\cv{\xi^{\epsilon}}{\eta^{\epsilon}}{\zeta^{\epsilon}},\cv{\xi}{\eta}{\zeta}\right)=F\left(\cv{\xi}{\eta}{\zeta}\right)$, for all $\cv{\xi}{\eta}{\zeta}\in(H^1_{\per}(0,2\pi))^3$, we may take $\cv{\xi_1}{0}{0}\in(H^1_{\per}(0,2\pi))^3$, so that we obtain
\begin{equation}\label{identity1}
\epsilon\int_{0}^{2\pi}\xi^{\epsilon}_x(\bar{\xi_1})_xdx+\bar{\rho}\int_{0}^{2\pi}\eta^{\epsilon}_x\bar{\xi_1}dx+\bar{u}\int_{0}^{2\pi}\xi^{\epsilon}_x\bar{\xi_1}dx+\nu\int_{0}^{2\pi}\xi^{\epsilon}\bar{\xi_1}dx=\int_{0}^{2\pi}f\bar{\xi_1}dx.
\end{equation}
Similarly, by taking  $\cv{0}{\eta_1}{0},\cv{0}{0}{\zeta_1}\in(H^1_{\per}(0,2\pi))^3$, we get
\begin{equation}\label{identity2}
\lambda_0\int_{0}^{2\pi}\eta^{\epsilon}_x(\bar{\eta_1})_xdx+\bar{u}\int_{0}^{2\pi}\eta^{\epsilon}_x\bar{\eta_1}dx+\frac{R\bar{\theta}}{\bar{\rho}}\int_{0}^{1}\xi^{\epsilon}_x\bar{\eta_1}dx+R\int_{0}^{2\pi}\zeta^{\epsilon}_x\bar{\eta_1}dx+\nu\int_{0}^{2\pi}\eta^{\epsilon}\bar{\eta_1}dx=\int_{0}^{2\pi}g\bar{\eta_1}dx,
\end{equation}
and
\begin{equation}\label{identity3}
\kappa_0\int_{0}^{2\pi}\zeta^{\epsilon}_x(\bar{\zeta_1})_xdx+\bar{u}\int_{0}^{2\pi}\zeta^{\epsilon}_x\bar{\zeta_1}dx+\frac{R\bar{\theta}}{c_0}\int_{0}^{2\pi}\eta^{\epsilon}_x\bar{\zeta_1}dx+\nu\int_{0}^{2\pi}\zeta^{\epsilon}\bar{\zeta_1}dx=\int_{0}^{2\pi}h\bar{\zeta_1}dx
\end{equation}
Integrating by parts, we get from equation \eqref{identity1} that,
\begin{equation*}
\epsilon\int_{0}^{2\pi}\xi^{\epsilon}_x(\bar{\xi_1})_xdx+\bar{\rho}\int_{0}^{2\pi}\eta^{\epsilon}_x\bar{\xi_1}dx-\bar{u}\int_{0}^{2\pi}\xi^{\epsilon}(\bar{\xi_1})_xdx+\nu\int_{0}^{2\pi}\xi^{\epsilon}\bar{\xi_1}dx=\int_{0}^{2\pi}f\bar{\xi_1}dx.
\end{equation*}
Then, passing to the limit $\epsilon\to0$, we obtain
\begin{equation*}
\bar{\rho}\int_{0}^{2\pi}\eta_x\bar{\xi_1}dx+\bar{u}\int_{0}^{2\pi}\xi_x\bar{\xi_1}dx+\nu\int_{0}^{2\pi}\xi\bar{\xi_1}dx=\int_{0}^{2\pi}f\bar{\xi_1}dx,
\end{equation*}
and the above relation is true  $\forall \xi_1\in\mathcal C_c^{\infty}(0,2\pi)$. As a consequence, 
\begin{equation*}
\bar{\rho}\eta_x+\bar{u}\xi_x+\nu\xi=f,
\end{equation*}
in the sense of distribution and therefore $\bar{u}\xi_x=f-\bar{\rho}\eta_x-\nu\xi\in L^2(0,2\pi)$; in other words, $\xi\in H^1(0,2\pi)$. We similarly have from identities \eqref{identity2} and \eqref{identity3}
\begin{align*}
\nu\eta+\frac{R\bar{\theta}}{\bar{\rho}}\xi_x-\lambda_0\eta_{xx}+\bar{u}\eta_x+R\zeta_x=g,\\[4pt]
\nu\zeta+\frac{R\bar{\theta}}{c_0}\eta_x-\kappa_0\zeta_{xx}+\bar{u}\zeta_x=h,
\end{align*}
in the sense of distribution and therefore $\eta,\zeta\in H^2(0,2\pi)$.

\smallskip 

\noindent\textit{Step 3.} We now show $\eta(0)=\eta(2\pi)$ and $\eta_x(0)=\eta_x(2\pi)$. Since the inclusion map $i:H^1(0,2\pi)\to\mathcal C^0(\bar{0,2\pi})$ is compact and $\eta^{\epsilon}\rightharpoonup\eta$ in $H^1(0,2\pi)$, we obtain
\begin{equation*}
\eta^{\epsilon}\to\eta \ \ \text{ in }  \mathcal C^0[0,2\pi].
\end{equation*}
Thus,  $(\eta^{\epsilon}(0),\eta^{\epsilon}(2\pi))\to (\eta(0),\eta(2\pi))$. Since $\eta^{\epsilon}(0)=\eta^{\epsilon}(2\pi)$ for all $\epsilon>0$, we have $$\eta(0)=\eta(2\pi).$$ From \eqref{identity2}, we have after passing the limit as $\epsilon\to0$
\begin{equation*}
\lambda_0\int_{0}^{2\pi}\eta_x(\bar{\eta_1})_xdx+\bar{u}\int_{0}^{2\pi}\eta_x\bar{\eta_1}dx+\frac{R\bar{\theta}}{\bar{\rho}}\int_{0}^{1}\xi_x\bar{\eta_1}dx+R\int_{0}^{2\pi}\zeta_x\bar{\eta_1}dx+\nu\int_{0}^{2\pi}\eta\bar{\eta_1}dx=\int_{0}^{2\pi}g\bar{\eta_1}dx.
\end{equation*}
Integrating by parts, we get
\begin{align*}
-\lambda_0\int_{0}^{2\pi}\eta_{xx}\bar{\eta_1}dx+\lambda_0(\eta_x(2\pi)\bar{\eta_1}(2\pi)-\eta_x(0)\bar{\eta_1}(0))+\bar{u}\int_{0}^{2\pi}\eta_x\bar{\eta_1}dx+\frac{R\bar{\theta}}{\bar{\rho}}\int_{0}^{1}\xi_x\bar{\eta_1}dx\\
+R\int_{0}^{2\pi}\zeta_x\bar{\eta_1}dx+\nu\int_{0}^{2\pi}\eta\bar{\eta_1}dx=\int_{0}^{2\pi}g\bar{\eta_1}dx,
\end{align*}
and therefore
\begin{equation*}
\eta_x(2\pi)\bar{\eta}_1(2\pi)-\eta_x(0)\bar{\eta}_1(0)=0
\end{equation*}
that is $\eta_x(0)=\eta_x(2\pi)$. In a similar way, we can obtain $\zeta(0)=\zeta(2\pi)$ and $\zeta_x(0)=\zeta_x(2\pi)$.

\smallskip

\noindent We now show $\xi(0)=\xi(2\pi)$. Recall that we have after taking limit as $\epsilon\to0$
\begin{equation*}
\bar{\rho}\int_{0}^{2\pi}\eta_x\bar{\xi_1}dx-\bar{u}\int_{0}^{2\pi}\xi(\bar{\xi_1})_xdx+\nu\int_{0}^{2\pi}\xi\bar{\xi_1}dx=\int_{0}^{2\pi}f\bar{\xi_1}dx.
\end{equation*}
Integrating by parts, we get
\begin{equation}%\label{compare-1}
\bar{\rho}\int_{0}^{2\pi}\eta_x\bar{\xi_1}dx+\bar{u}\int_{0}^{2\pi}\xi_x\bar{\xi_1}dx-\bar{u}(\xi(2\pi)\xi_1(2\pi)-\xi(0)\xi_1(0))+\nu\int_{0}^{2\pi}\xi\bar{\xi_1}dx=\int_{0}^{2\pi}f\bar{\xi_1}dx,
\end{equation}
and therefore $$\xi(0)=\xi(2\pi).$$ So, we get $\cv{\xi}{\eta}{\zeta}\in \mc{D}(A)$.  Hence, the operator $A$ is maximal.

{\blue
\section{Properties of the eigenvalues}\label{app_prop_ev_b}
\begin{proof}[Proof of Lemma \ref{prop_ev_b}]
Recall the expression of the eigenvalues
\begin{align*}	
\blue\nu_n^h&\blue=-\frac{1}{2}\left(\mu_0n^2-\sqrt{\mu_0^2n^4-4b\bar{\rho}n^2}-2\bar{u}in\right),\\
\blue\nu_n^{p}&\blue=-\frac{1}{2}\left(\mu_0n^2+\sqrt{\mu_0^2n^4-4b\bar{\rho}n^2}-2\bar{u}in\right),
\end{align*}
for $n\in\mb{Z}^*$. We prove each part separately.
\begin{enumerate}
\item[Part-(i).] Let us denote $n_0:=\frac{2\sqrt{b\bar{\rho}}}{\mu_0}$. Then, $\Im(\nu_n^h)=\bar{u}n$ for all $\mod{n}\geq n_0$ and therefore $\nu_n^h\neq\nu_l^h$ for all $\mod{n},\mod{l}\geq n_0$ with $n\neq l$. For $1\leq\mod{n}<n_0$, we have $\nu_n^h=-\frac{\mu_0}{2}n^2+i(\bar{u}n+\frac{1}{2}\sqrt{4b\bar{\rho}n^2-\mu_0^2n^4})$. Since $\Im(\nu_n^h)\neq\Im(\nu_{-n}^h)$ for all $1\leq\mod{n}<n_0$, we readily have $\nu_n^h\neq\nu_l^h$ for all $1\leq\mod{n},\mod{l}<n_0$.

\smallskip
		
\item[Part-(ii).] Note that $\Im(\nu_n^p)=\bar{u}n$ for all $\mod{n}\geq n_0$ and therefore $\nu_n^p\neq\nu_l^p$ for all $\mod{n},\mod{l}\geq n_0$ with $n\neq l$. For $1\leq\mod{n}<n_0$, we have $\nu_n^p=-\frac{\mu_0}{2}n^2+i(\bar{u}n-\frac{1}{2}\sqrt{4b\bar{\rho}n^2-\mu_0^2n^4})$. Then, $\Im(\nu_n^p)=-\Im(\nu_{-n}^p)$ for all $1\leq\mod{n}<n_0$, which implies $\nu_n^p=\nu_{-n}^p$ holds only if $\bar{u}n-\frac{1}{2}\sqrt{4b\bar{\rho}n^2-\mu_0^2n^4}=0$, that is, when $n=\frac{2\sqrt{b\bar{\rho}-\bar{u}^2}}{\mu_0}$. Moreover, $\Re(\nu_n^p)\neq\Re(\nu_l^p)$ for remaining values of $n,l\in\mb{Z}^* (n\neq l)$, implying $\nu_n^p\neq\nu_l^p$.
		
\smallskip		
		
\item[Part-(iii).] Let $n,l\in\mb{Z}^*$ with $\mod{n},\mod{l}>n_0$. Since $\Im(\nu_n^h)=\Im(\nu_n^p)=\bar{u}n$, therefore $\Im(\nu_n^h)=\Im(\nu_l^p)$ is true if and only if $n=l$ and $\Re(\nu_n^h)\neq\Re(\nu_n^p)$. This proves that $\nu_n^h\neq\nu_l^p$ for all $\mod{n},\mod{l}>n_0$. For $1\leq\mod{n},\mod{l}<n_0$, $\Re(\nu_n^h)=\Re(\nu_n^p)=-\frac{\mu_0}{2}n^2$ and therefore $\Re(\nu_n^h)=\Re(\nu_l^p)$ holds if and only if $n=\pm l$. On the other hand, $\Im(\nu_n^h)\neq\Im(\nu_l^p)$ for $n=\pm l$, which implies $\nu_n^h\neq\nu_l^p$ for all $1\leq\mod{n},\mod{l}<n_0$. Let $1\leq\mod{n}\leq n_0$ and $\mod{l}>n_0$. Then, $\nu_n^h=-\frac{\mu_0}{2}n^2+i(\bar{u}n+\frac{1}{2}\sqrt{4b\bar{\rho}n^2-\mu_0^2n^4})$ and $\nu_l^p=-\frac{\mu_0}{2}l^2-\frac{1}{2}\sqrt{\mu_0^2l^4-4b\bar{\rho}l^2}+\bar{u}il$. Thus $\nu_n^h=\nu_l^p$ implies $\frac{1}{2}\sqrt{\mu_0^2l^4-4b\bar{\rho}l^2}=-\frac{\mu_0}{2}(l^2-n^2)<0$, which is not possible. Therefore, the only possible case is $\mod{n}=\mod{l}=n_0$ and in this case, we have $\nu_{n}^h=\nu_l^p$, which proves part (iii)

\smallskip
		
\item[Part-(iv).] Follows from parts (i), (ii) and (iii).
\end{enumerate}	
\end{proof}
}
%%%%%%%%%%%%%%%%%%%%%%%%%%%%%%%%%%%%%%%%%%%%%%%%%%%%%%%%%%%%%%%%%%%%%%%%%
\bibliographystyle{siam}
\bibliography{references.bib}

\end{document}